\pdfoutput=1
\documentclass[letterpaper]{article}

\usepackage[T1]{fontenc}

\usepackage[margin=1.5in]{geometry}

\usepackage{amsmath,amssymb,amsthm}

\usepackage[framemethod=tikz]{mdframed}

\usepackage{quiver}

\usepackage[backend=biber,style=numeric]{biblatex}
\usepackage[unicode=true,hidelinks]{hyperref}
\usepackage[nameinlink]{cleveref}

\usepackage{microtype}

\usepackage{enumitem}

\DeclareFontFamily{U}{min}{}
\DeclareFontShape{U}{min}{m}{n}{<-> dmjhira}{}
\newcommand{\yo}{\text{\usefont{U}{min}{m}{n}\symbol{'110}}}

\newtheorem{thm}{Theorem}[section]
\newtheorem{letterthm}{Theorem}

\newtheorem{lem}{Lemma}[section]
\newtheorem*{lem*}{Lemma}
\newtheorem{prop}{Proposition}[section]
\newtheorem*{prop*}{Proposition}
\newtheorem{cor}{Corollary}[section]

\theoremstyle{definition}
\newtheorem{defn}{Definition}[section]
\newtheorem{exam}{Example}[section]
\newtheorem*{notn}{Notation}

\newtheoremstyle{slogan}
  {\topsep}   
  {\topsep}   
  {\normalfont}  
  {5pt}       
  {\bfseries \color{blue!70!black}} 
  {.}         
  {5pt plus 1pt minus 1pt} 
  {}          
\theoremstyle{slogan}
\newtheorem{slogan}{Slogan}

\newtheoremstyle{warning}
  {\topsep}   
  {\topsep}   
  {\normalfont}  
  {5pt}       
  {\bfseries \color{orange!70!black}} 
  {.}         
  {5pt plus 1pt minus 1pt} 
  {}          
\theoremstyle{warning}
\newtheorem*{warning}{Warning}

\mdfdefinestyle{leftbar}{
  hidealllines=true,
  leftline=true,
  innerleftmargin=10pt,
  innertopmargin=7pt,
}
\surroundwithmdframed[style=leftbar]{proof}

\newcommand{\C}{\mathcal{C}}
\newcommand{\D}{\mathcal{D}}
\newcommand{\OO}{\mathcal{O}}
\newcommand{\set}{\mathsf{set}}
\newcommand{\Set}{\mathsf{Set}}
\newcommand{\Ab}{\mathsf{Ab}}
\newcommand{\CMon}{\mathsf{CMon}}

\newcommand{\etale}{\mathsf{\acute{e}tal\acute{e}}}
\newcommand{\Sh}[1]{\mathsf{Sh}(#1)}

\newcommand{\fet}{\mathsf{f\acute{e}t}}
\newcommand{\Sm}{\mathsf{Sm}}

\newcommand{\A}[1]{\mathcal{A}_{#1}}
\newcommand{\U}[1]{\mathcal{P}_{#1}}
\newcommand{\Mack}[1]{\mathsf{Mack}(#1)}
\newcommand{\Tamb}[1]{\mathsf{Tamb}(#1)}
\newcommand{\SMack}[1]{\mathsf{SMack}(#1)}
\newcommand{\STamb}[1]{\mathsf{STamb}(#1)}
\newcommand{\Fun}[2]{\mathsf{Fun}(#1,#2)}
\newcommand{\op}{\text{op}}
\DeclareMathOperator{\Lan}{Lan}
\DeclareMathOperator{\Hom}{Hom}

\DeclareMathOperator{\id}{id}
\DeclareMathOperator{\Spec}{Spec}

\makeatletter
\newcommand{\xRightarrow}[2][]{\ext@arrow 0359\Rightarrowfill@{#1}{#2}}
\makeatother

\title{Norms of Generalized Mackey and Tambara Functors}
\author{Ben Spitz}
\date{}

\begin{document}

\maketitle

\begin{abstract}
  Let $G$ be a finite group. A $G$-Tambara functor can be defined as a product-preserving functor $\U{G} \to \Set$ (satisfying one additional condition), where $\U{G}$ is a category that is constructed in a straightforward way from the category of finite $G$-sets.

  By replacing the category of finite $G$-sets with other categories, we obtain a more general notion of ``Tambara functor''. This more general notion subsumes the notion of motivic Tambara functors introduced by Bachmann. In this article, we extend a result of Hoyer about $G$-Tambara functors to this more general context.
\end{abstract}

\tableofcontents

\section{Introduction}
\label{section:intro}

Let $G$ be a finite group. In $G$-equivariant homotopy theory, the primary algebraic invariants of interest are $G$-Mackey functors and $G$-Tambara functors. These are generalizations of abelian groups and commutative rings (respectively) ``in the equivariant direction'', meaning that for $G=1$ the respective notions coincide. Whenever one would see an abelian group in ordinary homotopy theory, one can expect to find a Mackey functor in equivariant homotopy theory; likewise for commutative rings and Tambara functors.

\begin{figure}[h!]
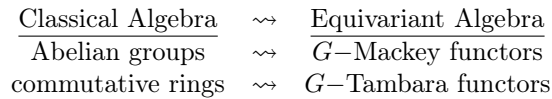

  \[\begin{array}{ccc}
      {\text{\underline{Classical Algebra}}} & \rightsquigarrow & {\text{\underline{Equivariant Algebra}}} \\
      {\text{Abelian groups}}                & \rightsquigarrow & {G{-}\text{Mackey functors}}             \\
      {\text{commutative rings}}             & \rightsquigarrow & {G{-}\text{Tambara functors}}
    \end{array}\]
  \caption{A dictionary between classical and equivariant algebra.}
\end{figure}

An abelian group is a commutative monoid with inverses, and a commutative monoid is simply an algebraic structure with exactly one operation of each arity. Said another way, operations $M^n \to M$ on a commutative monoid $M$ are in bijective correspondence with functions $n \to 1$; there is exactly one of each.

\bigskip

More generally, if $M$ is a commutative monoid and $n, m$ are natural numbers, the operations $M^n \to M^m$ are indexed by isomorphism classes of \emph{spans} from $n$ to $m$ (diagrams of the form $n \leftarrow \bullet \rightarrow m$) in finite sets. Explicitly, a span
\[n \xleftarrow{f} k \xrightarrow{g} m\]
encodes the operation
\[(x_j)_{j=1}^n \mapsto \left(\sum_{z \in g^{-1}(i)} x_{f(z)}\right)_{i=1}^m : M^n \to M^m.\]

There is a category $\A{\set}$ (to be defined more carefully later) whose objects are finite sets and whose morphisms are spans; it will turn out that the category of product-preserving functors $\A{\set} \to \Set$ is equivalent to the category $\CMon$ of commutative monoids. In other words, $\A{\set}$ is the Lawvere theory of commutative monoids, and $\CMon$ is the category of Lawvere algebras of $\A{\set}$.

\bigskip

Among commutative monoids, there are some which admit inverses, and these objects are called abelian groups. The corresponding subcategory of the category of product-preserving functors $\A{\set} \to \Set$ therefore gives a ``purely syntactic'' construction of $\Ab$.

\bigskip

In a very similar way, the the operations of a $G$-Mackey functor (ignoring additive inverses) are indexed by spans of finite $G$-sets. To define $G$-Mackey functors, we begin by generalizing the construction of the Lawvere theory $\A{\set}$ of commutative monoids, replacing finite sets with finite $G$-sets to obtain a category $\A{G{-}\set}$. Then we take the Lawvere algebras of $\A{G{-}\set}$, and consider only those which admit additive inverses (in a precise sense to be defined later). These objects are called $G$-Mackey functors.

\bigskip

There is an analogous story with Tambara functors: we take the construction of the Lawvere theory $\U{\set}$ of commutative semirings, and replace $\set$ with $G{-}\set$ to obtain a category $\U{G{-}\set}$. Then we take the Lawvere algebras of $\U{G{-}\set}$, and consider only those which admit additive inverses. These objects are called $G$-Tambara functors.

\bigskip

We need not stop here. Instead of replacing $\set$ by $G{-}\set$, we could use any category $\C$, so long as we can still interpret the constructions of $\A{\C}$ and $\U{\C}$, and so long as we can still interpret what it means for a Lawvere algebra of $\A{\C}$ or $\U{\C}$ to admit additive inverses. The structure needed by $\C$ to interpret all of this is that $\C$ is \emph{locally cartesian closed with finite disjoint coproducts} (abbreviated LCCDC).

\bigskip

This procedure gives us a definition of ``$\C$-Mackey functors'' and ``$\C$-Tambara functors'' for any LCCDC category $\C$. Among possible choices for $\C$, of particular interest are the category of finite $G$-sets (for $G$ some finite group) and the category of finite {\'e}tale $S$-schemes (for some scheme $S$). The former recovers precisely the classical notions of $G$-Mackey functors and $G$-Tambara functors; on the other hand, the latter yields precisely the notion of motivic Tambara functors, as first defined in \cite{Bachmann-GW}.

\bigskip

You are currently reading \Cref{section:intro}.

\bigskip

In \Cref{section:Preliminaries}, we set up the theory of Mackey and Tambara functors indexed over an LCCDC category, as well as some category-theoretic machinery we will need in the proofs to follow.

\bigskip

In \Cref{section:Incompleteness}, we introduce the notion of an \emph{index}, which encodes the data necessary to define (bi-)incomplete Tambara functors (\`a la Blumberg-Hill) indexed over an LCCDC category. We also introduce the notion of \emph{separability}, a crucial property of the category of finite $G$-sets which gives rise to many important features in the theory of Mackey and Tambara functors.

\bigskip

\Cref{section:Hoyer} contains our main result: we take an important theorem of Hoyer from the theory of $G$-Mackey/Tambara functors and boost it up to the general (bi-)incomplete LCCDC setting.

\bigskip

It is worth taking a moment here to discuss the history of this result. Mazur \cite{Mazur} originally proved a version of this theorem in the case that $G$ is a cyclic $p$-group. Hoyer \cite{Hoyer} then reformulated and reproved this result for an arbitrary finite group $G$. Chan \cite{Chan} later generalized Hoyer's result to the bi-incomplete setting. We will henceforth refer to these pre-existing instances of the result as the \emph{Hoyer-Mazur-Chan Theorem}. Here, we will prove a further generalization, which includes both (bi-)incompleteness and the replacement of $G{-}\set$ with an arbitrary LCCDC category.

\bigskip

\begin{letterthm}[cf. \Cref{actual-main-theorem}]
  \label{main-theorem-intro-version}
  Let $\C$ be a locally essentially small LCCDC category, and let $\mathcal{O} = (\mathcal{O}_a, \mathcal{O}_m)$ be a compatible pair of indexing subcategories of $\C$. For each morphism $i : x \to y$ lying in $\mathcal{O}_m$, there is a square
  \[\begin{tikzcd}[ampersand replacement=\&]
      {\STamb{\C/x,\mathcal{O}/x}} \& {\STamb{\C/y,\mathcal{O}/y}} \\
      {\SMack{\C/x,\mathcal{O}_a/x}} \& {\SMack{\C/y, \mathcal{O}_a/y}}
      \arrow[from=1-1, to=1-2]
      \arrow[from=1-1, to=2-1]
      \arrow[from=2-1, to=2-2]
      \arrow[from=1-2, to=2-2]
    \end{tikzcd}\]
  of functors between categories of semi-Mackey and semi-Tambara functors which commutes up to natural isomorphism. The vertical maps are the canonical forgetful functors and the top map is left adjoint to the canonical ``restriction along $i$'' functor. When $(\C,\mathcal{O})$ satisfies additional hypotheses and $i$ is an epimorphism, we can remove ``semi'' everywhere, i.e. these functors restrict to the categories of Mackey and Tambara functors.
\end{letterthm}

This theorem recovers, in particular, Hoyer's theorem (\cite{Hoyer}, Theorem 2.1.1) and Chan's generalization (\cite{Chan}, Theorem 7.7) by choosing $\C$ to be the category of finite $G$-sets and $i$ to be the unique morphism $G/H \to G/G$ for $H$ a subgroup of $G$ (with appropriate choices of indexing categories). The new content of this theorem is the extension to arbitrary LCCDC categories. In particular, we can apply this theorem to the motivic setting to obtain a result about categories of motivic Mackey and Tambara functors, which we discuss in \Cref{section:Consequences}.
\begin{letterthm}[cf. \Cref{cor:apply-to-motivic}]
  Fix a base scheme $S$, and let $\fet_S$ denote the category of schemes finite {\'e}tale over $S$. For any morphism $f : X \to Y$ in $\fet_S$, we obtain a square of categories
  \[\begin{tikzcd}[ampersand replacement=\&]
      {T(X)} \& {T(Y)} \\
      {M(X)} \& {M(Y)}
      \arrow["{N_f}", from=1-1, to=1-2]
      \arrow["{U_Y}", from=1-2, to=2-2]
      \arrow["{U_X}"', from=1-1, to=2-1]
      \arrow[from=2-1, to=2-2]
    \end{tikzcd}\]
  which commutes up to natural isomorphism, where $M$ and $T$ denote the categories of motivic Mackey functors and motivic Tambara functors, respectively, as defined in \cite{Bachmann-GW}. The morphism $N_f$ is left adjoint to the canonical restriction functor, and the morphisms $U_X$ and $U_Y$ are the canonical forgetful functors.
\end{letterthm}

\subsection*{Acknowledgements}

The author would like to thank Peter Haine, whose side comment in a casual conversation inspired a key simplification in the proof of \Cref{actual-main-theorem}. Sincere thanks also goes to Mike Hill for his guidance over many years, helpful discussions on every aspect of this project, and careful review of earlier versions of this article. Thanks also to J.D. Quigley for helpful comments and suggestions.

This work is primarily a recompilation of the author's disseration \cite{mythesis}, and was supported by NSF Grant DMS-2136090.

\section{Preliminaries}
\label{section:Preliminaries}

In this section, we introduce the notion of an LCCDC category (\Cref{defn:LCCDC}), and give an introduction to Mackey and Tambara functors. We also prove some basic facts about LCCDC categories and Mackey/Tambara functors which will be essential in the following sections.

Our motivating example of an LCCDC category is the category $G{-}\set$ of finite $G$-sets. This gives rise to the classical notions of $G$-Mackey functors and $G$-Tambara functors. Our other key example is $\fet/S$, the category of schemes finite \'etale over $S$, which gives rise to the notions of motivic Mackey and Tambara functors introduced by Bachmann \cite{Bachmann-GW}. Because LCCDC categories encompass both of these examples, they provide a framework for translating results between equivariant and motivic homotopy theory -- we perform one such translation in \Cref{section:Hoyer} and discuss its consequences in \Cref{section:Consequences}.

\subsection{An Introduction to LCC Categories}

A category $\C$ is said to be \emph{cartesian} if it admits all finite products (including the empty product, i.e. a terminal object). Dually, a category is said to be \emph{cocartesian} if it admits all finite coproducts. Of course, any complete category is in particular cartesian, but there are many important examples of categories which are cartesian but not complete -- e.g. the category of schemes.

``Cartesian'' is a property of categories, but it can also be viewed as a structure. Every cartesian category $\C$ has a ``canonical'' monoidal structure called the \emph{cartesian monoidal structure} -- the monoidal operation is given by the categorical product, and the unit is the terminal object. The word canonical is in quotes in the preceeding sentence because this description does not, strictly speaking, uniquely determine a monoidal structure on $\C$. However, it does determine a monoidal structure unique up to unique isomorphism, as things usually go in category theory.

So, in the case that $\C$ is cartesian, we may ask whether or not this canonical monoidal structure is \emph{closed} -- i.e. if, for all objects $x \in \C$, the functor $x \times {-} : \C \to \C$ admits a right adjoint. If this condition is met, we say that $\C$ is \emph{cartesian closed}. This is a much stronger condition than simply being cartesian -- for example, the category of schemes is cartesian but not cartesian closed.

Finally, for any property $P$ of categories, we can speak of categories which are ``locally $P$'', meaning that each slice category has property $P$. To recall:

\begin{defn}
  Given a category $\C$ and an object $x \in \C$, the slice category $\C/x$ is the category whose:
  \begin{itemize}
    \item Objects are morphisms in $\C$ with codomain $x$;
    \item Morphisms $\alpha \to \beta$ are morphisms $\gamma : a \to b$ in $\C$ such that $\beta \circ \gamma = \alpha$, where $a$ and $b$ are the domains of $\alpha$ and $\beta$, respectively, as in the diagram below
          \[\begin{tikzcd}[ampersand replacement=\&,column sep=tiny]
              a \&\& b \\
              \& x
              \arrow["\alpha"', from=1-1, to=2-2]
              \arrow["\beta", from=1-3, to=2-2]
              \arrow["\gamma", from=1-1, to=1-3]
            \end{tikzcd};\]
    \item Composition is performed as in $\C$.
  \end{itemize}
\end{defn}

There are a few particular instances of local properties we will be most interested in here, so we highlight these:

\begin{defn}
  A category $\C$ is said to be
  \begin{enumerate}
    \item \emph{locally cartesian} iff the slice category $\C/x$ is cartesian for all objects $x \in \C$;
    \item \emph{locally cocartesian} iff the slice category $\C/x$ is cocartesian for all objects $x \in \C$;
    \item \emph{locally cartesian closed} iff the slice category $\C/x$ is cartesian closed for all objects $x \in \C$.
  \end{enumerate}
  We will often abbreviate ``locally cartesian closed'' as \emph{LCC}.
\end{defn}

It is worth making explicit what these properties amount to. $\C/x$ always has a terminal object, namely $\id_x$. Now, given two objects $\alpha, \beta \in \C/x$, it turns out that their product in $\C/x$ (if it exists) is simply their pullback in $\C$, i.e. the limit of the diagram $\bullet \xrightarrow{\alpha} x \xleftarrow{\beta} \bullet$. Thus, $\C$ being locally cartesian is equivalent to $\C$ admitting all pullbacks.

Now consider a morphism $i : x \to y$ in $\C$. There is a canonical functor $\Sigma_i : \C/x \to \C/y$, given by $\Sigma_i(\alpha) = i \circ \alpha$, as in the diagram below:
\[\begin{tikzcd}[ampersand replacement=\&]
    \bullet \\
    x \& y
    \arrow["i"', from=2-1, to=2-2]
    \arrow["\alpha"', from=1-1, to=2-1]
    \arrow["{\Sigma_i \alpha}", from=1-1, to=2-2]
  \end{tikzcd}\]

Assuming that $\C$ is locally cartesian guarantees that this functor has a right adjoint $i^*$, given by sending $\beta \in \C/y$ to its pullback $i^*\beta$ in the diagram
\[\begin{tikzcd}[ampersand replacement=\&]
    \bullet \& \bullet \\
    x \& y
    \arrow["i"', from=2-1, to=2-2]
    \arrow["\beta", from=1-2, to=2-2]
    \arrow["{i^*\beta}"', from=1-1, to=2-1]
    \arrow[from=1-1, to=1-2]
    \arrow["\lrcorner"{anchor=center, pos=0.125}, draw=none, from=1-1, to=2-2]
  \end{tikzcd}\]
More precisely, there may be many functors $i^*$ which are right adjoint to $\Sigma_i$, but for each of them we have that $i^*\beta$ and $\beta$ fit in a pullback square with $i$ as above. And, as always, $i^*$ is unique up to unique isomorphism. In practice, however, we often find that the category $\C$ comes equipped with a canonical pullback construction, yielding a canonical choice of $i^*$.

To summarize the story so far: having that $\C$ is locally cartesian (i.e. $\C/x$ is cartesian for all $x$) is equivalent to saying that, for each morphism $i : x \to y$ in $\C$, the functor $\Sigma_i : \C/x \to \C/y$ admits a right adjoint $i^*$.

Being LCC is indeed a stronger condition -- it amounts to saying that the right adjoint to each $\Sigma_i$ admits a further right adjoint.

\begin{prop}
  \label{prop:characterization-of-LCC}
  A category $\C$ is LCC if and only if, for all morphisms $i : x \to y$ in $\C$, the functor $\Sigma_i : \C/x \to \C/y$ fits in an adjoint triple $\Sigma_i \dashv i^* \dashv \Pi_i$.
\end{prop}

For a proof, we refer the reader to \cite{Elephant}, Corollary A1.5.3.

Again, $\Sigma_i$ is always precisely defined by $\alpha \mapsto i \circ \alpha$, but $i^*$ and $\Pi_i$ are only determined up to unique isomorphism, as are the particular choices of adjunction data (i.e. the unit and counit of each adjunction). In any case, these functors are often known by the names ``dependent sum'' ($\Sigma_i$), ``pullback'' ($i^*$), and ``dependent product'' ($\Pi_i$).

In light of \Cref{prop:characterization-of-LCC}, an LCC category comes with a plethora of functors between its slices. These operations (in some sense) encode the structure of Tambara functors, and so it will be necessary to get a handle on the behavior of these operations and the slice categories of an LCC category. We refer the curious reader to \cite{Gambino-Kock} for a thorough introduction to the topic, with more detail than we will be able to cover here.

First, we note one small fact of functor yoga. Using our previous observation that products in slice categories are given by pullback squares, we have that
\begin{prop}
  \label{prop:product-is-LR}
  For any morphism $i : x \to y$ in an LCC category $\C$, $\Sigma_i \circ i^*$ and $i \times {-}$ are isomorphic as endofunctors of $\C/y$.
\end{prop}
\begin{proof}
  Let $\alpha$ be an arbitrary object of $\C/y$. Pulling back along $i$ yields a cartesian square
  \[\begin{tikzcd}[ampersand replacement=\&]
      \bullet \& \bullet \\
      \bullet \& \bullet
      \arrow["{i^* \alpha}"', from=1-1, to=2-1]
      \arrow["i"', from=2-1, to=2-2]
      \arrow["\alpha", from=1-2, to=2-2]
      \arrow[from=1-1, to=1-2]
      \arrow["\lrcorner"{anchor=center, pos=0.125}, draw=none, from=1-1, to=2-2]
    \end{tikzcd}\]
  Now $\Sigma_i i^* \alpha = i \circ i^* \alpha$ exhibits the top-left corner of the square as an object over $y$ -- since this square is cartesian, this is is also the product of $i$ and $\alpha$ in $\C/y$.
\end{proof}

Now, we make an observation about local properties of categories in general.

\begin{prop}
  \label{slice-of-locally-P-is-locally-P}
  Let $P$ be a property of categories which is invariant under isomorphism, and let $\C$ be a category which is locally $P$. Then for all objects $x \in \C$, $\C/x$ is locally $P$.
\end{prop}
\begin{proof}
  The key idea is:
  \begin{slogan}
    \label{slogan:slice-of-slice-is-slice}
    A slice of a slice is a slice.
  \end{slogan}
  Let $\alpha : a \to x$ be an object of $\C/x$. Then $(\C/x)/\alpha$ is isomorphic to $\C/a$ by sending an object $\beta : b \to a$ of $\C/a$ to $\beta : \alpha \circ \beta \to \alpha$ in $(\C/x)/\alpha$ (and acting as the identity on morphisms). Since $\C/a$ is $P$ by assumption, $(\C/x)/\alpha$ is also $P$. Since $\alpha$ was arbitrary, we conclude that $\C/x$ is locally $P$.
\end{proof}

\begin{cor}
  \label{cor:slice-of-lcc-is-lcc}
  Every slice category of an LCC category is LCC.
\end{cor}

Next, we make an interesting observation about LCC categories -- they admit no nontrivial morphisms to initial objects (just as how, in $\Set$, there are no functions from a nonempty set to the empty set).

\begin{prop}
  \label{prop:lcc-slice-over-empty}
  Let $\C$ be a locally cartesian closed category. Let $x \in \C$ be some object and let $\varnothing \in \C$ be an initial object. Then any morphism $x \to \varnothing$ is an isomorphism; i.e. $\C/\varnothing$ is equivalent to the terminal category.
\end{prop}
\begin{proof}
  In the category $\C/\varnothing$, $\id_\varnothing$ is both initial and terminal. Since $\C$ is locally cartesian closed, $\C/\varnothing$ is cartesian closed, and so we apply \Cref{lem:cc-with-zero-object} (below).
\end{proof}

\begin{lem}
  \label{lem:cc-with-zero-object}
  Let $\C$ be a cartesian closed category with an object $0 \in \C$ which is both initial and terminal. Then $\C$ is equivalent to the terminal category; i.e. every object in $\C$ is isomorphic to $0$.
\end{lem}
\begin{proof}
  Let $x,y \in \C$ be arbitrary. Then $\C(x,y) \cong \C(x \times 0, y)$ because $0$ is terminal, and $\C(x \times 0, y) \cong \C(0, [x,y])$ by cartesian closure. Finally, $\C(0,[x,y])$ is a singleton because $0$ is initial. Since $y$ was arbitrary, this shows that $x$ is initial, and thus $x \cong 0$. Since $x$ was arbitrary, this completes the proof.
\end{proof}

\subsubsection{Coproducts in LCC Categories}

Being locally cocartesian is a condition which (to the author's knowledge) does not admit a nice reformulation in terms of adjoint functors. However, it is worth noting that (in contrast with cartesian structure) cocartesian structure is automatically inherited by slice categories.

\begin{prop}
  \label{slice-of-cocartesian-is-cocartesian}
  Let $\C$ be a cocartesian category. Then $\C$ is locally cocartesian, and in particular the coproduct of a tuple of objects $(\alpha_i : a_i \to x)_{i =1}^n$ in $\C/x$ is the object $(\alpha_i)_{i=1}^n : \coprod_{i=1}^n a_i \to x$ with structure morphisms $\alpha_j \to \coprod_{i=1}^n \alpha_i$ equal to the structure morphisms $a_j \to \coprod_{i=1}^n a_i$ in $\C$.
\end{prop}

In this paper we will consider categories which are locally cartesian closed, cocartesian, and satisfying one additional condition, namely that finite coproducts are \emph{disjoint}.

\begin{defn}
  A cocartesian category $\C$ is said to have \emph{disjoint coproducts} if, for all coproduct diagrams $x \xrightarrow{i} x \amalg y \xleftarrow{j} y$, $i$ and $j$ are monomorphisms, and
  \[\begin{tikzcd}[ampersand replacement=\&]
      \varnothing \& x \\
      y \& {x \amalg y}
      \arrow[from=1-1, to=1-2]
      \arrow[from=1-1, to=2-1]
      \arrow["j"', from=2-1, to=2-2]
      \arrow["i", from=1-2, to=2-2]
    \end{tikzcd}\]
  is a cartesian square, where $\varnothing$ is an initial object of $\C$.
\end{defn}

Intuitively, this says that the coproduct in $\C$ behaves more like a disjoint union operation (in e.g. $\Set$) than a $\max$ operation (in e.g. a poset). Indeed, $\Set$ is an example of a cocartesian category with disjoint coproducts, while the poset $(\{0,1\}, \leq)$ is an example of a cocartesian category whose coproducts are not disjoint.

The importance of this condition for us is that in cocartesian, locally cartesian closed categories satisfying this disjointness property, slice categories over coproducts are well-behaved -- an object living over $a \amalg b$ splits into a piece living over $a$ and a piece living over $b$.

\begin{prop}
  \label{prop:disjoint-coproducts-slice-products}
  Let $\C$ be a category which is locally cartesian closed and cocartesian with disjoint coproducts. Then any coproduct diagram $x \xrightarrow{i} x \amalg y \xleftarrow{j} y$ induces an equivalence of categories
  \[\C/(x \amalg y) \xrightarrow{(i^*, j^*)} \C/x \times \C/y.\]
\end{prop}
We relegate the proof to \Cref{appendix:lcc-lemmas}.

It will be convenient to have a shorthand phrase for ``locally cartesian closed categories which are cocartesian with disjoint coproducts'', since these will be our main objects of study. Thus,

\begin{defn}\label{defn:LCCDC}
  We will say that a category $\C$ is LCCDC if it is locally cartesian closed, cocartesian, and has disjoint coproducts.
\end{defn}

Just as with LCC, LCCDC is a local property. That is,

\begin{prop}
  \label{prop:slice-of-lccdc-is-lccdc}
  Let $\C$ be an LCCDC category. Then, for all $x \in \C$, the slice category $\C/x$ is LCCDC.
\end{prop}
\begin{proof}
  That $\C/x$ is LCC is covered by \Cref{cor:slice-of-lcc-is-lcc}. \Cref{slice-of-cocartesian-is-cocartesian} says that $\C/x$ is cocartesian, and moreover that a coproduct diagram in $\C/x$ is a coproduct diagram $a \xrightarrow{i} a \amalg b \xleftarrow{j} b$ in $\C$ which happens to lie over $x$. The assumption that $\C$ has disjoint coproducts says that the cospan $a \xrightarrow{i} a \amalg b \xleftarrow{j} b$ has pullback $\varnothing$ in $\C$. Consequently, the pullback of $a \xrightarrow{i} a \amalg b \xleftarrow{j} b$ in $\C/x$ is an object over $x$ whose domain (as an object of $\C$) admits a map to $\varnothing$. By \Cref{prop:lcc-slice-over-empty}, this shows that the limit of $a \xrightarrow{i} a \amalg b \xleftarrow{j} b$ in $\C/x$ is also $\varnothing$, which is initial in $\C/x$ by \Cref{slice-of-cocartesian-is-cocartesian}. Since $i$ and $j$ are monic in $\C$, they are also monic in $\C/x$.
\end{proof}

We now record a couple straightforward lemmas about the mechanics of LCCDC categories.

\begin{lem}
  \label{lem:restriction-commutes-with-sum}
  Let $\C$ be an LCCDC category and let $f : x \to y$ and $g : x' \to y'$ be morphisms in $\C$. Let $i_x : x \to x \amalg x'$ and $i_y : y \to y \amalg y'$ be the canonical inclusions. Then
  \[\begin{tikzcd}[ampersand replacement=\&]
      x \& y \\
      {x \amalg x'} \& {y \amalg y'}
      \arrow["f", from=1-1, to=1-2]
      \arrow["{i_x}"', from=1-1, to=2-1]
      \arrow["{f \amalg g}"', from=2-1, to=2-2]
      \arrow["{i_y}", from=1-2, to=2-2]
    \end{tikzcd}\]
  is a cartesian square.
\end{lem}
We again relegate this proof to \Cref{appendix:lcc-lemmas}.

\begin{cor}
  \label{cor:fold-square-is-cartesian}
  Let $\C$ be an LCCDC category and let $f : x \to y$ be a morphism in $\C$. Let $\nabla_x : x \amalg x \to x$ and $\nabla_y : y \amalg y \to y$ be the codiagonals. Then
  \[\begin{tikzcd}[ampersand replacement=\&]
      {x \amalg x} \& {y \amalg y} \\
      x \& y
      \arrow["{f \amalg f}", from=1-1, to=1-2]
      \arrow["{\nabla_x}"', from=1-1, to=2-1]
      \arrow["f"', from=2-1, to=2-2]
      \arrow["{\nabla_y}", from=1-2, to=2-2]
    \end{tikzcd}\]
  is a cartesian square.
\end{cor}
\begin{proof}
  $f^* : \C/y \to \C/x$ is both a left and right adjoint, so it preserves finite coproducts and terminal objects. Thus, we have
  \[f^* \nabla_y = f^*(\id_y \amalg \id_y) = f^* \id_y \amalg f^* \id_y = \id_x \amalg \id_x = \nabla_x.\]
  This establishes that there is a cartesian square of the form
  \[\begin{tikzcd}
      {x \amalg x} & {y \amalg y} \\
      x & y
      \arrow[from=1-1, to=1-2]
      \arrow["{\nabla_x}"', from=1-1, to=2-1]
      \arrow["{\nabla_y}", from=1-2, to=2-2]
      \arrow["f"', from=2-1, to=2-2]
    \end{tikzcd}\]
  and we need only identify the top morphism. To do so, we form a further pullback
  \[\begin{tikzcd}
      \bullet & y \\
      {x \amalg x} & {y \amalg y} \\
      x & y
      \arrow[from=1-1, to=1-2]
      \arrow[from=1-1, to=2-1]
      \arrow["\lrcorner"{anchor=center, pos=0.125}, draw=none, from=1-1, to=2-2]
      \arrow["{i_1}", from=1-2, to=2-2]
      \arrow["{\id_y}", curve={height=-30pt}, from=1-2, to=3-2]
      \arrow[from=2-1, to=2-2]
      \arrow["{\nabla_x}"', from=2-1, to=3-1]
      \arrow["{\nabla_y}", from=2-2, to=3-2]
      \arrow["f"', from=3-1, to=3-2]
    \end{tikzcd}\]
  since the right-hand triangle commutes, the left-hand column must compose to $f^* \id_y = \id_x$, and thus the top morphism in this diagram must be $f$. We would get a similar result forming a further pullback along $i_2 : y \to y \amalg y$, and thus by \Cref{prop:disjoint-coproducts-slice-products} we have established the claim.
\end{proof}

\begin{lem}\label{lem:epi-from-initial}
  Let $\C$ be a category with finite disjoint coproducts. Let $\varnothing$ be an initial object of $\C$, and let $\varphi : \varnothing \to z$ be an epimorphism. Then $\varphi$ is an isomorphism.
\end{lem}
\begin{proof}
  Since $\varphi$ is an epimorphism and $\varnothing$ is initial, we have $\lvert \C(z,y) \rvert \leq 1$ for all objects $y$. In particular, the two coprojections $i_1, i_2 : z \to z \amalg z$ are equal. Note that the following two squares are both cartesian:
  \[\begin{tikzcd}
      \varnothing & z & z & z \\
      z & {z \amalg z} & z & {z \amalg z}
      \arrow[from=1-1, to=1-2]
      \arrow[from=1-1, to=2-1]
      \arrow["{i_2}", from=1-2, to=2-2]
      \arrow["\id", from=1-3, to=1-4]
      \arrow["\id"', from=1-3, to=2-3]
      \arrow["{i_1}", from=1-4, to=2-4]
      \arrow["{i_1}"', from=2-1, to=2-2]
      \arrow["{i_1}"', from=2-3, to=2-4]
    \end{tikzcd}\]
  Since $i_1 = i_2$, the top-left corners of these squares must be isomorphic, i.e. $\varnothing \cong z$. Thus $\varphi$ is an isomorphism.
\end{proof}

\begin{lem}\label{lem:dep-prod-of-initial}
  Let $i : x \to y$ be an epimorphism in an LCCDC category $\C$. Then $\Pi_i : \C/x \to \C/y$ sends initial objects to initial objects.
\end{lem}
\begin{proof}
  We will use the symbol $\varnothing$ to denote all initial objects, with the specific meaning to be interpreted from context. The existence of the counit $i^* \Pi_i \varnothing \to \varnothing$ forces $i^* \Pi_i \varnothing = \varnothing$ by \Cref{prop:lcc-slice-over-empty}. Thus, we have a cartesian square
  \[\begin{tikzcd}
      \varnothing & z \\
      x & y
      \arrow["{(\Pi_i \varnothing)^* i}", from=1-1, to=1-2]
      \arrow[from=1-1, to=2-1]
      \arrow["\lrcorner"{anchor=center, pos=0.125}, draw=none, from=1-1, to=2-2]
      \arrow["{\Pi_i \varnothing}", from=1-2, to=2-2]
      \arrow["i"', from=2-1, to=2-2]
    \end{tikzcd}\]
  Since $i$ is epic and $(\Pi_i \varnothing)^*$ is a left adjoint, the top morphism is also epic. By \Cref{lem:epi-from-initial}, we conclude that $z = \varnothing$, i.e. $\Pi_i \varnothing = \varnothing$.
\end{proof}

\subsection{An Introduction to Mackey and Tambara Functors}

We will now introduce the notions of Mackey and Tambara functors, which are our key objects of study. This comes in two stages: first, we define categories $\A{\C}$ and $\U{\C}$, which syntactically encode the operations in Mackey and Tambara functors. Then, we define Mackey and Tambara functors to be certain types of functors from these categories to $\Set$.

For the remainder of this section, we fix an LCCDC category $\C$. The prototypical example of such a category for us is the category of finite $G$-sets, where $G$ is some group. Here are some other examples:

\begin{exam}
  \label{exam:etale-is-LCCDC}
  The category $\etale$ of topological spaces with local homeomorphisms between them is an LCCDC category. To see this, we make use of a crucial fact: if
  \[\begin{tikzcd}
      X && Y \\
      \\
      & Z
      \arrow["f", from=1-1, to=1-3]
      \arrow["h"', from=1-1, to=3-2]
      \arrow["g", from=1-3, to=3-2]
    \end{tikzcd}\]
  is a commutative diagram of topological spaces, and $h$ and $g$ are local homeomorphisms, then $f$ is a local homeomorphism. Thus, every slice category $\etale/X$ is the same as the category of \'etal\'e spaces over $X$, with all continuous maps between them. We conclude that $\etale/X$ is equivalent to $\Sh{X}$ (the category of sheaves of sets on $X$) by the \'etal\'e space construction \cite[\S{}II.5 and \S{}II.6]{SIGAL}. We know that $\Sh{X}$ is cartesian closed, so we conclude that $\etale$ is LCC. Additionally, $\etale$ clearly admits finite disjoint coproducts (by disjoint union).
\end{exam}

\begin{exam}
  Any Grothendieck topos is an LCCDC category. This is because Grothendieck topoi admit finite disjoint coproducts, Grothendieck topoi are cartesian closed, and every slice of a Grothendieck topos is a Grothendieck topos.
\end{exam}

\begin{exam}
  Let $\fet$ denote the category of schemes with finite {\'e}tale maps between them. Then $\fet$ is an LCCDC category (see \cite{Bachmann-GW}), and thus so is $\fet/S$ for any scheme $S$.
\end{exam}

\subsubsection{The Lindner Category}

\begin{defn}
  The \emph{Lindner category} of $\C$, denoted $\A{\C}$, is the category whose:
  \begin{itemize}
    \item Objects are the same as the objects of $\C$;
    \item Morphisms $x \to y$ are isomorphism classes of diagrams $x \leftarrow z \rightarrow y$ in $\C$, where two diagrams $x \leftarrow z \rightarrow y$ and $x \leftarrow z' \rightarrow y$ are said to be isomorphic if and only if there exists an isomorphism $z \to z'$ making
          \[\begin{tikzcd}[ampersand replacement=\&,row sep=tiny]
              \& z \\
              x \&\& y \\
              \& {z'}
              \arrow[from=1-2, to=2-1]
              \arrow[from=1-2, to=2-3]
              \arrow[from=3-2, to=2-1]
              \arrow[from=3-2, to=2-3]
              \arrow["\cong"{description}, from=1-2, to=3-2]
            \end{tikzcd}\]
          commute.
    \item Composition is given by pullback: given morphisms $[x \leftarrow z \rightarrow y]$ and $[w \leftarrow u \rightarrow x]$, we form a pullback
          \[\begin{tikzcd}[ampersand replacement=\&,row sep=tiny]
              \&\& \bullet \\
              \& u \&\& z \\
              w \&\& x \&\& y
              \arrow[from=2-2, to=3-1]
              \arrow[from=2-2, to=3-3]
              \arrow[from=2-4, to=3-3]
              \arrow[from=2-4, to=3-5]
              \arrow[from=1-3, to=2-2]
              \arrow[from=1-3, to=2-4]
              \arrow["\lrcorner"{anchor=center, pos=0.125, rotate=-45}, draw=none, from=1-3, to=3-3]
            \end{tikzcd}\]
          to obtain a diagram $w \leftarrow \bullet \rightarrow y$, whose isomorphism class is declared to be the composite $[x \leftarrow z \rightarrow y] \circ [w \leftarrow u \rightarrow x]$.
  \end{itemize}
\end{defn}

Since we only needed to construct pullbacks to define the category $\A{\C}$, this construction makes sense for any locally cartesian category $\C$. There are a few unsurprising facts to learn about the Lindner category. First, it is self-dual: ``flipping'' morphisms
\[[x \leftarrow z \rightarrow y] \mapsto [y \leftarrow z \rightarrow x]\]
yields an isomorphism $\A{\C}^\op \to \A{\C}$.

Next, $\A{\C}$ is essentially small whenever $\C$ is --- of course, the collection of objects of $\A{\C}$ is always in bijection with that of $\C$, and when $\C_0$ is a small skeleton of $\C$, any morphism $x \to y$ in $\A{\C}$ can be realized by a span $x \leftarrow z \rightarrow y$ in $\C$ with $z \in \C_0$, from which it follows that $\A{\C}(x,y)$ is small.

Our last unsurprising fact is that every morphism $[x \xleftarrow{f} z \xrightarrow{g} y]$ in $\A{\C}$ factors as $T_g \circ R_f$, where
\begin{align*}
  T_g & := [z \xleftarrow{\id_z} z \xrightarrow{g} y]  \\
  R_f & := [x \xleftarrow{f} z \xrightarrow{\id_z} z];
\end{align*}
in other words, we have
\[[x \xleftarrow{f} z \xrightarrow{g} y] = [z \xleftarrow{\id_z} z \xrightarrow{g} y] \circ [x \xleftarrow{f} z \xrightarrow{\id_z} z]\]
for all $f,g$.

There is another important perspective one can take on the category $\A{\C}$:

\begin{slogan}
  \label{slogan:A-is-syntax}
  The morphisms $T_g$ and $R_f$ in $\A{\C}$ syntactically model the functors $\Sigma_g$ and $f^*$ (respectively) between the slices of $\C$.
\end{slogan}

In other words, assigning to each object $x$ the slice category $\C/x$ and to each morphism $T_g \circ R_f$ the functor $\Sigma_g \circ f^*$ yields a faithful functor from $\A{\C}$ to the category of categories with isomorphism classes of functors between them. In particular, for every cartesian square
\[\begin{tikzcd}[ampersand replacement=\&]
    \bullet \& \bullet \\
    \bullet \& \bullet
    \arrow["f"', from=1-1, to=2-1]
    \arrow["g", from=1-1, to=1-2]
    \arrow["{g'}"', from=2-1, to=2-2]
    \arrow["{f'}", from=1-2, to=2-2]
    \arrow["\lrcorner"{anchor=center, pos=0.125}, draw=none, from=1-1, to=2-2]
  \end{tikzcd}\]
in $\C$, the equality $T_g \circ R_f = R_{f'} \circ T_{g'}$ is reflected by the fact that $\Sigma_g \circ f^* \cong (f')^* \circ \Sigma_{g'}$. This isomorphism of functors is sometimes known as the \emph{Beck-Chevalley isomorphism}.

As a result, facts about dependent sum and pullback translate to give facts about $\A{\C}$. With \Cref{prop:disjoint-coproducts-slice-products} in mind, we obtain:
\begin{prop}
  \label{product_in_A}
  If $\C$ is LCCDC, then $\A{\C}$ admits all finite products, given by the coproduct in $\C$. That is:
  \begin{itemize}
    \item Any initial object of $\C$ is terminal in $\A{\C}$;
    \item Given a coproduct diagram $x \xrightarrow{i_1} x \amalg y \xleftarrow{i_2} y$ in $\C$,
          \[x \xleftarrow{R_{i_1}} x \amalg y \xrightarrow{R_{i_2}} y\]
          is a product diagram in $\A{\C}$.
  \end{itemize}
\end{prop}
\begin{proof}
  First, let $\varnothing$ be initial in $\C$. A morphism $x \to \varnothing$ in $\A{\C}$ is given by a diagram $x \leftarrow z \rightarrow \varnothing$ in $\C$. \Cref{prop:lcc-slice-over-empty} nows tells us that this diagram is isomorphic to $x \leftarrow \varnothing \rightarrow \varnothing$, and thus is uniquely determined up to isomorphism. Thus, $\varnothing$ is terminal in $\A{\C}$.

  Next, we claim that the natural transformation
  \[\A{\C}({-},x \amalg y) \xrightarrow{((R_{i_1})_*, (R_{i_2})_*)} \A{\C}({-},x) \times \A{\C}({-},y)\]
  has inverse
  \[\A{\C}({-},x) \times \A{\C}({-},y) \xrightarrow{(T_{i_1})_* \times (T_{i_2})_*} \A{\C}({-},x \amalg y) \times \A{\C}({-},x \amalg y) \xrightarrow{+} \A{\C}({-},x \amalg y),\]
  where $+$ sends a pair of morphisms
  \[([t \leftarrow z \rightarrow x \amalg y], [t \leftarrow w \rightarrow x \amalg y])\]
  to
  \[([t \leftarrow z \amalg w \rightarrow x \amalg y]).\]
  Checking that these natural transformations are inverses is a direct translation of the proof of \Cref{prop:disjoint-coproducts-slice-products}.
\end{proof}

Since $\A{\C}$ is self-dual, these finite products are also finite coproducts, and indeed $\A{\C}$ admits all finite biproducts. We state \Cref{product_in_A} in the form above to align both with \Cref{prop:disjoint-coproducts-slice-products} and (later) with \Cref{product_in_U}.

By \Cref{prop:slice-of-lccdc-is-lccdc} and \Cref{product_in_A}, we see:

\begin{cor}
  If $\C$ is an LCCDC category, then we can also speak of $\A{\C/x}$ for any object $x \in \C$, which again admits all finite products.
\end{cor}

\subsubsection{Mackey Functors}

A Mackey functor is a functor $F : \A{\C} \to \Set$ satisfying two important conditions. The first is easy to state, and must be assumed in order to even interpret the second. Thus, we give functors satisfying just this first condition a name:

\begin{defn}
  A \emph{$\C$-semi-Mackey functor} is a finite-product-preserving functor $\A{\C} \to \Set$. The category of semi-Mackey functors (indexed by $\C$), denoted $\SMack{\C}$, is the full subcategory of $\Fun{\A{\C}}{\Set}$ spanned by the semi-Mackey functors.
\end{defn}

In light of \Cref{product_in_A}, a semi-Mackey functor $F$ satisfies $F(x \amalg y) \cong F(x) \times F(y)$ for all objects $x,y \in \A{\C}$, where $\amalg$ denotes the coproduct in $\C$. More specifically, if $x \xrightarrow{i_1} x \amalg y \xleftarrow{i_2} y$ is a coproduct diagram, then this isomorphism is given by $(F(R_{i_1}), F(R_{i_2})) : F(x \amalg y) \to F(x) \times F(y)$. Now, given a semi-Mackey functor $F$ and an object $x \in \A{\C}$, we can consider the morphism \[\nabla := (\id_x, \id_x) : x \amalg x \to x\] in $\C$, which yields the morphism $T_\nabla : x \amalg x \to x$ in $\A{\C}$. Then since $F$ is finite-product-preserving, we obtain a binary operation
\[F(x) \times F(x) \xrightarrow{(F(R_{i_1}), F(R_{i_2}))^{-1}} F(x \amalg x) \xrightarrow{F(T_\nabla)} F(x)\]
which we denote by $+_{F,x}$ (or simply $+$ when clear from context). It is not hard to check that this operation is associative and has an identity element, i.e.

\begin{prop}
  $+_{F,x}$ makes $F(x)$ into a commutative monoid.
\end{prop}

We should make note of what this identity element is. Since $\varnothing$ (the initial object of $\C$) is terminal in $\A{\C}$, $F$ being finite-product-preserving implies that $F(\varnothing)$ is a singleton. Now given an object $x$, we take the unique morphism ${!} : \varnothing \to x$ and consider $T_! : \varnothing \to x$ in $\A{\C}$. $F(T_{!})$ is then a function from the singleton set $F(\varnothing)$ to $F(x)$. The element of $F(x)$ in the image of this function is the identity element of $(F(x),+)$. That this element actually is an identity essentially follows from the categorical fact $\varnothing \amalg {-}$ is naturally isomorphic to the identity functor $\C \to \C$.

So, each semi-Mackey functor $F : \A{\C} \to \Set$ comes with a canonical commutative monoid structure on each of its output objects, and actually even more is true -- we can fully upgrade $F$ to a functor $\A{\C} \to \CMon$.

\begin{prop}
  \label{prop:smack-factors}
  If $F : \A{\C} \to \Set$ is a semi-Mackey functor, then $F$ factors uniquely through the forgetful functor $\CMon \to \Set$. This unique factorization is given by endowing each output set $F(x)$ with the binary operation $+_{F,x}$.
\end{prop}
We relegate the proof to \Cref{appendix:lcc-lemmas}.

Noting that the forgetful functor $\CMon \to \Set$ preserves and reflects products, we have

\begin{cor}
  \label{cor:equivalent-characterization-semimackey}
  $\SMack{\C}$ is isomorphic to the category of finite-product-preserving functors $\A{\C} \to \CMon$ via postcomposition with the forgetful functor $\CMon \to \Set$.
\end{cor}

We are now ready to state the full definition of a Mackey functor:

\begin{defn}
  A \emph{$\C$-Mackey functor} is a semi-Mackey functor $F : \A{\C} \to \Set$ such that $(F(x),+)$ is an abelian group for all $x$. The category of Mackey functors (indexed by $\C$), denoted $\Mack{\C}$, is the full subcategory of $\SMack{\C}$ spanned by the Mackey functors.
\end{defn}

In other words, a semi-Mackey functor is Mackey if and only if, for all objects $x$, the binary operation $+_{F,x}$ admits inverses. \Cref{cor:equivalent-characterization-semimackey} tells us that, equivalently, $\Mack{\C}$ can be viewed as the category of finite-product-preserving functors $\A{\C} \to \Ab$.

Since $\Ab$ is a reflective and coreflective subcategory of $\CMon$, it follows that $\Mack{\C}$ is reflective and coreflective in $\SMack{\C}$. That is:

\begin{prop}
  The inclusion $\Mack{\C} \to \SMack{\C}$ admits both a left and right adjoint.
\end{prop}
\begin{proof}
  Let ${(-)}^+, {(-)}^* : \CMon \to \Ab$ denote the left and right adjoints (respectively) to the inclusion $\Ab \to \CMon$. To remind, for a commutative monoid $X$, $X^+$ is the abelian group generated by the elements of $X$ modulo the relations present in $X$, and $X^*$ is the submonoid of invertible elements in $X$.

  These functors both preserve products: ${(-)}^*$ clearly because it is a right adjoint, and ${(-)}^+$ because it is a left adjoint and $\CMon$ and $\Ab$ both have finite biproducts. Thus, postcomposition with these functors define functors $\SMack{\C} \to \Mack{\C}$. It follows formally that these functors are left and right adjoint to the inclusion $\Mack{\C} \to \SMack{\C}$.
\end{proof}

We denote the left adjoint of the inclusion $\Mack{\C} \to \SMack{\C}$ by ${(-)}^+$ and the right adjoint ${(-)}^*$.

\subsubsection{The Polynomial Category}
\label{subsection:polynomial-category}

We now produce from $\C$ a category $\U{\C}$, from which we will define the notion of $\C$-Tambara functors. When $\C = G{-}\set$, these will exactly coincide with the standard notions of $G$-Tambara functors.

\begin{defn}
  The \emph{Polynomial Category} of $\C$, denoted $\U{\C}$, is the category whose:
  \begin{itemize}
    \item Objects are the same as the objects of $\C$;
    \item Morphisms $x \to y$ are isomorphism classes of diagrams $x \leftarrow z \rightarrow w \rightarrow y$, where two diagrams $x \leftarrow z \rightarrow w \rightarrow y$ and $x \leftarrow z' \rightarrow w' \rightarrow y$ are said to be isomorphic if and only if there exist isomorphisms $z \to z'$ and $w \to w'$ making
          \[\begin{tikzcd}[ampersand replacement=\&,row sep=tiny]
              \& z \& w \\
              x \&\&\& y \\
              \& {z'} \& {w'}
              \arrow[from=1-2, to=1-3]
              \arrow[from=1-3, to=2-4]
              \arrow[from=3-2, to=2-1]
              \arrow[from=3-2, to=3-3]
              \arrow[from=3-3, to=2-4]
              \arrow[from=1-2, to=2-1]
              \arrow["\cong"{description}, from=1-2, to=3-2]
              \arrow["\cong"{description}, from=1-3, to=3-3]
            \end{tikzcd}\]
          commute.
  \end{itemize}
\end{defn}

In order to define the composition in $\U{\C}$, we will temporarily introduce an auxilliary construction of a category $\U{\C}'$. The objects of $\U{\C}'$ are the same as those of $\U{\C}$, i.e. they are the same as the objects of $\C$. Each morphism $f : x \to y$ in $\C$ will give rise to three distinguished morphisms in $\U{\C}'$, denoted by $T_f : x \to y$, $N_f : x \to y$, and $R_f : y \to x$. The category $\U{\C}'$ will be generated by these morphisms, modulo some relations which we will now describe.

Just as with $\A{\C}$, these generating morphisms are meant to syntactically encode the functors $\Sigma_f$, $f_*$, and $\Pi_f$ between slices of $\C$ which are given to us by the locally cartesian closed structure. As such, the morphisms of type $T$ and $R$ will compose exactly as in $\A{\C}$, and it suffices to explain how composition with morphisms of type $N$ works.

First of all, we set $N_a \circ N_b = N_{a \circ b}$ for any pair of composable morphisms $(a,b)$ in $\C$. Next, given any cartesian square
\[\begin{tikzcd}[ampersand replacement=\&]
    a \& b \\
    c \& d
    \arrow["{g'}", from=1-1, to=1-2]
    \arrow["{f'}"', from=1-1, to=2-1]
    \arrow["f"', from=2-1, to=2-2]
    \arrow["g", from=1-2, to=2-2]
    \arrow["\lrcorner"{anchor=center, pos=0.125}, draw=none, from=1-1, to=2-2]
  \end{tikzcd}\]
we set $R_g \circ N_f = N_{f'} \circ R_{g'}$.

Finally, we introduce a complicated composition relation. Given a pair of composable morphisms $x \xrightarrow{f} y \xrightarrow{g} z$ in $\C$, we take a dependent product along $g$ to get
\[\begin{tikzcd}[ampersand replacement=\&]
    \& \bullet \\
    y \& z
    \arrow["g"', from=2-1, to=2-2]
    \arrow["{\Pi_g f}", from=1-2, to=2-2]
  \end{tikzcd}\]
and then form a pullback along $g$ to get
\[\begin{tikzcd}[ampersand replacement=\&]
    \bullet \& \bullet \\
    y \& z
    \arrow["g"', from=2-1, to=2-2]
    \arrow["{\Pi_g f}", from=1-2, to=2-2]
    \arrow["{(\Pi_g f)^* g}", from=1-1, to=1-2]
    \arrow["{g^*\Pi_g f}"', from=1-1, to=2-1]
    \arrow["\lrcorner"{anchor=center, pos=0.125}, draw=none, from=1-1, to=2-2]
  \end{tikzcd}\]
Now the counit of the adjunction $g^* \dashv \Pi_g$ gives a morphism
\[\varepsilon^{\text{coind}}_{f} : g^* \Pi_g f \to f\]
in $\C/y$, and so in total we have a commutative diagram (called a \emph{distributor diagram}\footnote{These are also sometimes called \emph{exponential diagrams} in the literature.})
\[\begin{tikzcd}[ampersand replacement=\&]
    \& \bullet \& \bullet \\
    x \& y \& z
    \arrow["g"', from=2-2, to=2-3]
    \arrow["{\Pi_g f}", from=1-3, to=2-3]
    \arrow["{(\Pi_g f)^* g}", from=1-2, to=1-3]
    \arrow[from=1-2, to=2-2]
    \arrow["\lrcorner"{anchor=center, pos=0.125}, draw=none, from=1-2, to=2-3]
    \arrow["f"', from=2-1, to=2-2]
    \arrow["{\varepsilon^{\text{coind}}_f}"', from=1-2, to=2-1]
  \end{tikzcd}\]
in $\C$. We then declare that
\[N_g \circ T_f = T_{\Pi_g f} \circ N_{(\Pi_g f)^* g} \circ R_{\varepsilon^{\text{coind}}_f}.\]

With these generators and relations (plus the relations between $T$'s and $R$'s as in $\A{\C}$) in place, we certainly obtain some category $\U{\C}'$. Moreover, every morphism in $\U{\C}'$ can be written in the form $T_f N_g R_h$, since a composite of two such morphisms can be reduced as

\begin{align*}
  (TNR)(TNR) & \rightsquigarrow TN(TR)NR  \\
             & \rightsquigarrow T(TNR)RNR \\
             & \rightsquigarrow TNRNR     \\
             & \rightsquigarrow TN(NR)R   \\
             & \rightsquigarrow TNR.
\end{align*}

Moreover, it turns out that two parallel morphisms $T_f N_g R_h$ and $T_{f'} N_{g'} R_{h'}$ are equal if and only if the bispans
\[\bullet \xleftarrow{h} \bullet \xrightarrow{g} \bullet \xrightarrow{f} \bullet\]
and
\[\bullet \xleftarrow{h'} \bullet \xrightarrow{g'} \bullet \xrightarrow{f'} \bullet\]
are isomorphic -- see \cite[Lemma 2.15]{Gambino-Kock} for a proof.
Thus, for any objects $x,y$, we obtain a bijective correspondence between the Hom-sets $\U{\C}(x,y)$ and $\U{\C}'(x,y)$. We can then transport the composition operation from $\U{\C}'$ to $\U{\C}$, and henceforth entirely identify these categories. With this identification in place, we have that
\begin{align*}
  T_f & = [x \xleftarrow{\id_x} x \xrightarrow{\id_x} x \xrightarrow{f} y] \\
  N_f & = [x \xleftarrow{\id_x} x \xrightarrow{f} y \xrightarrow{\id_y} y] \\
  R_f & = [y \xleftarrow{f} x \xrightarrow{\id_x} x \xrightarrow{\id_x} x]
\end{align*}
and
\[T_f N_g R_h = [x \xleftarrow{h} z \xrightarrow{g} w \xrightarrow{h} y].\]

\begin{slogan}
  \label{slogan:P-is-syntax}
  The morphisms $T_h$, $N_g$, and $R_f$ in $\U{\C}$ syntactically model the functors $\Sigma_h$, $\Pi_g$, and $f^*$ (respectively) between the slices of $\C$.
\end{slogan}

For a careful account of the category $\U{\C}$, its construction, and its properties, we refer the reader to \cite{Gambino-Kock} (for general LCC categories) and \cite{Strickland} (in the case of $G{-}\set$).

As opposed to $\A{\C}$, $\U{\C}$ is typically not self-dual. However, it is still true that $\U{\C}$ is essentially small whenever $\C$ is. And, as with $\A{\C}$, the existence of finite disjoint coproducts in $\C$ induces finite products in $\U{\C}$.

\begin{prop}
  \label{product_in_U}
  If $\C$ is LCCDC, then $\U{\C}$ admits all finite products, given by the coproduct in $\C$. That is:
  \begin{itemize}
    \item Any initial object of $\C$ is terminal in $\U{\C}$;
    \item Given a coproduct diagram $x \xrightarrow{i_1} x \amalg y \xleftarrow{i_2} y$ in $\C$,
          \[x \xleftarrow{R_{i_1}} x \amalg y \xrightarrow{R_{i_2}} y\]
          is a product diagram in $\U{\C}$.
  \end{itemize}
\end{prop}
\begin{proof}
  The proof is exactly the same as that of \Cref{product_in_A}, since the natural transformations involved are all of the form $T_*$ and $R_*$, which compose in $\U{\C}$ in exactly the same way that they do in $\A{\C}$.
\end{proof}

\subsubsection{Tambara Functors}

\begin{defn}
  A \emph{$\C$-semi-Tambara functor} is a finite-product-preserving functor $\U{\C} \to \Set$. We write $\STamb{\C}$ to denote the category of $\C$-semi-Tambara functors (which is a full subcategory of $\Fun{\U{\C}}{\Set}$).
\end{defn}

Now given a semi-Tambara functor $F : \U{\C} \to \Set$ and an object $x \in \U{\C}$, the morphism $T_\nabla : x \amalg x \to x$ in $\U{\C}$ yields a binary operation $+_{F,x}$ on $F(x)$, and the morphism $N_\nabla : x \amalg x \to x$ yields a binary operation $\cdot_{F,x}$ on $F(x)$. As before, $(F(x), +_{F,x})$ is a commutative monoid, and by the same argument so is $(F(x), \cdot_{F,x})$. Distributor diagrams are so named because the composition relation they impose in $\U{\C}$ says in particular that $\cdot_{F,x}$ distributes over $+_{F,x}$. Thus, $(F(x), +, \cdot)$ is a commutative semiring\footnote{A semiring is a triple $(A,+,\cdot)$ satisfying the same axioms as a ring, except that we do not require the existence of additive inverses. These are also sometimes known as \emph{rig}s.} for all objects $x$.

\begin{defn}
  A \emph{Tambara functor} (indexed by $\C$) is a semi-Tambara functor $F$ such that $(F(x),+,\cdot)$ is a ring for all $x$. We write $\Tamb{\C}$ to denote the category of $\C$-Tambara functors (which is a full subcategory of $\STamb{\C}$).
\end{defn}

\begin{warning}
  As opposed to (semi-)Mackey functors, Tambara functors \emph{cannot} be viewed as functors into commutative monoids, or semirings, etc. This is because, given a semi-Tambara functor $F : \U{\C} \to \Set$, the functions $F(N_f)$ will generally not respect the additive structure, and the functions $F(T_f)$ will generally not respect the multiplicative structure. Indeed, in the other direction, if $F : \U{\C} \to \CMon$ preserves products, the composition $\U{\C} \to \CMon \to \Set$ with the forgetful functor will be a semi-Tambara functor, and the Eckman-Hilton argument will show that $+$ and $\cdot$ coincide on each $F(x)$. Then by uniqueness of identity elements, we will have that $0 = 1$ in each commutative semiring $F(x)$, and thus $F(x) = 0$ for all $x$.
\end{warning}

Since the morphisms of type $T$ and $R$ in $\U{\C}$ compose exactly as in $\A{\C}$, $\A{\C}$ embeds as a wide subcategory of $\U{\C}$. More explictly, this embedding functor $e : \A{\C} \to \U{\C}$ acts as the identity on objects and acts on morphisms by
\[[x \xleftarrow{f} z \xrightarrow{g} y] \mapsto [x \xleftarrow{f} z \xrightarrow{\id} z \xrightarrow{g} y].\]
\Cref{product_in_A} and \Cref{product_in_U} tell us that $e$ sends product diagrams in $\A{\C}$ to product diagrams in $\U{\C}$; i.e. $e$ preserves products. Thus, precomposition with $e$ yields a ``forgetful functor''
\[\STamb{\C} \to \SMack{\C}\]
which we will denote by $U$. This forgetful functor also preserves the binary operation $+$, i.e. $+_{F,x}$ and $+_{U(F),x}$ are equal for all semi-Tambara functors $F$ and all objects $x$. Thus, $U$ restricts to a functor $\Tamb{\C} \to \Mack{\C}$ (which we will also denote by $U$). From this point of view, a Tambara functor is just a semi-Tambara functor whose underlying semi-Mackey functor is Mackey. In total, we obtain a pullback diagram
\[\begin{tikzcd}
    {\Tamb{\C}} && {\STamb{\C}} \\
    \\
    {\Mack{\C}} && {\SMack{\C}}
    \arrow[hook, from=1-1, to=1-3]
    \arrow["U"', from=1-1, to=3-1]
    \arrow["U", from=1-3, to=3-3]
    \arrow[hook, from=3-1, to=3-3]
  \end{tikzcd}\]

For $\C = \fet/S$, this recovers precisely the notion of motivic Tambara functors introduced by Bachmann \cite{Bachmann-GW}.

\begin{prop}
  Fix a scheme $S$, and let $\fet$ be the category of schemes with finite {\'e}tale morphisms between them. Then a $\fet/S$-Tambara functor is precisely the notion of \emph{Tambara functor} defined in \cite{Bachmann-GW} (which are also called \emph{naive motivic Tambara functors} in \cite{Bachmann-MTF}).
\end{prop}

\section{(Bi-)Incompleteness and Separability}
\label{section:Incompleteness}

In equivariant homotopy theory, $G$-Tambara functors arise as the $\pi_0$ of genuine $G{-}E_\infty$ ring spectra, by which we mean an algebra in genuine $G$-spectra for a so-called ``$G{-}E_\infty$'' operad. The operations encoded by such an operad give rise to the norm maps in the resulting Tambara structure on $\pi_0$, and $G{-}E_\infty$ ring spectra arise naturally in practice (they played a key role in the work of Hill, Hopkins, and Ravenel in their resolution of the Kervaire Invariant One Problem \cite{HHR}). However, it is also common to encounter equivariant ring spectra without quite as much structure: for example, a Bousfield localization of a $G{-}E_\infty$ ring spectrum always admits a homotopy-coherent multiplication (on $\pi_0$, this gives norms along codiagonals $x \amalg x \to x$), but will not always retain the structure of a $G{-}E_\infty$ algebra. So, we can ask: what are the possible collections of norm maps that an equivariant spectrum may admit? The answer is certainly not ``any collection whatsoever'', since for example the collection of norm maps which a given spectrum admits will be closed under composition.

In \cite{Blumberg-Hill-Ninfty}, Blumberg and Hill axiomatize these possible collections of norm maps and define the operads indexing these collections of operations, which they christened ``$N_\infty$ operads''. Furthermore, in combination with the following work of Bonventre-Pereira \cite{BonventrePereira}, Gutierrez-White \cite{GutierrezWhite}, and Rubin \cite{Rubin2021}, they show that the homotopy category of $N_\infty$ operads (for a fixed finite group $G$) is equivalent to a finite poset (the poset of \emph{indexing systems} for the group $G$). For a fixed $N_\infty$ operad $\mathcal{O}$, the $\pi_0$ of an $\mathcal{O}$-algebra in geniune $G$-spectra will be a Tambara functor admitting some norms but perhaps not all -- such objects were christened ``incomplete Tambara functors'' in \cite{Blumberg-Hill-Incomplete}. This unifies the study of Tambara functors and Green functors -- the maximal $N_\infty$ operad parametrizes the operations of a Tambara functor, while the minimal $N_\infty$ operad parametrizes the operations of a Green functor. Thus, general $N_\infty$ operads interpolate between these two structures.

The very same indexing systems which describe admissable collections of norms can also be viewed as describing admissable collections of transfers in a Mackey functor, and so we can also consider ``incomplete Mackey functors'', and from here we could ask what is possible if one wishes to simultaneously restrict the admissible norms \emph{and} transfers of a Tambara functor. This was first explored by Blumberg and Hill in \cite{Blumberg-Hill-BiIncomplete}, and the resulting notion of \emph{bi-incomplete Tambara functors} was introduced. These bi-incomplete Tambara functors arise, for example, as the $\pi_0$ of algebras in equivariant spectra for an $N_\infty$ operad (specifying the admissable norms), where the equivariant stable homotopy category is developed with respect to a not-necessarily-complete $G$-universe (specifying the admissable transfers).

(Bi-)incompleteness also turns out to be very important in the study of motivic Tambara functors. When motivic Tambara functors were first introduced in \cite{Bachmann-GW}, Bachmann considered only complete Tambara functors indexed by $\fet/S$, since his main objects of interest had this structure. Later, in \cite{Bachmann-MTF} and \cite{Bachmann-Hoyois}, Bachmann and Bachmann-Hoyois investigate incomplete and bi-incomplete Tambara functors indexed by $\Sm/S$, $\fet/S$, and related categories. Actually, the categories over which they index are not neccesarily LCCDC (e.g. $\Sm/S$ is not, because it does not admit dependent products along all morphisms), so it is not quite right to say (in our terminology) that Bachmann and Hoyois study bi-incomplete Tambara functors indexed over $\Sm/S$. Instead, their observation is that $\Sm/S$ does admit dependent products along finite {\'e}tale morphisms, and so a version of the polynomial category can still be constructed so as to ensure that every dependent product which needs to be computed does indeed exist, from which these sorts of Tambara functors can be defined. This exactly parallels (at a more categorical level) the idea of bi-incomplete Tambara functors, as we will see in this section. In forthcoming work, we will investigate this not-quite-LCCDC situation further.

The point of developing this theory here is to prove (in the proceeding section) a new theorem for generalized Tambara functors, e.g. naive motivic Tambara functors. Mazur \cite{Mazur} (for cyclic $p$-groups) and Hoyer \cite{Hoyer} (for arbitrary finite groups) showed that Tambara functors are the same as $G$-commutative monoids in Mackey functors, in the sense of \cite{Hill-Hopkins}. Later, Chan \cite{Chan} generalized their work to prove the conjecture of Blumberg and Hill that bi-incomplete Tambara functors are the same as $\OO$-commutative monoids in incomplete Mackey functors. Our work in the proceeding section will generalize this further to the context of bi-incomplete Tambara functors indexed over arbitrary LCCDC categories.

In this section, we generalize the theory of (bi-)incomplete Mackey and Tambara functors to the LCCDC context. We also introduce the notion of \emph{separability} (\Cref{defn:separable}), a condition on LCCDC categories which gives rise to desireable behavior in our two primary examples of interest (the categories $G{-}\set$ and $\fet/S$).

\subsection{Indexing Subcategories and Compatibility}

\begin{defn}
  \label{defn:indexing-category}
  Let $\C$ be a cocartesian and locally cartesian category. A subcategory $\OO$ of $\C$ is said to be an \emph{indexing category on $\C$} (or an \emph{indexing subcategory of $\C$}) if it is:
  \begin{enumerate}
    \item Wide, i.e. all objects of $\C$ are also objects of $\OO$;
    \item Pullback stable, i.e. for all morphisms $f$ in $\OO$ and all cartesian squares
          \[\begin{tikzcd}[ampersand replacement=\&]
              \bullet \& \bullet \\
              \bullet \& \bullet
              \arrow[from=1-1, to=1-2]
              \arrow["f", from=1-2, to=2-2]
              \arrow["{f'}"', from=1-1, to=2-1]
              \arrow[from=2-1, to=2-2]
              \arrow["\lrcorner"{anchor=center, pos=0.125}, draw=none, from=1-1, to=2-2]
            \end{tikzcd}\]
          in $\C$, $f'$ is in $\OO$;
    \item Finite-coproduct complete, i.e. the initial object of $\C$ is also initial in $\OO$, and every binary coproduct diagram of $\C$ is a binary coproduct diagram in $\OO$ (and in particular lies in $\OO$).
  \end{enumerate}
\end{defn}

\begin{defn}
  Let $\C$ be a cocartesian and locally cartesian category and let $\OO$ be an indexing category on $\C$. The category $\A{\C,\OO}$ is defined to be the wide subcategory of $\A{\C}$ containing precisely the morphisms $T_f R_g$ such that $f \in \OO$.
\end{defn}

The second condition in \Cref{defn:indexing-category} ensures that $\A{\C,\OO}$ is closed under composition; and the first condition ensures that $\A{\C,\OO}$ contains all identity morphisms of $\A{\C}$. The third condition ensures that $\A{\C,\OO}$ is finite-product complete.

\begin{prop}
  \label{prop:ACOO-fpc-in-AC}
  Let $\C$ be an LCCDC category and let $\OO$ be an indexing category on $\C$. Then $\A{\C,\OO}$ is a finite-product complete subcategory of $\A{\C}$ -- in particular, $\A{\C,\OO}$ admits all finite products and the inclusion $\A{\C,\OO} \to \A{\C}$ is finite-product-preserving.
\end{prop}

In fact, something more general is true. Indexing subcategories of $\C$ form a (possibly large) poset $\mathcal{I}_\C$ under inclusion, and any inclusion of indexing categories yields a finite-product-complete inclusion of Lindner categories.

\begin{prop}
  \label{prop:ACOO-fpc-in-ACOO'}
  Let $\C$ be an LCCDC category and let $\OO \subseteq \OO'$ be an inclusion of indexing subcategories of $\C$. Then $\A{\C,\OO}$ is a finite-product complete subcategory of $\A{\C,\OO'}$.
\end{prop}

Clearly, for any subclass $\mathcal{S} \subseteq \mathcal{I}_\C$, we have $\bigcap \mathcal{S} \in \mathcal{I}_\C$. In particular, $\mathcal{I}_\C$ has a maximum element, namely $\C$ itself, and a minimal element $\OO^{\mathrm{triv}}$, which consists of finite coproducts of morphisms isomorphic to codiagonal maps. In other words, the morphisms in $\OO^{\mathrm{triv}}$ are precisely those of the form
\[\left(\coprod_{i=1}^{a_1} x_1\right) \amalg \dots \amalg \left(\coprod_{i=1}^{a_n} x_n\right) \xrightarrow{(f_1)_{i=1}^{a_1} \amalg \dots \amalg (f_n)_{i=1}^{a_n}} y_1 \amalg \dots \amalg y_n\]
where each $f_i : x_i \to y_i$ is an isomorphism and each $a_i$ is a natural number. \Cref{prop:ACOO-fpc-in-AC} is the special case of \Cref{prop:ACOO-fpc-in-ACOO'} for an inclusion $\OO \subseteq \C$.

\subsubsection{Incomplete Mackey Functors}

The purpose of defining these ``incomplete Lindner categories'' is to index the operations of ``incomplete Mackey functors''. So, we make the following definition.

\begin{defn}
  Let $\C$ be an LCCDC category and let $\OO$ be an indexing category on $\C$. An \emph{$(\C,\OO)$-semi-Mackey functor} is a finite-product-preserving functor $\A{\C,\OO} \to \Set$. The category of $(\C,\OO)$-semi-Mackey functors (denoted $\SMack{\C,\OO}$) is the full subcategory of $\Fun{\A{\C,\OO}}{\Set}$ spanned by the $\OO$-semi-Mackey functors.
\end{defn}

Since $\A{\C,\OO}$ is finite-product-complete in $\A{\C,\OO}$, it contains all the morphisms needed in the definition of the binary operation $+$. Thus, we can also define the notion of a $(\C,\OO)$-Mackey functor.

\begin{defn}
  Let $\C$ be an LCCDC category and let $\OO$ be an indexing category on $\C$. A \emph{$(\C,\OO)$-Mackey functor} is a $(\C,\OO)$-semi-Mackey functor $F$ such that $(F(x),+)$ is an abelian group for all objects $x$. The category of $(\C,\OO)$-Mackey functors (denoted $\Mack{\C,\OO}$) is the full subcategory of $\SMack{\C,\OO}$ spanned by the $(\C,\OO)$-Mackey functors.
\end{defn}

Exactly as before, a $(\C,\OO)$-semi-Mackey functor factors uniquely through $\CMon$, and consequently the inclusion of $(\C,\OO)$-Mackey functors into $(\C,\OO)$-semi-Mackey functors has both a left and right adjoint.

\begin{prop}
  Let $\C$ be an LCCDC category, and let $\OO$ be an indexing subcategory of $\C$. Then every $(\C,\OO)$-semi-Mackey functor $F : \A{\C,\OO} \to \Set$ factors uniquely through the forgetful functor $\CMon \to \Set$.
\end{prop}

\begin{prop}
  Let $\C$ be an LCCDC category, and let $\OO$ be an indexing subcategory of $\C$. Then the inclusion $\Mack{\C,\OO} \to \SMack{\C,\OO}$ has both a left and right adjoint, given by post-composition with the left and right adjoints (respectively) of the inclusion $\Ab \to \CMon$.
\end{prop}

In light of \Cref{prop:ACOO-fpc-in-ACOO'}, an inclusion of indexing subcategories also gives a forgetful functor between the corresponding categories of (semi-)Mackey functors:

\begin{prop}
  \label{prop:forget-mack-to-mack}
  Let $\C$ be an LCCDC category, and let $\OO \subseteq \OO'$ be an inclusion of indexing categories on $\C$. Then precomposition with the inclusion $\A{\C,\OO} \to \A{\C,\OO'}$ yields a forgetful functor from $(\C,\OO')$-(semi-)Mackey functors to $(\C,\OO)$-(semi-)Mackey functors, forming a commutative square
  \[\begin{tikzcd}[ampersand replacement=\&]
      {\Mack{\C,\OO'}} \& {\SMack{\C,\OO'}} \\
      {\Mack{\C,\OO}} \& {\SMack{\C,\OO}}
      \arrow[from=1-1, to=1-2]
      \arrow[from=1-1, to=2-1]
      \arrow[from=2-1, to=2-2]
      \arrow[from=1-2, to=2-2]
    \end{tikzcd}\]
\end{prop}

\subsubsection{Compatible Pairs of Indexing Categories}

\begin{defn}
  Let $\C$ be an LCCDC category, and let $(\OO_a, \OO_m)$ be a pair of indexing subcategories of $\C$. We say that $(\OO_a, \OO_m)$ is \emph{compatible} if, for all morphisms $i : x \to y$ in $\OO_m$ and all morphisms $\alpha : a \to x$ in $\OO_a$, $\Pi_i \alpha$ lies in $\OO_a$.
\end{defn}

\begin{warning}
  Compatibility is not a symmetric notion, i.e. it is possible for one of $(\OO_a, \OO_m)$ and $(\OO_m, \OO_a)$ to be compatible but not the other.
\end{warning}

The definition of indexing category was cooked up precisely so that the morphisms in the indexing category could be used as the $T$-components of spans, yielding a nice subcategory of the entire Lindner category. Likewise, a compatible pair $(\OO_a, \OO_m)$ of indexing categories yields a nice subcategory of the polynomial category, where the morphisms in $\OO_a$ are the $T$-components of morphisms and the morphisms in $\OO_m$ are the $N$-components.

\begin{defn}
  Let $\C$ be an LCCDC category and let $\OO = (\OO_a, \OO_m)$ be a compatible pair of indexing subcategories of $\C$. $\U{\C,\OO}$ is the wide subcategory of $\U{\C}$ containing precisely the morphisms $T_f N_g R_h$ with $f \in \OO_a$ and $g \in \OO_m$.
\end{defn}

The well-definedness of the incomplete Linder category was obvious, but it is less immediately clear that the above description of $\U{\C,\OO}$ is actually well-defined. In particular, we must check that morphisms of the form $T_f N_g R_h$ with $f \in \OO_a$ and $g \in \OO_m$ are closed under composition. The only nontrivial part of this check is that a composition $N_g \circ T_f$ with $g \in \OO_m$ and $f \in \OO_a$ is still of this desired form. Recalling the discussion of distributor diagrams from \Cref{subsection:polynomial-category}, we have
\[N_g \circ T_f = T_{\Pi_g f} \circ N_{(\Pi_g f)^* g} \circ R_{\varepsilon^{\text{coind}}_f}.\]
But we know that $\Pi_g f \in \OO_a$ by the compatiblity condition, and so that $(\Pi_g f)^* g \in \OO_m$ by pullback stability, whence things are fine.

The canonical partial order on indexing categories induces a partial order on compatible pairs of indexing categories: we say $(\OO_a, \OO_m) \leq (\OO'_a, \OO'_m)$ if and only if $\OO_a \subseteq \OO'_a$ and $\OO_m \subseteq \OO'_m$. We use $\mathcal{B}_\C$ to denote the poset of compatible pairs of indexing categories.

\begin{prop}
  Let $\C$ be an LCCDC category and let $\OO \leq \OO'$ be a morphism in $\mathcal{B}_\C$. Then $\U{\C,\OO}$ is a finite-product-complete subcategory of $\U{\C,\OO'}$.
\end{prop}

Again, $\mathcal{B}_\C$ has a maximum element, namely $(\C,\C)$, and so the above proposition implies (in particular) that finite products in $\U{\C,\OO}$ are the same as finite products in $\U{\C}$. And, analogously to the complete setting, we have

\begin{prop}
  Let $\C$ be an LCCDC category and let $\OO_a$ be an indexing subcategory of $\C$. Then, for any indexing subcategory $\OO_m$ of $\C$ such that $(\OO_a, \OO_m)$ is compatible, $\A{\C,\OO_a}$ naturally embeds as a wide, finite-product-complete subcategory of $\U{\C,\OO_a,\OO_m}$, via sending a morphism $[x \xleftarrow{f} z \xrightarrow{g} y]$ to $[x \xleftarrow{f} z \xrightarrow{\id} z \xrightarrow{g} y]$.
\end{prop}

For any indexing category $\OO$, the pair $(\OO, \OO^{\text{triv}})$ is compatible, because forming a dependent product along a codiagonal is the same as forming a product in a slice category -- since $\OO$ is pullback-stable and closed under composition, it follows that it is closed under dependent products along codiagonals.

It is also the case that $(\C, \OO)$ is compatible for all $\OO$ -- the compatibility condition becomes trivial. Overall, we see that $\mathcal{B}_\C$ has a minimum element $(\OO^{\text{triv}}, \OO^{\text{triv}})$ and a maximum element $(\C, \C)$.

For the remainder of this section, we will often have in play a triple $(\C, \OO_a, \OO_m)$ of an LCCDC category $\C$ and a compatible pair $(\OO_a, \OO_m)$ of indexing categories on $\C$. It will become cumbersome to continue writing out such a triple of data in full, so we now introduce a definition to simply the notation in our exposition.

\begin{defn}
  An \emph{index} is a triple $(\C, \OO_a, \OO_m)$, where $\C$ is an LCCDC category and $(\OO_a, \OO_m)$ is a compatible pair of indexing categories on $\C$. We will often abuse notation and use the same symbol (in this case, $\C$) to refer to both the index and the underlying LCCDC category. In this case we will use $\C_a$ and $\C_m$ to denote the two components of the compatible pair of indexing categories.
\end{defn}

If $\C$ is simply an LCCDC category, we may also view it as an index $(\C, \C, \C)$ (the ``fully complete index'' on $\C$). We may do this in the following work when clear from context.

\begin{defn}
  For an index $\C$, we will use $\A{\C}$ to denote the category $\A{\C,\C_a}$ and $\U{\C}$ to denote the category $\U{\C,\C_a,\C_m}$. Note that, when $\C$ is the fully complete index on an LCCDC category, $\A{\C}$ and $\U{\C}$ agree with their definitions from \Cref{section:Preliminaries}.
\end{defn}

\subsubsection{Bi-Incomplete Tambara Functors}

The point of introducing indices was to define a generalized notion of bi-incomplete Tambara functors, which we will now do.

\begin{defn}
  Let $\C$ be an index. A \emph{$\C$-semi-Tambara functor} is finite-product-preserving functor $\U{\C} \to \Set$. The category of $\C$-semi-Tambara functors (denoted $\STamb{\C}$) is the full subcategory of $\Fun{\U{\C}}{\Set}$ spanned by the $\C$-semi-Tambara functors.
\end{defn}

Since finite products in $\U{\C}$ are the same as finite products in $\U{\C}$, each $\C$-semi-Tambara functor comes equipped with a semiring structure on each of its output sets. Thus, we can define

\begin{defn}
  Let $\C$ be an index. A \emph{$\C$-Tambara functor} is a $\C$-semi-Tambara functor $F$ such that $(F(x),+)$ is an abelian group for all objects $x$. The category of $\C$-Tambara functors (denoted $\Tamb{\C}$) is the full subcategory of $\STamb{\C}$ spanned by the $\C$-Tambara functors.
\end{defn}

The rest of the basic setup falls into place exactly as expected.

\begin{prop}
  \label{prop:forget-tamb-to-mack}
  Let $\C$ be an index. Precomposition with the canonical embedding $\A{\C} \to \U{\C}$ defines a forgetful functor $U : \STamb{\C} \to \SMack{\C}$ which identifies $\Tamb{\C}$ as the full subcategory of $\STamb{\C}$ lying over $\Mack{\C}$.
\end{prop}

\begin{prop}
  \label{prop:comm-cube-of-forgetfuls}
  Let $\C$ be an LCCDC category, and let $\OO \leq \OO'$ be an inclusion of pairs of compatible indexing categories on $\C$. Then precomposition with the inclusion $\U{\C,\OO} \to \U{\C,\OO'}$ yields a forgetful functor from $(\C,\OO')$-(semi-)Tambara functors to $(\C,\OO)$-(semi-)Tambara functors, forming a commutative square
  \[\begin{tikzcd}[ampersand replacement=\&]
      {\Tamb{\C,\OO'}} \& {\STamb{\C,\OO'}} \\
      {\Tamb{\C,\OO}} \& {\STamb{\C,\OO}}
      \arrow[from=1-1, to=1-2]
      \arrow[from=1-1, to=2-1]
      \arrow[from=2-1, to=2-2]
      \arrow[from=1-2, to=2-2]
    \end{tikzcd}\]
  This assembles with the forgetful functors of \Cref{prop:forget-tamb-to-mack} and the square of \Cref{prop:forget-mack-to-mack} to form a commutative cube
  \[\begin{tikzcd}[row sep={40,between origins}, column sep={60,between origins}]
      & \Tamb{\C, \OO'} \ar[rr]\ar{dd}\ar{dl} & & \STamb{\C, \OO'} \ar{dd}\ar{dl} \\
      \Tamb{\C, \OO} \ar[crossing over,rr] \ar{dd} & & \STamb{\C, \OO} \\
      & \Mack{\C, \OO'}  \ar{rr} \ar{dl} & &  \SMack{\C, \OO'} \ar{dl} \\
      \Mack{\C, \OO} \ar{rr} && \SMack{\C, \OO} \ar[from=uu,crossing over]
    \end{tikzcd}\]
\end{prop}

We note here some special cases of interest. The complete index $(G{-}\set, G{-}\set, G{-}\set)$ yields the usual notion of $G$-Tambara functors, while $(G{-}\set, G{-}\set, \mathcal{O}^{\mathrm{triv}})$ yields $G$-Green functors. Likewise, $(\fet/S, \fet/S, \fet/S)$ yields (naive) motivic Tambara functors, and $(\fet/S, \fet/S, \mathcal{O}^{\mathrm{triv}})$ yields what we call \emph{(naive) motivic Green functors}.

\subsection{Compatibility with Slices}

As mentioned, we wish to develop this theory of bi-incomplete Tambara functors for the purpose of generalizing a theorem of Hoyer, Mazur, and Chan. This Hoyer-Mazur-Chan theorem concerns the norm functor $\mathcal{N}_H^G : \Tamb{H{-}\set} \to \Tamb{G{-}\set}$ for $H \leq G$ an inclusion of finite groups, and this norm functor is given by left Kan extension along the induction functor $H{-}\set \to G{-}\set$. At first glance, it is not clear how this should be generalized to the context of LCCDC categories -- what relationships would we have in general between two LCCDC categories like $H{-}\set$ and $G{-}\set$? However, there is an interesting observation to be made about this setup which greatly clarifies the situation.

\begin{prop}
  \label{prop:slices-of-G-set}
  Let $G$ be a finite group. Given a morphism $f : X \to G/H$ in $G{-}\set$, the fiber $f^{-1}(H)$ above the trivial coset is a sub-$H$-set of $X$. Sending $f$ to $f^{-1}(H)$ defines an equivalence of categories $G{-}\set/(G/H) \to H{-}\set$ (where we act on morphisms by restriction).
\end{prop}

Now, via the equivalences $G{-}\set/(G/G) \cong G{-}\set$ and $G{-}\set/(G/H) \cong H{-}\set$, the restriction functor $G{-}\set \to H{-}\set$ corresponds to pulling back along the unique map $i : G/H \to G/G$. Thus, induction (which is left adjoint to restriction) corresponds to $\Sigma_i$. From this perspective, we see that the Hoyer-Mazur-Chan theorem is really about a single LCCDC category ($G{-}\set$) and Tambara functors indexed over its slice categories.

For this reason, if we want to prove a bi-incomplete version of the generalized Hoyer-Mazur-Chan theorem, we need to understand how indexing categories interact with slices. The news here is good -- everything works very smoothly.

\begin{defn}
  Let $\OO$ be an indexing subcategory of $\C$. For an object $x \in \C$, we use $\OO/x$ to denote the subcategory of $\C/x$ consisting of precisely those morphisms
  \[\begin{tikzcd}[column sep=tiny]
      a && b \\
      & x
      \arrow["f", from=1-1, to=1-3]
      \arrow[from=1-1, to=2-2]
      \arrow[from=1-3, to=2-2]
    \end{tikzcd}\]
  with $f \in \OO$. Note that $\OO$ is wide in $\C$, so $\OO/x$ is wide in $\C/x$: the object $a \to x$ itself need not lie in $\OO$!
\end{defn}

This does overload our notation -- the meaning of $\OO/x$ now depends on whether $\OO$ is viewed as an indexing subcategory of $\C$ or independently as a category with no relationship to $\C$. In what follows, we will take care to make sure it is clear from context which is meant.

\begin{prop}
  Let $\C$ be an LCCDC category, and let $\OO_a$ be an indexing subcategory of $\C$. Then, for each object $x \in \C$, $\OO_a/x$ is an indexing subcategory of $\C/x$. Moreover, for any indexing subcategory $\OO_m$ such that $(\OO_a, \OO_m)$ is compatible, $(\OO_a/x, \OO_m/x)$ is a compatible pair of indexing subcategories of $\C/x$.
\end{prop}
\begin{proof}
  Let $(\OO_a, \OO_m)$ be a compatible pair of indexing subcategories of $\C$. Then $\OO_a/x$ is wide (because $\OO_a$ is wide) and pullback-stable (because pullbacks in $\C/x$ are computed as in $\C$). \Cref{slice-of-cocartesian-is-cocartesian} shows that the initial object of $\C/x$ is $\varnothing$ (which is also initial in $\OO/x$), and binary coproduct diagrams in $\C/x$ are simply coproduct diagrams in $\C$ which happen to lie over $x$ -- since the binary coproduct diagrams in $\C$ all lie in $\OO_a$, the binary coproduct diagrams in $\C/x$ all lie in $\OO_a/x$. Thus, $\OO_a/x$ is an indexing subcategory of $\C/x$. We only used that $\OO_a$ is indexing, so $\OO_m/x$ is also an indexing subcategory of $\C/x$.

  For $(\OO_a/x, \OO_m/x)$ to be compatible means that, for all morphisms $i : \alpha \to \beta$ in $\OO_m/x$ and all morphisms $g : \gamma \to \alpha$ in $\OO_a/x$, $\Pi_i g$ lies in $\OO_a/x$. Diagramatically, this setup looks like
  \[\begin{tikzcd}
      \bullet & \bullet & \bullet \\
      & x
      \arrow["g", from=1-1, to=1-2]
      \arrow["\gamma"', from=1-1, to=2-2]
      \arrow["i", from=1-2, to=1-3]
      \arrow["\alpha"', from=1-2, to=2-2]
      \arrow["\beta", from=1-3, to=2-2]
    \end{tikzcd}\]
  with $g \in \OO_a$ and $i \in \OO_m$. Then compatibility of $(\OO_a, \OO_m)$ says that $\Pi_i g \in \OO_a$, as desired.
\end{proof}

In light of the above proposition, when $\C$ is an index, we obtain an index $\C/x$ for each object $x \in \C$.

\begin{prop}
  \label{prop:functor-induced-on-A}
  Let $\C_1$ and $\C_2$ be LCCDC categories, let $\OO_1$ be an indexing category on $\C_1$, and let $\OO_2$ be an indexing category on $\C_2$. Let $F : \C_1 \to \C_2$ be a functor which preserves cartesian squares, and sends morphisms in $\OO_1$ to morphisms in $\OO_2$. Then $F$ induces a functor $\A{F} : \A{\C_1, \OO_1} \to \A{\C_2, \OO_2}$ by sending a morphism $T_f R_g$ to $T_{F(f)} R_{F(g)}$. If $F$ preserves finite coproducts, then $\A{F}$ preserves finite products.
\end{prop}

\begin{cor}
  \label{cor:pi-induced-on-A}
  Let $\C$ be an index. For any morphism $f : x \to y$ in $\C_a$, $\Sigma_f$ induces a finite-product-preserving functor between $\A{\C/x} \to \A{\C/y}$. For any morphism $i : x \to y$ in $\C_m$, $\Pi_i$ induces a functor $\A{\C/x} \to \A{\C/y}$.
\end{cor}
\begin{proof}
  $\Sigma_f$ preserves cartesian squares by \Cref{prop:induction-is-cartesian}, and sends morphisms in $\C_a/x$ to morphisms in $\C_a/y$ since it acts as the identity on morphisms. Additionally, $\Sigma_f$ preserves finite coproducts because it is a left adjoint.

  $\Pi_i$ preserves cartesian squares because it is a right adjoint, and sends morphisms in $\C_a/x$ to morphisms in $\C_a/y$ by the compatibility condition.
\end{proof}

\begin{prop}
  \label{prop:functor-induced-on-P}
  Let $\C_1$ and $\C_2$ be indices. Let $F : \C_1 \to \C_2$ be a functor which preserves cartesian squares, distributor diagrams, and sends morphisms in $(\C_1)_a$ (resp. $(\C_1)_m$) to morphisms in $(\C_2)_a$ (resp. $(\C_2)_m$). Then $F$ induces a functor $\U{F} : \U{\C_1} \to \U{\C_2}$ by sending a morphism $T_f N_g R_h$ to $T_{F(f)} N_{F(g)} R_{F(f)}$. If $F$ preserves finite coproducts, then $\U{F}$ preserves finite products.
\end{prop}

Here, ``preserves distributor diagrams'' means that, given $N_f T_g = T_a N_b R_c$ in $\U{\C_1}$, we also have $N_{F(f)} T_{F(g)} = T_{F(a)} N_{F(b)} R_{F(c)}$ in $\U{\C_2}$.

\begin{cor}
  \label{cor:sigma-induced-on-P}
  Let $\C$ be an index. For any morphism $i : x \to y$ in $\OO_m$, $\Sigma_i$ induces a finite-product-preserving functor $\U{\C/x} \to \U{\C/y}$.
\end{cor}
\begin{proof}
  We know that $\Sigma_i$ preserves cartesian squares by \Cref{prop:induction-is-cartesian}. $\Sigma_i$ sends morphisms in $\C_a/x$ (resp. $\C_m/x$) to morphisms in $\C_a/y$ (resp. $\C_m/y$) because it acts as the identity on morphisms. $\Sigma_i$ preserves finite coproducts because it is a left adjoint. All that remains to be shown is that $\Sigma_i$ preserves distributor diagrams, which turns out to be (essentially) just an application of \Cref{slogan:slice-of-slice-is-slice}. So, begin with a composable pair of morphisms $\alpha \xrightarrow{f} \beta \xrightarrow{g} \gamma$ in $\C/x$. Diagramatically, in $\C$, we have
  \[\begin{tikzcd}
      a & b & c \\
      & x
      \arrow["f", from=1-1, to=1-2]
      \arrow["\alpha"', from=1-1, to=2-2]
      \arrow["g", from=1-2, to=1-3]
      \arrow["\beta"', from=1-2, to=2-2]
      \arrow["\gamma", from=1-3, to=2-2]
    \end{tikzcd}\]
  We then form $\Pi_g f \in (\C/x)/\gamma$, whence $g^* \Pi_g f \in (\C/x)/\beta$, and put these together to form our distributor diagram in $\C/x$:
  \[\begin{tikzcd}
      & \bullet & \bullet \\
      a & b & c \\
      & x
      \arrow[from=1-2, to=1-3]
      \arrow["{\varepsilon^{\mathrm{coind}}_f}"', from=1-2, to=2-1]
      \arrow["{g^*\Pi_g f}", from=1-2, to=2-2]
      \arrow["{\Pi_gf}", from=1-3, to=2-3]
      \arrow["f", from=2-1, to=2-2]
      \arrow["\alpha"', from=2-1, to=3-2]
      \arrow["g", from=2-2, to=2-3]
      \arrow["\beta"', from=2-2, to=3-2]
      \arrow["\gamma", from=2-3, to=3-2]
    \end{tikzcd}\]
  Here, we are viewing $\Pi_g$ as a functor $(\C/x)/\beta \to (\C/x)/\gamma$, defined by being right adjoint to $g^* : (\C/x)/\gamma \to (\C/x)/\beta$. But by \Cref{slogan:slice-of-slice-is-slice}, we have isomorphisms $(\C/x)/\beta \cong \C/b$ and $(\C/x)/\gamma \cong \C/c$ under which $g^* : (\C/x)/\gamma \to (\C/x)/\beta$ is identified with $g^* : \C/c \to \C/b$. Thus, $\Pi_g : (\C/x)/\beta \to (\C/x)/\gamma$ is also identified with $\Pi_g : \C/b \to \C/c$. Likewise, the counit $\varepsilon^{\mathrm{coind}}_f : g^* \Pi_g f \to f$ in $(\C/x)/\alpha$ coincides with the counit $\varepsilon^{\mathrm{coind}}_f : g^* \Pi_g f \to f$ in $\C/b$. In other words, the entire construction of the distributor diagram can be performed simply in $\C$, forgetting the structure maps to $x$.

  Thus, when we form the distributor diagram for the composable pair $(\Sigma_i f, \Sigma_i g)$ in $\C/y$, we can also perform this construction directly in $\C$, forgetting the structure maps to $y$. But $\Sigma_i$ acts as the identity on morphisms, so this distributor diagram we end up constructing for $(\Sigma_i f, \Sigma_i g)$ is precisely the same as the distributor diagram we constructed originally for $(f,g)$ (which is the same as its image under $\Sigma_i$).
\end{proof}

Precomposition with $\U{\Sigma_i} : \U{\C/x} \to \U{\C/y}$ gives the ``restriction'' operation on Tambara functors.

\begin{prop}
  \label{prop:restrict-tamb-to-tamb}
  Let $\C$ be an index. For any morphism $i : x \to y$ in $\C_m$, precomposition with $\U{\Sigma_i} : \U{\C/x} \to \U{\C/y}$ gives a functor $\STamb{\C/y} \to \STamb{\C/x}$ which restricts to a functor $\Tamb{\C/y} \to \Tamb{\C/x}$.
\end{prop}
\begin{proof}
  Let $F \in \Tamb{\C/y}$. Its image under precomposition with $\U{\Sigma_i}$ is a semi-Tambara functor because $\U{\Sigma_i}$ preserves finite products. For any object $\alpha \in \C/x$, the operation $+_{F \circ \U{\Sigma_i}, \alpha}$ is exactly equal to the operation $+_{F, \Sigma_i \alpha}$, which has additive inverses by assumption.
\end{proof}

\subsection{Separability}

A crucial tool in the study of $G$-Tambara functors is the result of Mazur \cite{Mazur} which states in particular that, for all $G$-Tambara functors $S$, all morphisms $X \to Y$ between transitive $G$-sets, and all $a,b \in S(X)$, $S(N_f)(a)+S(N_f)(b)$ is a summand of $S(N_f)(a+b)$. This is in fact true of semi-Tambara functors, and implies that in a semi-Tambara functor, norms between transitive $G$-sets preserve additively invertible elements. We would like to have a similar result in the generalized context, but to do so an additional assumption on the index is required.

\begin{defn}
  A morphism $f : a \to c$ in a category $\C$ is said to be \emph{complemented} if there exists a morphism $g : b \to c$ such that $a \xrightarrow{f} c \xleftarrow{g} b$ is a coproduct diagram.
\end{defn}

\begin{defn}\label{defn:separable}
  Let $\C$ be an index. For any $f : x \to y$ in $\C_m$, letting $\nabla_x : x \amalg x \to x$ denote the codiagonal, we obtain an object $\Pi_f \nabla_x \in \C/y$. Let $j_f : \id_y \to \Pi_f \nabla_x$ denote the adjunct of the first coprojection $f^* \id_y \cong \id_x \to \nabla_x$ (noting that $\nabla_x$ is the coproduct $\id_x \amalg \id_x$ in $\C/x$). We say that $\C$ is \emph{separable} if $j_f$ is complemented in $\C/y$ for all $f \in \C_m$.
\end{defn}

We note that $j_f$ is always a monomorphism, because $\id_y$ is terminal in $\C/y$. So, if $\C$ is such that all monomorphisms in slices of $\C$ are complemented, then $\C$ is separable. For example, this is true of the category of finite $G$-sets.

\begin{prop}
  If $\C = G{-}\set$, then $\C$ is separable.
\end{prop}
\begin{proof}
  In this case (by \Cref{prop:slices-of-G-set} and \Cref{prop:disjoint-coproducts-slice-products}), a slice $\C/y$ is equivalent to $G_1{-}\set \times \dots \times G_n{-}\set$ for some finite collection of finite groups $G_1, \dots, G_n$. Since all monomorphisms in each $G_i{-}\set$ are complemented, all monomorphisms in $\C/y$ are complemented.
\end{proof}

Separability also holds in the motivic context, although this is slightly less trivial.

\begin{prop}
  If $\C = \fet/S$, then $\C$ is separable.
\end{prop}
\begin{proof}
  For all $x \in \fet/S$, $\nabla_x$ is a separated finite {\'e}tale morphism. Then for any finite {\'e}tale $f : x \to y$, $\Pi_f \nabla_x$ is separated finite {\'e}tale (\cite[\S7.6, Proposition 5]{NeronModels} for separated and \cite[\S3]{Bachmann-GW} for finite {\'e}tale). Now we apply \Cref{lemma:section-of-separated-is-complemented} below.
\end{proof}

\begin{lem}\label{lemma:section-of-separated-is-complemented}
  Let $f : X \to Y$ be a separated finite {\'e}tale morphism of schemes. Then any section of $f$ is complemented.
\end{lem}
\begin{proof}
  Let $g : Y \to X$ be a section of $f$. Thinking of $g$ as a morphism of $Y$-schemes, its domain and codomain are {\'e}tale, so $g$ is {\'e}tale. Since $g$ is a section of a morphism, it has degree $1$. Moreover, $\id_Y = f \circ g$ is finite, and $f$ is separated, so $g$ is finite. Since finite morphisms are closed and {\'e}tale morphisms are open, we conclude that $g$ is an isomorphism onto its image, which is clopen in $X$. Thus, $g$ is complemented.
\end{proof}

\section{A Generalized Hoyer-Mazur-Chan Theorem}
\label{section:Hoyer}

In \cite{Hoyer}, Hoyer establishes, for each inclusion of finite groups $H \leq G$, a commutative square

\[\begin{tikzcd}[ampersand replacement=\&]
    {\Tamb{H{-}\set}} \& {\Tamb{G{-}\set}} \\
    {\Mack{H{-}\set}} \& {\Mack{G{-}\set}}
    \arrow[from=1-1, to=1-2]
    \arrow[from=1-1, to=2-1]
    \arrow[from=2-1, to=2-2]
    \arrow[from=1-2, to=2-2]
  \end{tikzcd}\]
\noindent
where the vertical arrows are the forgetful functors, the top arrow is left adjoint to the restriction functor $\Tamb{G{-}\set} \to \Tamb{H{-}\set}$ (and is given by left Kan extension along induction $H{-}\set \to G{-}\set$), and the bottom arrow is given by left Kan extension along coinduction $H{-}\set \to G{-}\set$ (but is not left adjoint to any functor $\Mack{G{-}\set} \to \Mack{H{-}\set}$).

As discussed in \Cref{section:Incompleteness}, via the equivalences $G{-}\set/(G/G) \cong G{-}\set$ and $G{-}\set/(G/H) \cong H{-}\set$, the restriction functor $G{-}\set \to H{-}\set$ corresponds to pulling back along the unique map $i : G/H \to G/G$. Thus, induction (which is left adjoint to restriction) corresponds to $\Sigma_i$, and coinduction (which is right adjoint to restriction) corresponds to $\Pi_i$. We can now rewrite Hoyer's commutative square as

\[\begin{tikzcd}[ampersand replacement=\&]
    {\Tamb{G{-}\set/(G/H)}} \& {\Tamb{G{-}\set/(G/G)}} \\
    {\Mack{H{-}\set/(G/H)}} \& {\Mack{G{-}\set/(G/G)}}
    \arrow["{\Lan_{\U{\Sigma_i}}}", from=1-1, to=1-2]
    \arrow[from=1-1, to=2-1]
    \arrow["{\Lan_{\A{\Pi_i}}}"', from=2-1, to=2-2]
    \arrow[from=1-2, to=2-2]
  \end{tikzcd}\]

The goal of this section is to prove our main result, which is a generalized version of this theorem, where we replace $i : G/H \to G/G$ with an arbitrary morphism. However, our assumptions are not strong enough to ensure the the desired left Kan extension functors exist. So, in what follows, we will assume $\C$ is also locally essentially small, i.e. that $\C/x$ is essentially small for all $x$. All that is really needed is that all of the left Kan extensions referenced below exist, and this is simply a convenient sufficient condition for the existence of these Kan extensions. In the examples of primary interest to us ($\fet$ and $G{-}\set$ for any group $G$), this sufficient condition is satisfied.

\begin{thm}
  \label{actual-main-theorem}
  Let $\C$ be an index which is also locally essentially small. For each $i : x \to y$ in $\C_m$, we obtain a square
  \begin{equation}
    \label{eq:main-theorem-semi-square}
    \begin{tikzcd}
      {\STamb{\C/x}} & {\STamb{\C/y}} \\
      {\SMack{\C/x}} & {\SMack{\C/y}}
      \arrow["{\Lan_{\U{\Sigma_i}}}", from=1-1, to=1-2]
      \arrow[from=1-1, to=2-1]
      \arrow[from=1-2, to=2-2]
      \arrow["{\Lan_{\A{\Pi_i}}}"', from=2-1, to=2-2]
    \end{tikzcd}
  \end{equation}
  which commutes up to natural isomorphism, where the vertical maps are the forgetful functors of \Cref{prop:forget-tamb-to-mack}.

  Moreover, when $\C$ is separable and $i$ is an epimorphism, each functor in this square preserves Mackey structure, i.e. the above square restricts to give a square
  \[\begin{tikzcd}
      {\Tamb{\C/x}} & {\Tamb{\C/y}} \\
      {\Mack{\C/x}} & {\Mack{\C/y}}
      \arrow[from=1-1, to=1-2]
      \arrow[from=1-1, to=2-1]
      \arrow[from=1-2, to=2-2]
      \arrow[from=2-1, to=2-2]
    \end{tikzcd}\]
  which commutes up to natural isomorphism.
\end{thm}

There is prior work in this direction due to Chan \cite{Chan}, who proved this theorem for bi-incomplete $G$-Tambara functors (with $G$ any finite group), where $i$ is a morphism between transitive $G$-sets (note that such morphisms are necessarily surjective). Our generalization is to allow any underlying LCCDC category in the index $\C$, which places some restriction on the techniques we can use in the proof.

Our proof strategy is as follows: first, we note that all of these functors are restricted from functor categories (e.g. $\STamb{\C/x}$ is a full subcategory of $\Fun{\U{\C/x}}{\Set}$), and so square \eqref{eq:main-theorem-semi-square} is restricted from

\begin{equation*}
  \begin{tikzcd}
    \Fun{\U{\C/x}}{\Set} \rar["\Lan_{\U{\Sigma_i}}"] \dar["e^*"] & \Fun{\U{\C/y}}{\Set} \dar["e^*"] \\
    \Fun{\A{\C/x}}{\Set} \rar["\Lan_{\A{\Pi_i}}"] & \Fun{\A{\C/y}}{\Set}
  \end{tikzcd}
\end{equation*}

Thus, it suffices to show that this square commutes. But all of the functors here are cocontinuous\footnote{Notably, these functors are not all cocontinuous between the categories of semi-Mackey and semi-Tambara functors, where colimits are not computed pointwise! It is essential to consider their extensions to the full functor categories here.}, so it is equivalent to show that the two composites

\[\U{\C/x}^{\op} \xrightarrow{\yo} \Fun{\U{\C/x}}{\Set} \rightrightarrows \Fun{\A{\C/y}}{\Set}\]

are naturally isomorphic, where $\yo$ represents the Yoneda embedding.

Once we have established this, we must further show that $\Lan_{\A{\Pi_i}}$ sends Mackey functors to Mackey functors, which will require the additional assumption that $\C$ is separable and $i$ is an epimorphism.

\subsection{Comparing Kan Extensions along \texorpdfstring{$\Sigma$}{Σ} and \texorpdfstring{$\Pi$}{Π}}

In this section, we will complete the first step outlined above, that is, showing that the two composites

\[\U{\C/x}^{\op} \xrightarrow{\yo} \Fun{\U{\C/x}}{\Set} \rightrightarrows \Fun{\A{\C/y}}{\Set}\]

coming from the square

\begin{equation*}
  \begin{tikzcd}
    \Fun{\U{\C/x}}{\Set} \rar["\Lan_{\U{\Sigma_i}}"] \dar["e^*"] & \Fun{\U{\C/y}}{\Set} \dar["e^*"] \\
    \Fun{\A{\C/x}}{\Set} \rar["\Lan_{\A{\Pi_i}}"] & \Fun{\A{\C/y}}{\Set}
  \end{tikzcd}
\end{equation*}

are naturally isomorphic. For readability, we will abuse notation a bit and write $\A{\Pi_i}$ and $\U{\Sigma_i}$ simply as $\Pi_i$ and $\Sigma_i$, respectively, since this is their action on both objects and morphisms.

Taking an arbitrary object $\alpha \in \U{\C/x}^{\op}$ and going around the bottom-left of the square, we get

\[\Lan_{\Pi_i} e^* \yo \alpha = \Lan_{\Pi_i} e^* \U{\C/x}(\alpha,{-}) = \Lan_{\Pi_i} \U{\C/x}(\alpha,e({-})).\]

On the other hand, going around the top-right, we get

\[e^* \Lan_{\Sigma_i} \yo \alpha = e^* \Lan_{\Sigma_i} \U{\C/x}(\alpha,{-}) = e^* \U{\C/y}(\Sigma_i \alpha,{-}) = \U{\C/y}(\Sigma_i \alpha, e({-})).\]

So, we aim to show that, for all $\alpha \in \C/x$, \[\U{\C/y}(\Sigma_i \alpha,e({-})) : \A{\C/y} \to \Set\] is the left Kan extension of \[\U{\C/x}(\alpha,e({-})) : \A{\C/x} \to \Set\] along $\Pi_i : \A{\C/x} \to \A{\C/y}$, and then that this identification with the left Kan extension is natural in $\alpha$.

\begin{notn}
  For readability, we will henceforth elide writing the inclusion functors $e$, and use the shorthand notations
  \begin{align*}
    \U{x} & := \U{\C/x} \\
    \U{y} & := \U{\C/y} \\
    \A{x} & := \A{\C/x} \\
    \A{y} & := \A{\C/y}
  \end{align*}
\end{notn}

Thus, the desired Kan extension will be witnessed by a universal natural transformation $\omega$ as in the following triangle

\begin{equation}\label{eq:omega-triangle}
  \begin{tikzcd}[ampersand replacement=\&]
    {\A{x}} \&\& {\A{y}} \\
    \& \Set
    \arrow[""{name=0, anchor=center, inner sep=0}, "{\U{x}(\alpha, {-})}"', from=1-1, to=2-2]
    \arrow["{\U{y}(\Sigma_i \alpha, {-})}", from=1-3, to=2-2]
    \arrow["{\Pi_i}", from=1-1, to=1-3]
    \arrow["\omega"{description}, shorten <=8pt, shorten >=8pt, Rightarrow, from=0, to=1-3]
  \end{tikzcd}
\end{equation}

Universality here means that, for any functor $F : \A{y} \to \Set$ and any natural transformation $\tau : \U{x}(\alpha,{-}) \to F \circ \Pi_i$, there is a unique natural transformation $\sigma : \U{y}(\Sigma_i \alpha, {-}) \to F$ such that $\sigma \Pi_i \circ \omega = \tau$.

We will begin by constructing $\omega$, then proceed to show it is universal. In what follows, we fix adjunction data $\Sigma_i \dashv i^* \dashv \Pi_i$ in the form of unit/counit pairs.
\begin{notn}
  The unit and counit of $\Sigma_i \dashv i^*$ will be denoted $\eta^{\text{ind}}$ and $\varepsilon^{\text{ind}}$, respectively, and the unit and counit of $i^* \dashv \Pi$ will be denoted $\eta^{\text{coind}}$ and $\varepsilon^{\text{coind}}$, respectively.
\end{notn}
Our proof is similar to Chan's in \cite{Chan}, with some essential differences coming from the fact that $\C$ is arbitrary. A crucial tool in the proof below will be \Cref{prop:Hoyer-Lemma}, which generalizes \cite[Lemma 2.3.5]{Hoyer}. This will be proved in \Cref{appendix:lcc-lemmas}, but we state it here for accessibility to the reader:

\begin{prop*}[\Cref{prop:Hoyer-Lemma}]
  For any morphism $i : x \to y$ in an LCC category $\C$ and any object $b \in \C/y$, the functors $\Pi_{\varepsilon^\text{ind}_b} \circ (\Sigma_i / i^* b)$ and $(\eta^\text{coind}_b)^* \circ (\Pi_i / i^*b)$ are naturally isomorphic.
\end{prop*}

We will also make use of some key facts about the adjunction $\Sigma_i \dashv i^*$. Namely:

\begin{lem*}[\Cref{lem:slogan}, cf. \Cref{slogan:slice-of-slice-is-slice}]
  Let $\C$ be a category, and let $i : x \to y$ be a morphism in $\C$. For any object $\alpha \in \C/x$, the functor $\Sigma_i / \alpha : (\C/x)/\alpha \to (\C/y)/\Sigma_i \alpha$ is an isomorphism.
\end{lem*}

\begin{prop*}[\Cref{lem:dep-sum-pres-refl-pbs}, \Cref{prop:induction-is-cartesian}, \Cref{cor:adjunct-is-cartesian}]
  Let $\C$ be a locally cartesian category and let $i : x \to y$ be a morphism in $\C$. Then:
  \begin{enumerate}
    \item $\Sigma_i$ preserves and reflects cartesian squares;
    \item $i^*$ preserves cartesian squares;
    \item Each naturality square for the unit and counit of the adjunction $\Sigma_i \dashv i^*$ is cartesian;
    \item A commutative square (A) is cartesian if and only if its adjunct (B) is.
          \[\begin{tikzcd}[ampersand replacement=\&]
              {\Sigma_ia} \& c \& a \& {i^* c} \\
              {\Sigma_i b} \& d \& b \& {i^* d}
              \arrow[from=1-1, to=1-2]
              \arrow[""{name=0, anchor=center, inner sep=0}, "{\Sigma_i p}"', from=1-1, to=2-1]
              \arrow[from=2-1, to=2-2]
              \arrow[""{name=1, anchor=center, inner sep=0}, "q", from=1-2, to=2-2]
              \arrow[""{name=2, anchor=center, inner sep=0}, "p"', from=1-3, to=2-3]
              \arrow[from=1-3, to=1-4]
              \arrow[from=2-3, to=2-4]
              \arrow[""{name=3, anchor=center, inner sep=0}, "{i^* q}", from=1-4, to=2-4]
              \arrow["{(A)}"{description}, draw=none, from=0, to=1]
              \arrow["{(B)}"{description}, draw=none, from=2, to=3]
            \end{tikzcd}\]
  \end{enumerate}
\end{prop*}

With these notations and results in place, we are ready to continue the proof.

\subsubsection{The Natural Transformation \texorpdfstring{$\omega$}{ω}}

To define the desired natural transformation $\omega$ from \eqref{eq:omega-triangle}, we begin with an auxilliary construction.

\begin{prop}
  \label{prop:t}
  There is a natural transformation $t$ filling in the square
  \[\begin{tikzcd}[ampersand replacement=\&]
      {\A{x}} \&\& {\A{y}} \\
      \\
      {\U{x}} \&\& {\U{y}}
      \arrow["{\Pi_i}", from=1-1, to=1-3]
      \arrow["e"', from=1-1, to=3-1]
      \arrow["{\Sigma_i}"', from=3-1, to=3-3]
      \arrow["e", from=1-3, to=3-3]
      \arrow["t"{description}, shorten <=6pt, shorten >=6pt, Rightarrow, from=3-1, to=1-3]
    \end{tikzcd}\]
  defined by components
  \[t_{\alpha} := N_{\varepsilon^{\mathrm{ind}}_{\Pi_i \alpha}} R_{\Sigma_i \varepsilon^{\mathrm{coind}}_{\alpha}} = [\Sigma_i \alpha \xleftarrow{\Sigma_i \varepsilon^{\mathrm{coind}}_{\alpha}} \Sigma_i i^* \Pi_i \alpha \xrightarrow{\varepsilon^{\mathrm{ind}}_{\Pi_i \alpha}} \Pi_i \alpha \xrightarrow{\id} \Pi_i \alpha]\]
  for $\alpha \in \A{x}$.
\end{prop}
\begin{proof}
  First, we must ensure that the definition of $t_\alpha$ above parses at all. That is, we must check that $\varepsilon^{\text{ind}}_{\Pi_i \alpha}$ lies in $\OO_m/y$. For this, we note that $\varepsilon^{\text{ind}}_{\id_y} : \Sigma_i i^* \id_y \to \id_y$ is simply the morphism $i : i \to \id_y$, and thus we have a naturality square for $\varepsilon^{\text{ind}}$
  \[\begin{tikzcd}[ampersand replacement=\&]
      {\Sigma_i i^* \Pi_i \alpha} \& {\Pi_i \alpha} \\
      i \& {\id_y}
      \arrow["{\varepsilon^{\text{ind}}_{\Pi_i \alpha}}", from=1-1, to=1-2]
      \arrow[from=1-2, to=2-2]
      \arrow["i"', from=2-1, to=2-2]
      \arrow[from=1-1, to=2-1]
    \end{tikzcd}\]
  By \Cref{prop:induction-is-cartesian}, this square is cartesian, and since $i \in \OO_m/y$, we conclude that $\varepsilon^{\text{ind}}_{\Pi_i \alpha} \in \OO_m/y$ by pullback-stability.

  Next, we must check naturality. So, let $\varphi = [\alpha \xleftarrow{g} \zeta \xrightarrow{f} \beta]$ be an arbitrary morphism in $\A{x}$. Then $\Sigma_i e \varphi = T_{\Sigma_i f} R_{\Sigma_i g}$ and $e \Pi_i \varphi = T_{\Pi_i f} R_{\Pi_i g}$, so we must show that
  \[\begin{tikzcd}[ampersand replacement=\&]
      {\Sigma_i \alpha} \&\& {\Pi_i \alpha} \\
      \\
      {\Sigma_i \beta} \&\& {\Pi_i \beta}
      \arrow["{t_\alpha}", from=1-1, to=1-3]
      \arrow["{T_{\Pi_i f} R_{\Pi_i g}}", from=1-3, to=3-3]
      \arrow["{T_{\Sigma_i f} R_{\Sigma_i g}}"', from=1-1, to=3-1]
      \arrow["{t_\beta}"', from=3-1, to=3-3]
    \end{tikzcd}\]
  commutes in $\U{y}$. First, we will go around the top-right. We have
  \[
    \begin{aligned}
      T_{\Pi_i f} R_{\Pi_i g} t_\alpha & = T_{\Pi_i f} R_{\Pi_i g} N_{\varepsilon^{\mathrm{ind}}_{\Pi_i \alpha}} R_{\Sigma_i \varepsilon^{\mathrm{coind}}_{\alpha}}             & (\text{definition of $t$})                            \\
                                       & = T_{\Pi_i f} N_{\varepsilon^{\mathrm{ind}}_{\Pi_i \zeta}} R_{\Sigma_i i^* \Pi_i g} R_{\Sigma_i \varepsilon^{\mathrm{coind}}_{\alpha}} & (*)                                                   \\
                                       & = T_{\Pi_i f} N_{\varepsilon^{\mathrm{ind}}_{\Pi_i \zeta}} R_{\Sigma_i(\varepsilon^{\mathrm{coind}}_{\alpha} \circ i^* \Pi_i g)}       & (R_a \circ R_b = R_{b \circ a})                       \\
                                       & = T_{\Pi_i f} N_{\varepsilon^{\mathrm{ind}}_{\Pi_i \zeta}} R_{\Sigma_i(g \circ \varepsilon^{\mathrm{coind}}_{\zeta})}                  & (\text{naturality of $\varepsilon^{\mathrm{coind}}$})
    \end{aligned}
  \]
  where the starred equality comes from the square
  \[\begin{tikzcd}[ampersand replacement=\&]
      {\Sigma_i i^* \Pi_i \zeta} \&\& {\Pi_i \zeta} \\
      \\
      {\Sigma_i i^* \Pi_i \alpha} \&\& {\Pi_i \alpha}
      \arrow["{\varepsilon^{\mathrm{ind}}_{\Pi_i \zeta}}", from=1-1, to=1-3]
      \arrow["{\Sigma_i i^* \Pi_i g}"', from=1-1, to=3-1]
      \arrow["{\varepsilon^{\text{ind}}_{\Pi_i \alpha}}"', from=3-1, to=3-3]
      \arrow["{\Pi_i g}", from=1-3, to=3-3]
    \end{tikzcd}\]
  which we know to be cartesian by \Cref{prop:induction-is-cartesian}.

  Next, we go around the bottom-left. To start, we form a cartesian square
  \begin{equation}\label{eq:pullback-of-coinduction-counit}
    \begin{tikzcd}[ampersand replacement=\&]
      \gamma \&\& {i^* \Pi_i \beta} \\
      \\
      \zeta \&\& \beta
      \arrow["{(\varepsilon^{\mathrm{coind}}_{\beta})^* f}", from=1-1, to=1-3]
      \arrow["{h'}"', from=1-1, to=3-1]
      \arrow["{\varepsilon^{\mathrm{coind}}_{\beta}}", from=1-3, to=3-3]
      \arrow["f"', from=3-1, to=3-3]
      \arrow["\lrcorner"{anchor=center, pos=0.125}, draw=none, from=1-1, to=3-3]
    \end{tikzcd}
  \end{equation}
  For convenience, we denote $(\varepsilon^{\text{coind}}_\beta)^* f$ by $f'$. Then by \Cref{prop:induction-is-cartesian},
  \begin{equation}\label{eq:sigma-applied-to-pullback-of-coinduction-counit}
    \begin{tikzcd}[ampersand replacement=\&]
      {\Sigma_i \gamma} \& {\Sigma_i i^* \Pi_i \beta} \\
      {\Sigma_i \zeta} \& {\Sigma_i \beta}
      \arrow["{\Sigma_i h'}"', from=1-1, to=2-1]
      \arrow["{\Sigma_i f}"', from=2-1, to=2-2]
      \arrow["{\Sigma_i f'}", from=1-1, to=1-2]
      \arrow["{\Sigma_i \varepsilon^{\mathrm{coind}}_{\beta}}", from=1-2, to=2-2]
    \end{tikzcd}
  \end{equation}
  is also cartesian. Thus,
  \begin{align*}
    t_\beta T_{\Sigma_i f} R_{\Sigma_i g} & = N_{\varepsilon^{\mathrm{ind}}_{\Pi_i \beta}} R_{\Sigma_i \varepsilon^{\mathrm{coind}}_{\beta}} T_{\Sigma_i f} R_{\Sigma_i g} & (\text{definition of $t$})                                                       \\
                                          & = N_{\varepsilon^{\mathrm{ind}}_{\Pi_i \beta}} T_{\Sigma_i f'} R_{\Sigma_i h'} R_{\Sigma_i g}                                  & (\text{\eqref{eq:sigma-applied-to-pullback-of-coinduction-counit} is cartesian}) \\
                                          & = N_{\varepsilon^{\mathrm{ind}}_{\Pi_i \beta}} T_{\Sigma_i f'} R_{\Sigma_i (g \circ h')}                                       & (R_a \circ R_b = R_{b \circ a})
  \end{align*}
  Next, we will commute the $N_{\varepsilon^{\text{ind}}_{\Pi_i \beta}}$ past the $T_{\Sigma_i f'}$. So, we form a distributor diagram
  \begin{equation}\label{initial-exp-diagram}
    \begin{tikzcd}[ampersand replacement=\&]
      \&\& \bullet \&\& \bullet \\
      \\
      {\Sigma_i \gamma} \&\& {\Sigma_i i^* \Pi_i \beta} \&\& {\Pi_i \beta}
      \arrow["{\varepsilon^{\text{ind}}_{\Pi_i \beta}}"', from=3-3, to=3-5]
      \arrow[from=1-3, to=3-3]
      \arrow["{\Pi_{\varepsilon^{\text{ind}}_{\Pi_i \beta}} \Sigma_i f'}", from=1-5, to=3-5]
      \arrow["{(\Pi_{\varepsilon^{\text{ind}}_{\Pi_i \beta}} \Sigma_i f')^* \varepsilon^{\text{ind}}_{\Pi_i \beta}}", from=1-3, to=1-5]
      \arrow["\lrcorner"{anchor=center, pos=0.125}, draw=none, from=1-3, to=3-5]
      \arrow["{\Sigma_i f'}"', from=3-1, to=3-3]
      \arrow["{\varepsilon^{\text{coind}}_{\Sigma_i f'}}"', from=1-3, to=3-1]
    \end{tikzcd}
  \end{equation}
  By \Cref{prop:Hoyer-Lemma}, we have  a natural isomorphism
  \[\begin{tikzcd}[ampersand replacement=\&]
      {(\C/x)/i^*\Pi_i\beta} \&\& {(\C/y)/\Sigma_i i^* \Pi_i \beta} \\
      \\
      {(\C/y)/\Pi_i i^* \Pi_i\beta} \&\& {(\C/y)/\Pi_i\beta}
      \arrow["{\Sigma_i/i^*\Pi_i\beta}", from=1-1, to=1-3]
      \arrow["{(\eta^{\text{coind}}_{\Pi_i\beta})^*}"', from=3-1, to=3-3]
      \arrow["{\Pi_{\varepsilon^{\text{ind}}_{\Pi_i \beta}}}", from=1-3, to=3-3]
      \arrow["\cong"{description}, draw=none, from=3-1, to=1-3]
      \arrow["{\Pi_i/i^*\Pi_i \beta}"', from=1-1, to=3-1]
    \end{tikzcd}\]
  of functors $(\C/x)/i^* \Pi_i \beta \to (\C/y)/\Pi_i \beta$. The fact that $\Pi_i$ preserves pullbacks gives an isomorphism
  \[(\Pi_i \varepsilon^{\text{coind}}_\beta)^* \to \Pi_i \circ (\varepsilon^{\text{coind}}_\beta)^*\]
  of functors $(\C/x)/\beta \to (\C/y)/\Pi_i i^* \Pi_i \beta$. We can paste this natural isomorphism onto the above square to obtain
  \[\begin{tikzcd}[ampersand replacement=\&]
      \&\& {(\C/x)/i^*\Pi_i\beta} \&\& {(\C/y)/\Sigma_i i^* \Pi_i \beta} \\
      \\
      {(\C/x)/\beta} \&\& {(\C/y)/\Pi_i i^* \Pi_i\beta} \&\& {(\C/y)/\Pi_i\beta}
      \arrow["{\Sigma_i/i^*\Pi_i\beta}", from=1-3, to=1-5]
      \arrow["{(\eta^{\text{coind}}_{\Pi_i\beta})^*}"', from=3-3, to=3-5]
      \arrow["{\Pi_{\varepsilon^{\text{ind}}_{\Pi_i \beta}}}", from=1-5, to=3-5]
      \arrow["\cong"{description}, draw=none, from=3-3, to=1-5]
      \arrow[""{name=0, anchor=center, inner sep=0}, "{\Pi_i/i^*\Pi_i \beta}"', from=1-3, to=3-3]
      \arrow["{(\Pi_i \varepsilon^{\text{coind}}_\beta)^*}"', from=3-1, to=3-3]
      \arrow["{(\varepsilon^{\text{coind}}_\beta)^*}", from=3-1, to=1-3]
      \arrow["\cong"{description}, draw=none, from=3-1, to=0]
    \end{tikzcd}\]
  Then, the unit-counit identities give us
  \[\begin{tikzcd}
      && {(\C/x)/i^*\Pi_i\beta} && {(\C/y)/\Sigma_i i^* \Pi_i \beta} \\
      \\
      {(\C/x)/\beta} && {(\C/y)/\Pi_i i^* \Pi_i\beta} && {(\C/y)/\Pi_i\beta}
      \arrow["{\Sigma_i/i^*\Pi_i\beta}", from=1-3, to=1-5]
      \arrow[""{name=0, anchor=center, inner sep=0}, "{\Pi_i/i^*\Pi_i \beta}"', from=1-3, to=3-3]
      \arrow["{\Pi_{\varepsilon^{\text{ind}}_{\Pi_i \beta}}}", from=1-5, to=3-5]
      \arrow["{(\varepsilon^{\text{coind}}_\beta)^*}", from=3-1, to=1-3]
      \arrow["{(\Pi_i \varepsilon^{\text{coind}}_\beta)^*}"', from=3-1, to=3-3]
      \arrow[""{name=1, anchor=center, inner sep=0}, "{\Pi_i/\beta}"{description}, curve={height=30pt}, from=3-1, to=3-5]
      \arrow["\cong"{description}, draw=none, from=3-3, to=1-5]
      \arrow["{(\eta^{\text{coind}}_{\Pi_i\beta})^*}"', from=3-3, to=3-5]
      \arrow["\cong"{description}, draw=none, from=3-1, to=0]
      \arrow["\cong"{description, pos=0.7}, draw=none, from=1, to=3-3]
    \end{tikzcd}\]
  From this, we have that
  \begin{align}\label{eq:magical-iso}
    \Pi_{\varepsilon^{\text{ind}}_{\Pi_i \beta}} \Sigma_i f' & = \Pi_{\varepsilon^{\text{ind}}_{\Pi_i \beta}} \Sigma_i (\varepsilon^{\text{coind}}_\beta)^* f \\
                                                             & \cong (\eta^{\text{coind}}_{\Pi_i \beta})^* \Pi_i (\varepsilon^{\text{coind}}_\beta)^* f       \\
                                                             & \cong (\eta^{\text{coind}}_{\Pi_i \beta})^* (\Pi_i \varepsilon^{\text{coind}}_\beta)^* f       \\
                                                             & \cong \Pi_i f
  \end{align}
  Now we take \eqref{initial-exp-diagram} and pull back along this isomorphism to obtain
  \[\begin{tikzcd}[ampersand replacement=\&]
      \&\&\& \bullet \&\& {\Pi_i \zeta} \\
      \&\& \bullet \&\& \bullet \\
      \\
      {\Sigma_i \gamma} \&\& {\Sigma_i i^* \Pi_i \beta} \&\& {\Pi_i \beta}
      \arrow["{\varepsilon^{\text{ind}}_{\Pi_i \beta}}"', from=4-3, to=4-5]
      \arrow[from=2-3, to=4-3]
      \arrow["{\Sigma_i f'}"', from=4-1, to=4-3]
      \arrow["{\varepsilon^{\text{coind}}_{\Sigma_i f'}}"', from=2-3, to=4-1]
      \arrow["{\Pi_{\varepsilon^{\text{ind}}_{\Pi_i \beta}} \Sigma_i f'}"', from=2-5, to=4-5]
      \arrow["{(\Pi_{\varepsilon^{\text{ind}}_{\Pi_i \beta}} \Sigma_i f')^* \varepsilon^{\text{ind}}_{\Pi_i \beta}}"', from=2-3, to=2-5]
      \arrow["{\Pi_i f}", from=1-6, to=4-5]
      \arrow["\cong"{description}, from=1-4, to=2-3]
      \arrow[from=1-4, to=1-6]
      \arrow["\cong"{description}, from=1-6, to=2-5]
      \arrow["\lrcorner"{anchor=center, pos=0.125}, draw=none, from=1-4, to=2-5]
    \end{tikzcd}\]
  The composite of the two cartesian squares in this diagram is a cartesian square
  \[\begin{tikzcd}[ampersand replacement=\&]
      \bullet \&\& {\Pi_i \zeta} \\
      \\
      {\Sigma_i i^* \Pi_i \beta} \&\& {\Pi_i \beta}
      \arrow[from=1-1, to=1-3]
      \arrow["{\Pi_i f}", from=1-3, to=3-3]
      \arrow[from=1-1, to=3-1]
      \arrow["{\varepsilon^{\text{ind}}_{\Pi_i \beta}}"', from=3-1, to=3-3]
      \arrow["\lrcorner"{anchor=center, pos=0.125}, draw=none, from=1-1, to=3-3]
    \end{tikzcd}\]
  But now by \Cref{prop:induction-is-cartesian}, this square is isomorphic to
  \[\begin{tikzcd}[ampersand replacement=\&]
      {\Sigma_i i^* \Pi_i \zeta} \&\& {\Pi_i \zeta} \\
      \\
      {\Sigma_i i^* \Pi_i \beta} \&\& {\Pi_i \beta}
      \arrow["{\varepsilon^{\text{ind}}_{\Pi_i \zeta}}", from=1-1, to=1-3]
      \arrow["{\Pi_i f}", from=1-3, to=3-3]
      \arrow["{\Sigma_i i^* \Pi_i f}"', from=1-1, to=3-1]
      \arrow["{\varepsilon^{\text{ind}}_{\Pi_i \beta}}"', from=3-1, to=3-3]
    \end{tikzcd}\]
  This gives a new diagram
  \[\begin{tikzcd}
      && {\Sigma_i i^* \Pi_i \zeta} && {\Pi_i \zeta} \\
      \\
      {\Sigma_i \gamma} && {\Sigma_i i^* \Pi_i \beta} && {\Pi_i \beta}
      \arrow["{\varepsilon^{\text{ind}}_{\Pi_i \zeta}}", from=1-3, to=1-5]
      \arrow["\delta"', from=1-3, to=3-1]
      \arrow["{\Sigma_i i^* \Pi_i f}", from=1-3, to=3-3]
      \arrow["\lrcorner"{anchor=center, pos=0.125}, draw=none, from=1-3, to=3-5]
      \arrow["{\Pi_i f}", from=1-5, to=3-5]
      \arrow["{\Sigma_i f'}"', from=3-1, to=3-3]
      \arrow["{\varepsilon^{\text{ind}}_{\Pi_i \beta}}"', from=3-3, to=3-5]
    \end{tikzcd}\]
  whose encoded bispan (going around the top of the diagram) is isomorphic to the original from \eqref{initial-exp-diagram}. The diagonal morphism $\delta$ factors as
  \[\Sigma_i i^* \Pi_i \zeta \xrightarrow{\cong} \bullet \xrightarrow{\varepsilon^{\text{coind}}_{\Sigma_i f'}} \Sigma_i \gamma,\]
  and underlies a morphism $\Sigma_i i^* \Pi_i f \to \Sigma_i f'$ in $(\C/y)/\Sigma_i i^* \Pi_i \beta$. The isomorphism in this factorization is $(\varepsilon^{\text{ind}}_{\Pi_i \beta})^*$ applied to the isomorphism \eqref{eq:magical-iso} from $\Pi_i f$ to $\Pi_{\varepsilon^{\text{ind}}_{\Pi_i \beta}} \Sigma_i f'$, and so the composite $\delta$ is the adjunct of the isomorphism \eqref{eq:magical-iso} with respect to the adjunction $(\varepsilon^{\text{ind}}_{\Pi_i \beta})^* \dashv \Pi_{\varepsilon^{\text{ind}}_{\Pi_i \beta}}$.

  Also, by \Cref{lem:slogan}, $\delta$ is $\Sigma_i$ applied to a morphism $d : i^* \Pi_i f \to f'$ in $(\C/x)/i^*\Pi_i \beta$. This morphism $d$ fits in the diagram
  \[\begin{tikzcd}[ampersand replacement=\&]
      {i^* \Pi_i \zeta} \\
      \& \gamma \&\& {i^* \Pi_i \beta} \\
      \\
      \& \zeta \&\& \beta
      \arrow["{f'}", from=2-2, to=2-4]
      \arrow["{h'}"', from=2-2, to=4-2]
      \arrow["{\varepsilon^{\mathrm{coind}}_{\beta}}", from=2-4, to=4-4]
      \arrow["f"', from=4-2, to=4-4]
      \arrow["\lrcorner"{anchor=center, pos=0.125}, draw=none, from=2-2, to=4-4]
      \arrow["d"', from=1-1, to=2-2]
      \arrow["{i^* \Pi_i f}", from=1-1, to=2-4]
    \end{tikzcd}\]
  and, by tracing through the factorization above via the proof of \Cref{prop:Hoyer-Lemma}, we have $h' \circ d = \varepsilon^{\text{coind}}_\zeta$. We conclude that
  \[N_{\varepsilon^{\text{ind}}_{\Pi_i \beta}} T_{\Sigma_i f'} = T_{\Pi_i f} N_{\varepsilon^{\text{ind}}_{\Pi_i \zeta}} R_{\Sigma_i d}.\]
  Over all, this yields
  \begin{align*}
    t_\beta T_{\Sigma_i f} R_{\Sigma_i g}
     & = N_{\varepsilon^{\mathrm{ind}}_{\Pi_i \beta}} T_{\Sigma_i f'} R_{\Sigma_i (g \circ h')}                              \\
     & = T_{\Pi_i f} N_{\varepsilon^{\mathrm{ind}}_{\Pi_i \zeta}} R_{\Sigma_i d} R_{\Sigma_i (g \circ h')}                   \\
     & = T_{\Pi_i g} N_{\varepsilon^{\mathrm{ind}}_{\Pi_i \zeta}} R_{\Sigma_i (g \circ h' \circ d)}                          \\
     & = T_{\Pi_i g} N_{\varepsilon^{\mathrm{ind}}_{\Pi_i \zeta}} R_{\Sigma_i (g \circ \varepsilon^{\text{coind}}_{\zeta})},
  \end{align*}
  which is precisely the expression we found earlier for $T_{\Pi_i f} R_{\Pi_i g} t_\alpha$.
\end{proof}

With this in place, the definition of $\omega$ is straightforward.

\begin{defn}
  We define $\omega : \U{x}(\alpha, {-}) \to \U{y}(\Sigma_i \alpha, \Pi_i {-})$ to be the composite
  \[\U{x}(\alpha, e {-}) \xrightarrow{\U{\Sigma_i}} \U{y}(\Sigma_i \alpha, \Sigma_i e {-}) \xrightarrow{t_*} \U{y}(\Sigma_i \alpha, e \Pi_i {-}),\]
  where we recall that $\U{\Sigma_i} : \U{x} \to \U{y}$ is a well-defined functor by \Cref{cor:sigma-induced-on-P}.
\end{defn}

\subsubsection{Universality}

Now we must show that $\omega$ is the initial natural transformation from $\U{x}(\alpha,{-})$ to a functor precomposed with $\Pi_i$. Thus, let $f : \A{y} \to \Set$ be an arbitrary functor, and let $\tau : \U{x}(\alpha,{-}) \to \Pi_i^* F$ be an arbitrary natural transformation. We then aim to show that there exists a unique natural transformation $\sigma : \U{y}(\Sigma_i \alpha,{-}) \to F$ such that $\sigma \Pi_i \circ \omega = \tau$.

So, suppose we are given some equivalence class of bispans
\[[\Sigma_i \alpha \xleftarrow{h} a \to \bullet \to \beta] \in \U{y}(\Sigma_i \alpha, \beta).\]
By \Cref{lem:slogan}, the map $h : a \to \Sigma_i \alpha$ can also be (canonically) expressed as $\Sigma_i h : \Sigma_i \alpha h \to \Sigma_i \alpha$, where now $h : \alpha h \to \alpha$ is a morphism in $\C/x$. In other words, elements of $\U{y}(\Sigma_i \alpha, \beta)$ can be canonically expressed in the form $T_f N_g R_{\Sigma_i h} = [\Sigma_i \alpha \xleftarrow{\Sigma_i h} \Sigma_i a \xrightarrow{g} b \xrightarrow{f} c]$ for some morphism $h : a \to \alpha$ in $\C/x$. Now, if $\sigma$ is to exist as desired, it will need to satisfy a naturality square
\[\begin{tikzcd}[ampersand replacement=\&]
    {\U{y}(\Sigma_i \alpha, b)} \&\&\&\& {F(b)} \\
    \& \textcolor{rgb,255:red,92;green,92;blue,214}{N_g R_{\Sigma_i h}} \&\& \textcolor{rgb,255:red,92;green,92;blue,214}{\sigma_b(N_gR_{\Sigma_i h})} \\
    \\
    \& \textcolor{rgb,255:red,92;green,92;blue,214}{T_f N_g R_{\Sigma_i h}} \&\& \textcolor{rgb,255:red,92;green,92;blue,214}{\begin{array}{c}F(T_f)(\sigma_b(N_gR_{\Sigma_i h})) \\ = \sigma_\beta(T_f N_g R_{\Sigma_i h})\end{array}} \\
    {\U{y}(\Sigma_i \alpha, \beta)} \&\&\&\& {F(\beta)}
    \arrow["{\sigma_b}", from=1-1, to=1-5]
    \arrow["{(T_f)_*}"', from=1-1, to=5-1]
    \arrow["{\sigma_\beta}"', from=5-1, to=5-5]
    \arrow["{F(T_f)}", from=1-5, to=5-5]
    \arrow[draw={rgb,255:red,92;green,92;blue,214}, shorten <=13pt, shorten >=13pt, maps to, from=2-2, to=2-4]
    \arrow[draw={rgb,255:red,92;green,92;blue,214}, shorten <=7pt, shorten >=7pt, maps to, from=2-2, to=4-2]
    \arrow[draw={rgb,255:red,92;green,92;blue,214}, shorten <=9pt, shorten >=9pt, maps to, from=4-2, to=4-4]
    \arrow[draw={rgb,255:red,92;green,92;blue,214}, shorten <=6pt, shorten >=6pt, maps to, from=2-4, to=4-4]
  \end{tikzcd}\]
and thus it suffices to define $\sigma$ only on morphisms of the form $N_g R_{\Sigma_i h}$. We will next argue that the behaviour of $\sigma$ on such morphisms is completely forced. To do so, we make use of the following observation:

\begin{lem}
  \[R_{\eta^{\mathrm{coind}}_b} \circ \omega(N_{g^!} R_h) = N_g R_{\Sigma_i h}\]
  where $g^! : a \to i^*b$ is the adjunct of $g : \Sigma_i a \to b$.
\end{lem}

\begin{proof}
  By definition,
  \[R_{\eta^{\mathrm{coind}}_b} \circ \omega(N_{g^!} R_h) = R_{\eta^{\mathrm{coind}}_b} \circ t_{i^* b} \circ N_{\Sigma_i g^!} R_{\Sigma_i h} = R_{\eta^{\mathrm{coind}}_b} N_{\varepsilon^{\mathrm{ind}}_{\Pi_i i^* b}} R_{\Sigma_i \varepsilon^{\mathrm{coind}}_{i^* b}} N_{\Sigma_i g^!} R_{\Sigma_i h}.\]
  Now take a cartesian square
  \[\begin{tikzcd}[ampersand replacement=\&]
      c \& {i^* \Pi_i i^* b} \\
      a \& {i^*b}
      \arrow["{g^!}"', from=2-1, to=2-2]
      \arrow["{\varepsilon^{\text{coind}}_{i^* b}}", from=1-2, to=2-2]
      \arrow["p", from=1-1, to=1-2]
      \arrow["q"', from=1-1, to=2-1]
      \arrow["\lrcorner"{anchor=center, pos=0.125}, draw=none, from=1-1, to=2-2]
    \end{tikzcd}\]
  and apply $\Sigma_i$ (using \Cref{prop:induction-is-cartesian}) to obtain a cartesian square
  \[\begin{tikzcd}[ampersand replacement=\&]
      {\Sigma_i c} \& {\Sigma_i i^* \Pi_i i^* b} \\
      {\Sigma_i a} \& {\Sigma_i i^*b}
      \arrow["{\Sigma_i g^!}"', from=2-1, to=2-2]
      \arrow["{\Sigma_i \varepsilon^{\text{coind}}_{i^* b}}", from=1-2, to=2-2]
      \arrow["{\Sigma_i p}", from=1-1, to=1-2]
      \arrow["{\Sigma_i q}"', from=1-1, to=2-1]
      \arrow["\lrcorner"{anchor=center, pos=0.125}, draw=none, from=1-1, to=2-2]
    \end{tikzcd}\]
  showing that
  \[R_{\Sigma_i \varepsilon^{\mathrm{coind}}_{i^* b}} N_{\Sigma_i g^!} = N_{\Sigma_i p} R_{\Sigma_i q}.\]
  Thus,
  \begin{align*}
    R_{\eta^{\mathrm{coind}}_b} \circ \omega(N_{g^!} R_h)
     & = R_{\eta^{\mathrm{coind}}_b} N_{\varepsilon^{\mathrm{ind}}_{\Pi_i i^* b}} R_{\Sigma_i \varepsilon^{\mathrm{coind}}_{i^* b}} N_{\Sigma_i g^!} R_{\Sigma_i h} \\
     & = R_{\eta^{\mathrm{coind}}_b} N_{\varepsilon^{\mathrm{ind}}_{\Pi_i i^* b}} N_{\Sigma_i p} R_{\Sigma_i q} R_{\Sigma_i h}.
  \end{align*}
  Now, by \Cref{prop:induction-is-cartesian}, the naturality square
  \[\begin{tikzcd}[ampersand replacement=\&]
      {\Sigma_i i^* b} \& b \\
      {\Sigma_i i^* \Pi_i i^* b} \& {\Pi_i i^* b}
      \arrow["{\varepsilon^{\text{ind}}_{\Pi_i i^* b}}"', from=2-1, to=2-2]
      \arrow["{\eta^{\mathrm{coind}}_b}", from=1-2, to=2-2]
      \arrow["{\varepsilon^{\mathrm{ind}}_b}", from=1-1, to=1-2]
      \arrow["{\Sigma_i i^* \eta^{\mathrm{coind}}_b}"', from=1-1, to=2-1]
    \end{tikzcd}\]
  for $\varepsilon^{\mathrm{ind}}$ is cartesian. Thus,
  \[R_{\eta^{\mathrm{coind}}_b} N_{\varepsilon^{\mathrm{ind}}_{\Pi_i i^* b}} = N_{\varepsilon^{\mathrm{ind}}_b} R_{\Sigma_i i^* \eta^{\mathrm{coind}}_b},\]
  and so
  \begin{align*}
    R_{\eta^{\mathrm{coind}}_b} \circ \omega(N_{g^!} R_h)
     & = R_{\eta^{\mathrm{coind}}_b} N_{\varepsilon^{\mathrm{ind}}_{\Pi_i i^* b}} N_{\Sigma_i p} R_{\Sigma_i q} R_{\Sigma_i h}   \\
     & = N_{\varepsilon^{\mathrm{ind}}_b} R_{\Sigma_i i^* \eta^{\mathrm{coind}}_b} N_{\Sigma_i p} R_{\Sigma_i q} R_{\Sigma_i h}.
  \end{align*}
  Now, form a further pullback
  \[\begin{tikzcd}[ampersand replacement=\&]
      \bullet \& {\Sigma_i i^* b} \\
      {\Sigma_i c} \& {\Sigma_i i^* \Pi_i i^* b} \\
      {\Sigma_i a} \& {\Sigma_i i^*b}
      \arrow["{\Sigma_i g^!}"', from=3-1, to=3-2]
      \arrow["{\Sigma_i \varepsilon^{\text{coind}}_{i^* b}}", from=2-2, to=3-2]
      \arrow["{\Sigma_i p}", from=2-1, to=2-2]
      \arrow["{\Sigma_i q}"', from=2-1, to=3-1]
      \arrow["\lrcorner"{anchor=center, pos=0.125}, draw=none, from=2-1, to=3-2]
      \arrow["{\Sigma_i i^* \eta^{\mathrm{coind}}_b}", from=1-2, to=2-2]
      \arrow[from=1-1, to=1-2]
      \arrow[from=1-1, to=2-1]
      \arrow["\lrcorner"{anchor=center, pos=0.125}, draw=none, from=1-1, to=2-2]
    \end{tikzcd}\]
  By the unit-counit identities, the composite of the right column is
  \[\Sigma_i \varepsilon^{\mathrm{coind}}_{i^*b} \circ \Sigma_i i^* \eta^{\mathrm{coind}}_b = \Sigma_i (\varepsilon^{\mathrm{coind}}_{i^*b} \circ i^* \eta^{\mathrm{coind}}_b) = \Sigma_i \id_{i^*b} = \id_{\Sigma_i i^* b},\]
  and thus (up to isomorphism) this diagram is of the form
  \[\begin{tikzcd}[ampersand replacement=\&]
      {\Sigma_i a} \& {\Sigma_i i^* b} \\
      {\Sigma_i c} \& {\Sigma_i i^* \Pi_i i^* b} \\
      {\Sigma_i a} \& {\Sigma_i i^*b}
      \arrow["{\Sigma_i g^!}"', from=3-1, to=3-2]
      \arrow["{\Sigma_i \varepsilon^{\text{coind}}_{i^* b}}", from=2-2, to=3-2]
      \arrow["{\Sigma_i p}", from=2-1, to=2-2]
      \arrow["{\Sigma_i q}"', from=2-1, to=3-1]
      \arrow["\lrcorner"{anchor=center, pos=0.125}, draw=none, from=2-1, to=3-2]
      \arrow["{\Sigma_i i^* \eta^{\mathrm{coind}}_b}", from=1-2, to=2-2]
      \arrow["{\Sigma_i g^!}", from=1-1, to=1-2]
      \arrow["s"', from=1-1, to=2-1]
      \arrow["\lrcorner"{anchor=center, pos=0.125}, draw=none, from=1-1, to=2-2]
    \end{tikzcd}\]
  with $\Sigma_i q \circ s = \id_{\Sigma_i a}$. Now we have that $R_{\Sigma_i i^* \eta^{\mathrm{coind}}_b} N_{\Sigma_i p} = N_{\Sigma_i g^!} R_s$, and so
  \begin{align*}
    R_{\eta^{\mathrm{coind}}_b} \circ \omega(N_{g^!} R_h)
     & = N_{\varepsilon^{\mathrm{ind}}_b} N_{\Sigma_i g^!} R_s R_{\Sigma_i q} R_{\Sigma_i h}          \\
     & = N_{\varepsilon^{\mathrm{ind}}_b \circ \Sigma_i g^!} R_{\Sigma_i h \circ \Sigma_i q \circ s}.
  \end{align*}
  By definition of $\varepsilon^{\mathrm{ind}}$, $\varepsilon^{\mathrm{ind}}_b \circ \Sigma_i g^! = g$, and since $\Sigma_i q \circ s$ is an identity, we conclude that
  \[
    R_{\eta^{\mathrm{coind}}_b} \circ \omega(N_{g^!} R_h)
    = N_{\varepsilon^{\mathrm{ind}}_b \circ \Sigma_i g^!} R_{\Sigma_i h \circ (\Sigma_i q \circ s)}
    = N_g R_{\Sigma_i h},
  \]
  exactly as desired.
\end{proof}

Now, consider the following diagram.

\[\begin{tikzcd}[ampersand replacement=\&]
    {\U{x}(\alpha, i^* b)} \\
    {\U{y}(\Sigma_i a, \Pi_i i^* b)} \& {F \Pi_i i^* b} \\
    {\U{y}(\Sigma_a,b)} \& Fb
    \arrow["\omega", from=1-1, to=2-1]
    \arrow["{(R_{\eta^{\mathrm{coind}}_{b}})_*}", from=2-1, to=3-1]
    \arrow["{\sigma_{\Pi i^* b}}"', from=2-1, to=2-2]
    \arrow["{\sigma_b}"', from=3-1, to=3-2]
    \arrow["{F(R_{\eta^{\mathrm{coind}}_{b}})}", from=2-2, to=3-2]
    \arrow["{\tau_{i^* b}}", from=1-1, to=2-2]
  \end{tikzcd}\]

In order for $\sigma$ to be natural and satisfy $\sigma \Pi_i \circ \omega = \tau$, this diagram must commute. Now, starting with the element $N_{g!} R_h \in \U{x}(\alpha, i^*b)$, we chase
\[\begin{tikzcd}[ampersand replacement=\&]
    \textcolor{rgb,255:red,153;green,92;blue,214}{N_{g^!} R_h} \\
    \textcolor{rgb,255:red,153;green,92;blue,214}{\omega(N_{g^!} R_h)} \& \textcolor{rgb,255:red,153;green,92;blue,214}{\tau_{i^* b}(N_{g^!} R_h)} \\
    \textcolor{rgb,255:red,153;green,92;blue,214}{N_g R_{\Sigma_i h}} \& \textcolor{rgb,255:red,153;green,92;blue,214}{F(R_{\eta^{\mathrm{coind}}_b})(\tau_{i^* b}(N_{g^!} R_h))}
    \arrow[draw={rgb,255:red,153;green,92;blue,214}, maps to, from=1-1, to=2-1]
    \arrow[draw={rgb,255:red,153;green,92;blue,214}, maps to, from=2-1, to=3-1]
    \arrow[draw={rgb,255:red,153;green,92;blue,214}, maps to, from=1-1, to=2-2]
    \arrow[draw={rgb,255:red,153;green,92;blue,214}, maps to, from=2-2, to=3-2]
  \end{tikzcd}\]
and conclude that $\sigma_b(N_g R_{\Sigma_i h})$ must equal $F(R_{\eta^{\mathrm{coind}}_b})(\tau_{i^* b}(N_{g^!} R_h))$. This completely determines $\sigma$, which we can now write a complete formula for:

\begin{defn}
  $\sigma : \U{y}(\Sigma_i \alpha, {-}) \to F$ is given by the components
  \[\sigma_\beta([\Sigma_i \alpha \xleftarrow{\Sigma_i h} \Sigma_i a \xrightarrow{g} b \xrightarrow{f} \beta]) = F(T_f R_{\eta^{\mathrm{coind}}_b})(\tau_{i^* b}(N_{g^!} R_h)),\]
  where $g^! : a \to i^*b$ denotes the adjunct of $g : \Sigma_i a \to b$.
\end{defn}

We must show that $\sigma$ is natural and satisfies $\sigma \Pi_i \circ \omega = \tau$. Once that is done, the above argument implies that $\sigma$ is the unique natural transformation satisfying $\sigma \Pi_i \circ \omega = \tau$, so this will complete the proof.

First, to show naturality of $\sigma$, let $[\beta \xleftarrow{q} \zeta \xrightarrow{p} \beta']$ be an arbitrary morphism in $\A{\C/y}$. Then we wish to show that
\[\begin{tikzcd}
    {\U{y}(\Sigma_i \alpha, \beta)} && {F\beta} \\
    \\
    {\U{y}(\Sigma_i \alpha, \beta')} && {F\beta'}
    \arrow["{\sigma_\beta}", from=1-1, to=1-3]
    \arrow["{(T_p R_q)_*}"', from=1-1, to=3-1]
    \arrow["{F(T_p R_q)}", from=1-3, to=3-3]
    \arrow["{\sigma_{\beta'}}"', from=3-1, to=3-3]
  \end{tikzcd}\]
commutes. So, let $\varphi := [\Sigma_i \alpha \xleftarrow{\Sigma_i h} \Sigma_i a \xrightarrow{g} b \xrightarrow{f} \beta] \in \U{y}(\Sigma_i \alpha, \beta)$ be arbitrary. Going around the bottom-left, we have
\begin{multline*}
  \sigma_{\beta'}((T_p R_q)_*(\varphi))
  = \sigma_{\beta'}(T_p R_q T_f N_g R_{\Sigma_i h})
  = \sigma_{\beta'}(T_p T_{f'} R_{q'} N_g R_{\Sigma_i h})\\
  = \sigma_{\beta'}(T_{p \circ f'} N_{g'} R_{\Sigma_i h \circ q''})
  = \sigma_{\beta'}(T_{p \circ f'} N_{g'} R_{\Sigma_i (h \circ q'')})
  = F(T_{p \circ f'} R_{\eta^{\mathrm{coind}}_z})(\tau_{i^* z}(N_{g'^!} R_{h \circ q''}))\\
  = F(T_{p \circ f'} R_{\eta^{\mathrm{coind}}_z})(\tau_{i^* z}(N_{g'^!} R_{q''} R_h))
\end{multline*}
where $f', g', q', q''$ come from forming pullback squares
\[\begin{tikzcd}[ampersand replacement=\&]
    \&\& \bullet \& z \\
    \& {\Sigma_i a} \& b \&\& \zeta \\
    {\Sigma_i \alpha} \&\&\& \beta
    \arrow["{\Sigma_i h}"', from=2-2, to=3-1]
    \arrow["g", from=2-2, to=2-3]
    \arrow["f", from=2-3, to=3-4]
    \arrow["q"', from=2-5, to=3-4]
    \arrow["{f'}", from=1-4, to=2-5]
    \arrow["{q'}"', from=1-4, to=2-3]
    \arrow["{g'}", from=1-3, to=1-4]
    \arrow["{q''}"', from=1-3, to=2-2]
    \arrow["\lrcorner"{anchor=center, pos=0.125, rotate=-45}, draw=none, from=1-4, to=3-4]
    \arrow["\lrcorner"{anchor=center, pos=0.125, rotate=-45}, draw=none, from=1-3, to=2-3]
    \arrow["{(\text{A})}"{description}, draw=none, from=1-4, to=3-4]
  \end{tikzcd}\]

Next, going around the top-right of the square, we have
\[
  \begin{aligned}
    F(T_p R_q)(\sigma_{\beta}(\varphi))
     & = F(T_p R_q)(\sigma_{\beta}(T_f N_g R_{\Sigma_i h}))                                             & (\text{Definition of $\varphi$})               \\
     & = F(T_p R_q)(F(T_f R_{\eta^{\mathrm{coind}}_b})(\tau_{i^*b}(N_{g^!} R_h)))                       & (\text{Definition of $\sigma$})                \\
     & = F(T_p R_q T_f R_{\eta^{\mathrm{coind}}_b})(\tau_{i^*b}(N_{g^!} R_h))                           & (\text{Functoriality of $F$})                  \\
     & = F(T_p T_{f'} R_{q'} R_{\eta^{\mathrm{coind}}_b})(\tau_{i^*b}(N_{g^!} R_h))                     & (\text{(A) is cartesian})                      \\
     & = F(T_{p} \circ T_{f'} R_{\eta^{\mathrm{coind}}_b \circ q'})(\tau_{i^*b}(N_{g^!} R_h))           & (R_{a \circ b} = R_b \circ R_a)                \\
     & = F(T_{p \circ f'} R_{\eta^{\mathrm{coind}}_b \circ q'})(\tau_{i^*b}(N_{g^!} R_h))               & (T_{a \circ b} = T_a \circ T_b)                \\
     & = F(T_{p \circ f'} R_{\Pi_i i^* q' \circ \eta^{\mathrm{coind}}_z})(\tau_{i^*b}(N_{g^!} R_h))     & (\text{Naturality of $\eta^{\mathrm{coind}}$}) \\
     & = F(T_{p \circ f'} R_{\eta^{\mathrm{coind}}_z} R_{\Pi_i i^* q'})(\tau_{i^*b}(N_{g^!} R_h))       & (\text{$R_a \circ R_b = R_{b \circ a}$})       \\
     & = F(T_{p \circ f'} R_{\eta^{\mathrm{coind}}_z} \circ \Pi_i R_{i^* q'})(\tau_{i^*b}(N_{g^!} R_h)) & (\text{$\Pi_i R_a = R_{\Pi_i a}$})             \\
     & = F(T_{p \circ f'} R_{\eta^{\mathrm{coind}}_z})(F(\Pi_i R_{i^* q'})(\tau_{i^*b}(N_{g^!} R_h)))   & (\text{Functoriality of $F$})                  \\
     & = F(T_{p \circ f'} R_{\eta^{\mathrm{coind}}_z})(\tau_{i^*z}(R_{i^* q'} N_{g^!} R_h))             & (\text{Naturality of $\tau$})                  \\
  \end{aligned}
\]
Now the cartesian square
\[\begin{tikzcd}[ampersand replacement=\&]
    \bullet \& z \\
    {\Sigma_i a} \& b
    \arrow["{g'}", from=1-1, to=1-2]
    \arrow["{q'}", from=1-2, to=2-2]
    \arrow["g"', from=2-1, to=2-2]
    \arrow["{q''}"', from=1-1, to=2-1]
  \end{tikzcd}\]
above yields (by \Cref{cor:adjunct-is-cartesian}) a cartesian square
\[\begin{tikzcd}[ampersand replacement=\&]
    \bullet \& {i^* z} \\
    a \& {i^* b}
    \arrow["{g'^!}", from=1-1, to=1-2]
    \arrow["{i^* q'}", from=1-2, to=2-2]
    \arrow["{g^!}"', from=2-1, to=2-2]
    \arrow["{q''}"', from=1-1, to=2-1]
  \end{tikzcd}\]
Thus,
\[
  F(T_p R_q)(\sigma_{\beta}(\varphi))
  = F(T_{p \circ f'} R_{\eta^{\mathrm{coind}}_z})(\tau_{i^*z}(R_{i^* q'} N_{g^!} R_h))
  = F(T_{p \circ f'} R_{\eta^{\mathrm{coind}}_z})(\tau_{i^*z}(N_{g'^!} R_{q''} R_h)),
\]
as desired. This establishes the naturality of $\sigma$, and it only remains to be shown that $\sigma \Pi_i \circ \omega = \tau$.

So, let $\beta \in \A{x}$ and $[\alpha \xleftarrow{h} a \xrightarrow{g} b \xrightarrow{f} \beta] \in \U{x}(\alpha,\beta)$ be arbitrary. Then
\begin{align*}
  \omega_b(N_g R_h) = t_b N_{\Sigma_i h} R_{\Sigma_i h}
   & = N_{\varepsilon^{\mathrm{ind}}_{\Pi_i b}} R_{\Sigma_i
  \varepsilon^{\mathrm{coind}}_b} N_{\Sigma_i g} R_{\Sigma_ h}                                \\
   & = N_{\varepsilon^{\mathrm{ind}}_{\Pi_i b}} N_{\Sigma_i g'} R_{\Sigma_i e} R_{\Sigma_i h} \\
   & = N_{\varepsilon^{\mathrm{ind}}_{\Pi_i b} \circ \Sigma_i g'} R_{\Sigma_i (h \circ e)}
\end{align*}
where $g'$ and $e$ come from choosing a cartesian square
\[\begin{tikzcd}[ampersand replacement=\&]
    z \& {i^* \Pi_i b} \\
    a \& b
    \arrow["g"', from=2-1, to=2-2]
    \arrow["{\varepsilon^{\mathrm{coind}}_b}", from=1-2, to=2-2]
    \arrow["{g'}", from=1-1, to=1-2]
    \arrow["e"', from=1-1, to=2-1]
  \end{tikzcd}\]
Then
\begin{align*}
  \sigma_{\Pi_i b}(\omega_b(N_g R_h)) & = F(R_{\eta^{\mathrm{coind}}_{\Pi_i b}})(\tau_{i^* \Pi_i b}(N_{(\varepsilon^{\mathrm{ind}}_{\Pi_i b} \circ \Sigma_i g')^!} R_{h \circ e})) \\
                                      & = F(R_{\eta^{\mathrm{coind}}_{\Pi_i b}})(\tau_{i^* \Pi_i b}(N_{(\varepsilon^{\mathrm{ind}}_{\Pi_i b} \circ \Sigma_i g')^!} R_e R_h)).
\end{align*}
Where we recall that $(\varepsilon^{\mathrm{ind}}_{\Pi_i b} \circ \Sigma_i g')^!$ denotes the adjunct of $\varepsilon^{\mathrm{ind}}_{\Pi_i b} \circ \Sigma_i g'$ under the adjunction $\Sigma_i \dashv i^*$. Of course, by definition of $\varepsilon^{\mathrm{ind}}$, we have $(\varepsilon^{\mathrm{ind}}_{\Pi_i b} \circ \Sigma_i g')^! = g'$. Thus,
\[\sigma_{\Pi_i b}(\omega_b(N_g R_h)) = F(R_{\eta^{\mathrm{coind}}_{\Pi_i b}})(\tau_{i^* \Pi_i b}(N_{g'} R_e R_h)) = F(R_{\eta^{\mathrm{coind}}_{\Pi_i b}})(\tau_{i^* \Pi_i b}(R_{\varepsilon^{\mathrm{coind}}_b} N_g R_h)).\]
Then, by naturality of $\tau$, we obtain
\[\sigma_{\Pi_i b}(\omega_b(N_g R_h)) = F(R_{\eta^{\mathrm{coind}}_{\Pi_i b}})(F(\Pi_i R_{\varepsilon^{\mathrm{coind}}_b}) (\tau_b(N_g R_h))) = F(R_{\eta^{\mathrm{coind}}_{\Pi_i b}} R_{\Pi_i \varepsilon^{\mathrm{coind}}_b})(\tau_b(N_g R_h)).\]
By the unit-counit identities, $\Pi_i \varepsilon^{\mathrm{coind}}_b \circ \eta^{\mathrm{coind}}_{\Pi_i b} = \id_{\Pi_b}$. Thus, we have
\[\sigma_{\Pi_i b}(\omega_b(N_g R_h)) = \tau_b(N_g R_h).\]
Finally, by naturality of $\tau$, $\omega$, and $\sigma$, we have
\begin{align*}
  \sigma_{\Pi_i \beta}(\omega_\beta(T_f N_g R_h))
   & = \sigma_{\Pi_i \beta}((\Pi_i T_f)_* \omega_b(N_g R_h) \\
   & = F(\Pi_i T_f)(\sigma_{\Pi_i b}(\omega_b(N_g R_h)))    \\
   & = F(\Pi_i T_f)(\tau_b(N_gR_h))                         \\
   & = \tau_{\beta}(T_f N_g R_h).
\end{align*}
Since $\beta$ and $[\alpha \xleftarrow{h} a \xrightarrow{g} b \xrightarrow{f} \beta]$ were arbitrary, we conclude that $\sigma \Pi_i \circ \omega = \tau$. This completes the proof.

\subsection{Naturality in \texorpdfstring{$\alpha$}{α}}

We have now exhibited $\U{y}(\Sigma_i \alpha, \Pi_i {-})$ as the left Kan extension of $\U{x}(\alpha,{-})$ along $\Pi_i$, i.e. we have demonstrated that \eqref{eq:main-theorem-semi-square} commutes up to pointwise isomorphism. To show that this isomorphism is natural in $\alpha \in \U{x}^\op$, we need only show that, for any morphism $\varphi : \alpha \to \alpha'$ in $\U{x}$, the natural isomorphisms $\omega$ fit in a commutative square

\[\begin{tikzcd}[ampersand replacement=\&]
    {\U{x}(\alpha',{-})} \&\& {\U{y}(\Sigma_i \alpha, \Pi_i {-})} \\
    \\
    {\U{x}(\alpha,{-})} \&\& {\U{y}(\Sigma_i \alpha, \Pi_i {-})}
    \arrow["\omega", from=1-1, to=1-3]
    \arrow["{\varphi^*}"', from=1-1, to=3-1]
    \arrow["\omega"', from=3-1, to=3-3]
    \arrow["{(\Sigma_i \varphi)^*}", from=1-3, to=3-3]
  \end{tikzcd}\]

However, this is straightforward: by definition, the square above factors as

\[\begin{tikzcd}[ampersand replacement=\&]
    {\U{x}(\alpha',{-})} \& {\U{y}(\Sigma_i\alpha',\Sigma_i{-})} \& {\U{y}(\Sigma_i \alpha, \Pi_i {-})} \\
    \\
    {\U{x}(\alpha,{-})} \& {\U{y}(\Sigma_i \alpha, \Sigma_i {-})} \& {\U{y}(\Sigma_i \alpha, \Pi_i {-})}
    \arrow["{\varphi^*}"', from=1-1, to=3-1]
    \arrow["{(\Sigma_i \varphi)^*}", from=1-3, to=3-3]
    \arrow["{\Sigma_i}", from=1-1, to=1-2]
    \arrow["{t_*}", from=1-2, to=1-3]
    \arrow["{(\Sigma_i \varphi)^*}", from=1-2, to=3-2]
    \arrow["{t_*}"', from=3-2, to=3-3]
    \arrow["{\Sigma_i}"', from=3-1, to=3-2]
  \end{tikzcd}\]

Both sub-squares commute by direct computation.

This completes the proof that \eqref{eq:main-theorem-semi-square} commutes up to natural isomorphism.

\subsection{Separability and Preservation of Mackey Structure}

We have completed the proof of the first half of the theorem, and would now like to show that left Kan extension of semi-Mackey functors along $\A{\Pi_i}$ sends Mackey functors to Mackey functors. This requires some additional assumption on $i$, as shown by the following example.

\begin{exam}
  Consider the morphism $i : \varnothing \to G/e$ in $G{-}\set$. $G{-}\set/\varnothing$ is the terminal category, so $\SMack{G{-}\set/\varnothing}$ is equivalent to the terminal category -- each semi-Mackey functor indexed by $G{-}\set/\varnothing$ sends $\id_\varnothing$ to a singleton. In particular, every semi-Mackey functor indexed by $G{-}\set/\varnothing$ is Mackey.

  On the other hand, $G{-}\set/(G/e)$ is equivalent to $\set$, so $\SMack{G{-}\set/(G/e)}$ is equivalent to the category of commutative monoids.

  So, now consider the $(G{-}\set/\varnothing)$-Mackey functor $F$ represented by $\id_\varnothing$. Then $\Lan_{\Pi_i} F$ is represented by $\Pi_i \id_{\varnothing} = \id_{G/e}$. Under the equivalence $G{-}\set/(G/e) \simeq \set$, $\id_{G/e}$ corresponds to a singleton set, so $\Lan_{\Pi_i} F$ corresponds to the $\set$-semi-Mackey functor represented by a singleton, i.e. the free commutative monoid on a singleton. In other words, we have
  \[\begin{tikzcd}
      {\Mack{G{-}\set/\varnothing}} && {\SMack{G{-}\set/(G/e)}} & \CMon \\
      \textcolor{rgb,255:red,92;green,92;blue,214}{F} &&& \textcolor{rgb,255:red,92;green,92;blue,214}{\mathbb{N}}
      \arrow["{\Lan_{\Pi_i}}", from=1-1, to=1-3]
      \arrow["\cong"{marking, allow upside down}, draw=none, from=1-3, to=1-4]
      \arrow[color={rgb,255:red,92;green,92;blue,214}, maps to, from=2-1, to=2-4]
    \end{tikzcd}\]
  Since $\mathbb{N}$ is not an abelian group, we conclude that $\Lan_{\Pi_i}$ does not preserve Mackey functors in this case.
\end{exam}

It turns out that the correct assumption to place on $i$ is that it is an epimorphism. We can see that this should be the case by the following heuristic:

\begin{slogan}
  $\Lan_{\Pi_i}$ preserving Mackey functors is a categorification of $A(N_i)$ preserving additively invertible elements for all semi-Tambara functors $A$.
\end{slogan}

Indeed, we see that $A(N_i)(0) = 1$ when the domain of $i$ is an initial object. $1$ is not additively invertible at every level of an arbitrary semi-Tambara functor over an arbitrary index $\C$, so we cannot expect $\Lan_{\Pi_i}$ to preserve Mackey functors when the domain of $i$ is an initial object. On the other hand, $A(N_i)(0) = 0$ for all semi-Tambara functors $A$ whenever $i$ is an epimorphism. Thus, we should expect that $\Lan_{\Pi_i}$ preserves Mackey functors whenever $i$ is an epimorphism, and $0$ is always additively invertible.

To leverage the fact that $A(N_i)(0) = 0$ to prove that $A(N_i)$ preserves additively invertible elements, we would like to make use of a formula akin to Mazur's \cite[\S 1.4.1]{Mazur} for $G$-Tambara functors:
\begin{equation}
  \label{eq:mazur-style-formula}
  A(N_f)(a + b) = A(N_f)(a) + A(N_f)(b) + (\text{other terms})
\end{equation}
We don't actually need something quite this strong; it would suffice to have a formula like
\begin{equation}
  \label{eq:half-mazur-formula}
  A(N_f)(a+b) = A(N_f)(a) + (\text{other terms}).
\end{equation}
Then, whenever $a$ is additively invertible and $A(N_f)(0) = 0$, we would have
\[0 = A(N_f)(0) = A(N_f)(a+(-a)) = A(N_f)(a) + (\text{other terms}),\]
and so $A(N_f)(a)$ is additively invertible. Thinking about what \eqref{eq:half-mazur-formula} would say about a representable semi-Tambara functor, we find exactly that we need $\C$ to be separable.

\begin{prop}
  Let $\C$ be an index. The following are equivalent:
  \begin{enumerate}[label=(\roman*), ref=(\roman*)]
    \item $\C$ is separable;
    \item For every $\C$-semi-Tambara functor $A$, every morphism $i : x \to y$ in $\C_m$, and all elements $a,b \in A(x)$, $A(N_i)(a)$ is a summand of $A(N_i)(a+b)$;
    \item For every representable $\C$-semi-Tambara functor $A$, every morphism $i : x \to y$ in $\C_m$, every morphism $i : x \to y$ in $\C_m$, and all elements $a,b \in A(x)$, $A(N_i)(a)$ is a summand of $A(N_i)(a+b)$.
  \end{enumerate}
\end{prop}
\begin{proof}
  It's clear that (ii) implies (iii), so we will prove that (i) implies (ii) and that (iii) implies (i).

  First, suppose $\C$ is separable. Let $A \in \STamb{\C}$, $i : x \to y$ in $\C_m$, and $a,b \in A(x)$ be arbitrary. The element $A(N_i)(a+b)$ is produced by
  \[\begin{tikzcd}
      {A(x) \times A(x)} & {A(x \amalg x)} & {A(x)} & {A(y)} \\
      \textcolor{rgb,255:red,92;green,92;blue,214}{(a,b)} && \textcolor{rgb,255:red,92;green,92;blue,214}{a+b} & \textcolor{rgb,255:red,92;green,92;blue,214}{A(N_i)(a+b)}
      \arrow["\cong"{description}, draw=none, from=1-1, to=1-2]
      \arrow["{A(T_\nabla)}", from=1-2, to=1-3]
      \arrow["{A(N_i)}", from=1-3, to=1-4]
      \arrow[color={rgb,255:red,92;green,92;blue,214}, maps to, from=2-1, to=2-3]
      \arrow[color={rgb,255:red,92;green,92;blue,214}, maps to, from=2-3, to=2-4]
    \end{tikzcd}\]
  Now we compute $N_i \circ T_\nabla$ by forming a distributor diagram step-by-step. First, we produce the dependend product of $\nabla$ along $i$, and recall by separability of $\C$ that this decomposes as $\Pi_i \nabla = \id_y \amalg \gamma$ for some $\gamma \in \C/y$. Thus, we have
  \[\begin{tikzcd}
      && {y \amalg c} \\
      {x \amalg x} & x & y
      \arrow["{(\id_y, \gamma)}", from=1-3, to=2-3]
      \arrow["\nabla"', from=2-1, to=2-2]
      \arrow["i"', from=2-2, to=2-3]
    \end{tikzcd}\]
  Now we pull back along $i$ (recalling that $i^*$ commutes with coproducts) and attach the counit of the adjunction $i^* \dashv \Pi_i$, to get
  \[\begin{tikzcd}
      & {x \amalg c'} & {y \amalg c} \\
      {x \amalg x} & x & y
      \arrow["{i \amalg \zeta}", from=1-2, to=1-3]
      \arrow[from=1-2, to=2-1]
      \arrow["{(\id_x, \gamma')}", from=1-2, to=2-2]
      \arrow["\lrcorner"{anchor=center, pos=0.125}, draw=none, from=1-2, to=2-3]
      \arrow["{(\id_y, \gamma)}", from=1-3, to=2-3]
      \arrow["\nabla"', from=2-1, to=2-2]
      \arrow["i"', from=2-2, to=2-3]
    \end{tikzcd}\]
  By naturality of the counit, the diagonal morphism is equal to $\id_x \amalg \gamma$. Now we have that $N_f T_\nabla = T_{(\id_y, \gamma)} N_{i \amalg \zeta} R_{\id_x \amalg \gamma}$. This gives us a commutative diagram
  \[\begin{tikzcd}
      {A(x) \times A(x)} & {A(x) \times A(c')} & {A(y) \times A(c)} & {A(y) \times A(y)} \\
      {A(x \amalg x)} & {A(x \amalg c')} & {A(y \amalg c)} & {A(y \amalg y)} & {A(y)}
      \arrow["{\id \times A(R_{\gamma'})}", from=1-1, to=1-2]
      \arrow["\cong"', from=1-1, to=2-1]
      \arrow["{A(N_i) \times A(N_\zeta)}", shift left, from=1-2, to=1-3]
      \arrow["\cong"', from=1-2, to=2-2]
      \arrow["{\id \times A(T_\gamma)}", from=1-3, to=1-4]
      \arrow["\cong"', from=1-3, to=2-3]
      \arrow["\cong"', from=1-4, to=2-4]
      \arrow["{+}", from=1-4, to=2-5]
      \arrow["{A(R_{\id_x \amalg \gamma'})}"', from=2-1, to=2-2]
      \arrow["{A(N_{i \amalg \zeta})}"', from=2-2, to=2-3]
      \arrow["{A(T_{\id_y \amalg \gamma})}"', from=2-3, to=2-4]
      \arrow["{A(T_\nabla)}"', from=2-4, to=2-5]
    \end{tikzcd}\]
  Going along the top, we see that $(a,b)$ is sent to $A(N_i)(a) + \text{(stuff)}$. On the other hand, this is a factorization of $N_f T_\nabla$, so also $(a,b)$ is sent to $A(N_i)(a+b)$. Thus, $A(N_i)(a)$ is a summand of $A(N_i)(a+b)$.

  Next, suppose (iii) holds, and let $i : x \to y$ in $\C_m$ be arbitrary. Let $A$ be the $\C$-semi-Tambara functor represented by $x$. In $A(x)$, we have $\id_x + \id_x = T_\nabla R_\nabla$. We have by assumption that $A(N_i)(\id_x) = N_i$ is a summand of $A(N_i)(\id_x + \id_x) = N_i T_\nabla R_\nabla$. Now we form a distributor diagram to compute $N_i T_\nabla$:
  \[\begin{tikzcd}
      & \bullet & \bullet \\
      {x \amalg x} & x & y
      \arrow["{(\Pi_i \nabla)^* i}", from=1-2, to=1-3]
      \arrow[from=1-2, to=2-1]
      \arrow["{i^* \Pi_i \nabla}", from=1-2, to=2-2]
      \arrow["\lrcorner"{anchor=center, pos=0.125}, draw=none, from=1-2, to=2-3]
      \arrow["{\Pi_i \nabla}", from=1-3, to=2-3]
      \arrow["\nabla"', from=2-1, to=2-2]
      \arrow["i"', from=2-2, to=2-3]
    \end{tikzcd}\]
  and note by commutativity of the triangle therein that
  \[N_i T_\nabla R_\nabla = T_{\Pi_i \nabla} N_{(\Pi_i \nabla)^* i} R_{i^* \Pi_i \nabla}.\]
  Now the claim that $N_i$ is a summand of this morphism says that there is a decomposition $(\Pi_i \nabla)^* i$ (up to isomorphism) as a coproduct $i \amalg \zeta$ for some morphism $\zeta$, in such a way that the class
  \[[x \xleftarrow{i^* \Pi_i \nabla} \bullet \xrightarrow{(\Pi_i \nabla)^*i} \bullet \xrightarrow{\Pi_i \nabla} y]\]
  equals
  \[[x \xleftarrow{(\id_x, \gamma')} x \amalg c' \xrightarrow{i \amalg \zeta} y \amalg c' \xrightarrow{(\id_y, \gamma)} y]\]
  for some bispan $x \xleftarrow{\gamma'} c' \xrightarrow{\zeta} c \xrightarrow{\gamma} y$. In particular, this demonstrates that $j_i : \id_y \to \Pi_i \nabla$ is complemented. Since $i$ was arbitrary, we conclude that $\C$ is separable.
\end{proof}

So, we will now prove that $\Lan_{\Pi_i}$ preserves Mackey functors whenever $\C$ is separable and $i$ is an epimorphism.

\begin{prop}
  Let $\C$ be an index, and let $F \in \STamb{\C}$ be arbitrary. Let $i : x \to y$ be an epimorphism lying in $\C_m$. For each object $\alpha \in \C/x$, the unit map $\eta : F(\alpha) \to (\Lan_{\Pi_i} F)(\Pi_i \alpha)$ sends $0$ to $0$.
\end{prop}
\begin{proof}
  We recall that the element $0 \in F(\alpha)$ is the image under $F(T_!)$ of the unique element of $F(\varnothing)$, where $! : \varnothing \to \alpha$ is a morphism from an initial object. By naturality of $\eta$, we have a commutative square
  \[\begin{tikzcd}
      {F(\varnothing)} && {(\Lan_{\Pi_i} F)(\Pi_i \varnothing)} \\
      \\
      {F(\alpha)} && {(\Lan_{\Pi_i} F)(\Pi_i \alpha)}
      \arrow["\eta", from=1-1, to=1-3]
      \arrow["{F(T_!)}"', from=1-1, to=3-1]
      \arrow["{(\Lan_{\Pi_i} F)(T_{\Pi_i !})}", from=1-3, to=3-3]
      \arrow["\eta"', from=3-1, to=3-3]
    \end{tikzcd}\]
  so we may assume $\alpha = \varnothing$. Now \Cref{lem:dep-prod-of-initial} tells us that $\Pi_i \varnothing = \varnothing$ (this is where we use that $i$ is an epimorphism), so $(\Lan_{\Pi_i} F)(\Pi_i \varnothing) = (\Lan_{\Pi_i} F)(\varnothing)$ is the singleton $\{0\}$, and we are done.
\end{proof}

\begin{prop}
  Let $\C$ be a separable index, and let $F \in \STamb{\C}$ be arbitrary. Let $i : x \to y$ be an epimorphism lying in $\C_m$. For each object $\alpha \in \C/x$, the unit map $\eta : F(\alpha) \to (\Lan_{\Pi_i} F)(\Pi_i \alpha)$ preserves invertible elements.
\end{prop}
\begin{proof}
  Let $s \in F(\alpha)$ be an arbitrary invertible element, with inverse $s'$. By separability, the inclusion $j : \Pi_i \alpha \to \Pi_i(\alpha \amalg \alpha)$ is complemented; let $k : Q \to \Pi_i(\alpha \amalg \alpha)$ be a complement. Now we have
  \[\begin{tikzcd}
      {F(\alpha \amalg \alpha)} && {(\Lan_{\Pi_i}F)(\Pi_i(\alpha \amalg \alpha))} && {(\Lan_{\Pi_i}F)(\Pi_i \alpha \amalg Q)} \\
      &&&& {(\Lan_{\Pi_i} F)(\Pi_i \alpha \amalg \Pi_i \alpha)} \\
      {F(\alpha)} && {(\Lan_{\Pi_i}F)(\Pi_i \alpha)}
      \arrow["\eta", from=1-1, to=1-3]
      \arrow["{F(T_\nabla)}"', from=1-1, to=3-1]
      \arrow["{(\Lan_{\Pi_i}F)(T_{(j,k)^{-1}})}", from=1-3, to=1-5]
      \arrow["{(\Lan_{\Pi_i} F)(T_{\Pi_i \nabla})}"', from=1-3, to=3-3]
      \arrow["{(\Lan_{\Pi_i} F)(\id \amalg T_{\Pi_i \nabla \circ k})}", from=1-5, to=2-5]
      \arrow["{(\Lan_{\Pi_i} F)(T_\nabla)}", from=2-5, to=3-3]
      \arrow["\eta"', from=3-1, to=3-3]
    \end{tikzcd}\]
  The left-hand square commutes by naturality of $\eta$, and the right-hand quadrilateral commutes because
  \[\begin{tikzcd}
      {\Pi_i(\alpha \amalg \alpha)} & {\Pi_i \alpha \amalg Q} \\
      && {\Pi_i \alpha \amalg \Pi_i (\alpha \amalg \alpha)} \\
      {\Pi_i \alpha} && {\Pi_i \alpha \amalg \Pi_i \alpha}
      \arrow["{\Pi_i \nabla}"', from=1-1, to=3-1]
      \arrow["{(j,k)}"', from=1-2, to=1-1]
      \arrow["{(\id, k)}", from=1-2, to=2-3]
      \arrow["{(\id, \Pi_i \nabla \circ k)}", from=1-2, to=3-1]
      \arrow["{(\id, \Pi_i \nabla)}", from=2-3, to=3-3]
      \arrow["\nabla", from=3-3, to=3-1]
    \end{tikzcd}\]
  commutes. Overall, we conclude that
  \[0 = \eta(0) = \eta(s+s') = \eta(s) + \text{(other terms)},\]
  and thus $\eta(s)$ is invertible.
\end{proof}

\begin{prop}\label{prop:lan-pi-preserves-mackey}
  Let $\C$ be a separable bi-incomplete index, and let $i : x \to y$ be an epimorphism lying in $\C_m$. Then $\Lan_{\A{\Pi_i}} : \SMack{\C/x} \to \SMack{\C/y}$ sends Mackey functors to Mackey functors.
\end{prop}
\begin{proof}
  Let $F \in \Mack{\C/x}$ and $\beta \in \C/y$ be arbitrary. Take an arbitrary element $t \in (\Lan_{\Pi_i} F)(\beta)$. Then there is some $\alpha \in \C/x$, some morphism $\varphi : \Pi_i \alpha \to \beta$ in $\A{\C/y}$, and some element $s \in F(\alpha)$ such that $t$ is the image of $s$ under
  \[F(\alpha) \xrightarrow{\eta} (\Lan_{\Pi_i} F)(\Pi_i \alpha) \xrightarrow{(\Lan_{\Pi_i F})(\varphi)} (\Lan_{\Pi_i F})(\beta).\]
  We know that $\eta$ sends invertible elements to invertible elements, and $(\Lan_{\Pi_i F})(\varphi)$ is a homomorphism of commutative monoids. Since $F$ is a Mackey functor, $s$ is invertible, so we are done.
\end{proof}

\section{Conclusions}
\label{section:Consequences}

We now summarize the results of the preceeding section and their consequences. In the case that $i : x \to y$ is an epimorphism in $\C_m$ for $\C$ a locally essentially small, separable index, the two commutative squares established in \Cref{actual-main-theorem} give a commutative cube

\[\begin{tikzcd}
    {\Tamb{\C/x}} && {\Tamb{\C/y}} \\
    & {\STamb{\C/x}} && {\STamb{\C/y}} \\
    {\Mack{\C/x}} && {\Mack{\C/y}} \\
    & {\SMack{\C/x}} && {\SMack{\C/y}}
    \arrow["{\mathcal{N}_i}", from=1-1, to=1-3]
    \arrow[hook, from=1-1, to=2-2]
    \arrow[from=1-1, to=3-1]
    \arrow[hook, from=1-3, to=2-4]
    \arrow[from=1-3, to=3-3]
    \arrow[from=2-4, to=4-4]
    \arrow["{\mathcal{N}_i}"'{pos=0.7}, from=3-1, to=3-3]
    \arrow[hook, from=3-1, to=4-2]
    \arrow[hook, from=3-3, to=4-4]
    \arrow["{\mathcal{N}_i}"', from=4-2, to=4-4]
    \arrow["{\mathcal{N}_i}"{pos=0.3}, crossing over, from=2-2, to=2-4]
    \arrow[crossing over, from=2-2, to=4-2]
  \end{tikzcd}\]

For emphasis, we will spell out what this tells us. Under these hypotheses, two things are true:

\begin{enumerate}
  \item The restriction functor $\mathcal{R}_i : \Tamb{\C/y} \to \Tamb{\C/x}$ has a left adjoint $\mathcal{N}_i : \Tamb{\C/x} \to \Tamb{\C/y}$.
  \item Given a Tambara functor $A \in \Tamb{\C/x}$, we can compute the underlying Mackey functor of $\mathcal{N}_i A$ by $\Lan_{\A{\Pi_i}} e^* A$, where $e^* A$ is the underlying Mackey functor of $A$.
\end{enumerate}

In particular, this result holds for naive motivic Tambara functors. The only part of the assumptions which we have not yet checked is that $\fet$ is locally essentially small, which we will now show.

\begin{prop}
  For any scheme $S$, $\fet/S$ is essentially small.
\end{prop}
\begin{proof}
  In fact, the category of schemes finite over $S$ is already essentially small. Recall that a morphism $f : X \to S$ of schemes is \emph{finite} if and only if there exists an open cover $\{U_i\}_{i \in I}$ of $S$ by affine subschemes such that for each $i$, $f^{-1}(U_i)$ is affine the ring homomorphism $\mathcal{O}_S(U_i) \to \mathcal{O}_X(f^{-1}(U_i))$ is finite.

  So, a finite morphism $X \to S$ is determined up to isomorphism by the data of an affine open cover $\{U_i\}_{i \in I}$ of $S$, a set of (isomorphism classes of) finite ring homomorphisms $\{\mathcal{O}_S(U_i) \to A_i\}_{i \in I}$, and gluing data for the set of maps $\{\Spec(A_i) \to U_i \hookrightarrow S\}_{i \in I}$.

  There is a set of affine open covers of $S$. For a fixed affine open cover $\{U_i\}_{i \in I}$ of $S$ and a fixed $i \in I$, there is a set's worth of isomorphism classes of finite ring homomorphisms $\mathcal{O}_S(U_i) \to A_i$ (because there is a set's worth of isomorphism classes of finitely generated $\mathcal{O}_S(U_i)$-modules, and a set of $\mathcal{O}_S(U_i)$-algebra structures on any given finitely generated $\mathcal{O}_S(U_i)$-module). Finally, for a fixed affine open cover $\{U_i\}_{i \in I}$ of $S$ and a fixed set of finite ring homomorphisms $\{\mathcal{O}_S(U_i) \to A_i\}$, there is a set of possible gluing data for the maps $\{\Spec(A_i) \to U_i \hookrightarrow S\}_{i \in I}$. We conclude that the collection of isomorphism classes of finite $S$-schemes forms a set, and thus the category of finite $S$-schemes is essentially small.
\end{proof}

\begin{cor}\label{cor:apply-to-motivic}
  For any finite {\'e}tale cover $f : X \to S$ of schemes, the restriction functor
  \[\mathcal{R}_f : \Tamb{\fet/S} \to \Tamb{\fet/X}\]
  between categories of naive motivic Tambara functors given by
  \[(\mathcal{R}_f A)(g) = A(f \circ g)\]
  has a left adjoint
  \[\mathcal{N}_f : \Tamb{\fet/X} \to \Tamb{\fet/S}\]
  such that the value of $\mathcal{N}_f A$ on an object $h \in \fet/S$ can be computed (as an abelian group) with the coend
  \[(\mathcal{N}_f A)(h) = \int^{g \in \A{\fet/S}} \A{\fet/S}(\Pi_i g, h) \times A(g).\]
\end{cor}

\Cref{cor:apply-to-motivic} was obtained by applying \Cref{actual-main-theorem} to the index $(\fet/S, \fet/S, \fet/S)$. By instead applying the theorem to the index $(\fet/S, \mathcal{O}^{\mathrm{triv}},\mathcal{O}^{\mathrm{triv}})$, we obtain an analogous theorem for motivic Green functors. The upshot is that the functor $\mathcal{N}_f$ (for either motivic Green functors or motivic Tambara functors) can be computed directly as a coend over $\A{\fet/S}$. This substantially simplifies the comptuation of $\mathcal{N}_f$, just because spans are much simpler than bispans.

Future work will leverage this result to compute Rost norms in Grothendieck-Witt rings, as well as extending this result to the kinds of motivic Tambara functors defined in \cite{Bachmann-MTF}.

\newpage

\appendix

\section{Technical Lemmas}
\label{appendix:lcc-lemmas}

\subsection[The Proof of \Cref*{prop:disjoint-coproducts-slice-products}]{The Proof of \Cref{prop:disjoint-coproducts-slice-products}}

\begin{prop*}
  Let $\C$ be a category which is locally cartesian closed and cocartesian with disjoint coproducts. Then any coproduct diagram $x \xrightarrow{i} x \amalg y \xleftarrow{j} y$ induces an equivalence of categories
  \[\C/(x \amalg y) \xrightarrow{(i^*, j^*)} \C/x \times \C/y.\]
\end{prop*}
\begin{proof}
  The desired quasi-inverse functor is given by the composite
  \[\C/x \times \C/y \xrightarrow{\Sigma_i \times \Sigma_j} \C/(x \amalg y) \times \C/(x \amalg y) \xrightarrow{\amalg} \C/(x \amalg y),\]
  where the second functor is the categorical coproduct described by \Cref{slice-of-cocartesian-is-cocartesian}. We will show that both composites are naturally isomorphic to the identity. First, let $\alpha \in \C/(x \amalg y)$ be arbitrary. Then we have a natural isomorphism
  \[\Sigma_i i^* \alpha \amalg \Sigma_j j^* \alpha \cong (i \times \alpha) \amalg (j \times \alpha) \cong (i \amalg j) \times \alpha\]
  in $\C/(x \amalg y)$, where the first isomorphism comes from \Cref{prop:product-is-LR} and the second isomorphism comes from the fact that ${-} \times \alpha$ is a left adjoint (since $\C/(x \amalg y)$ is cartesian closed), and thus commutes with coproducts. However, we also have that $i \amalg j \cong \id_{x \amalg y}$ (this follows simply from the characterization of coproducts in a slice category from \Cref{slice-of-cocartesian-is-cocartesian}). Since $\id_{x \amalg y}$ is terminal in $\C/(x \amalg y)$, we conclude that
  \[\Sigma_i i^* \alpha \amalg \Sigma_j j^* \alpha \cong \alpha\]
  naturally in $\alpha$. Thus, the composite
  \[\C/(x \amalg y) \to \C/x \times \C/y \to \C/(x \amalg y)\]
  is naturally isomorphic to the identity.

  Next, let $\beta \in \C/x$ and $\gamma \in \C/y$ be arbitrary. We have
  \[i^* (\Sigma_i \beta \amalg \Sigma_j \gamma) \cong i^* \Sigma_i \beta \amalg i^* \Sigma_j \gamma,\]
  so we want to show that $i^* \Sigma_j \gamma \cong \varnothing$ and $i^* \Sigma_i \beta \cong \beta$ naturally in $\beta$. Now notice that $\Sigma_j \gamma$ naturally lives as an object of $(\C/(x \amalg y))/j$, and thus, $i^* \Sigma_j \gamma$ lives as an object of $(\C/x)/i^*j$. But, since finite coproducts in $\C$ are disjoint, $i^*j$ is the initial object of $\C/x$. By \Cref{prop:lcc-slice-over-empty}, we conclude that $i^* \Sigma_j \gamma \cong \varnothing$. Next, we wish to show that $i^* \Sigma_i \beta \cong \beta$ naturally in $\beta$. By hypothesis of disjointness of coproducts, $i$ is a monomorphism, so
  \[\begin{tikzcd}[ampersand replacement=\&]
      x \& x \\
      x \& {x \amalg y}
      \arrow["{\id_x}", from=1-1, to=1-2]
      \arrow["{\id_x}"', from=1-1, to=2-1]
      \arrow["i"', from=2-1, to=2-2]
      \arrow["i", from=1-2, to=2-2]
    \end{tikzcd}\]
  is a cartesian square. This says also that $i^* \Sigma_i \id_x = i^* i \cong \id_x$ (with this isomorphism being unique since $\id_x$ is terminal in $\C/x$). Now we view $\Sigma_i \beta$ as an object of $(\C/(x \amalg y))/i$ and pull back along $i$ to obtain
  \[\begin{tikzcd}[ampersand replacement=\&]
      \bullet \& \bullet \\
      x \& x \\
      x \& {x \amalg y}
      \arrow["{\id_x}", from=2-1, to=2-2]
      \arrow["{\id_x}"', from=2-1, to=3-1]
      \arrow["i"', from=3-1, to=3-2]
      \arrow["i", from=2-2, to=3-2]
      \arrow[from=1-1, to=1-2]
      \arrow["\beta", from=1-2, to=2-2]
      \arrow["{\Sigma_i \beta}", curve={height=-24pt}, from=1-2, to=3-2]
      \arrow["{i^* \Sigma_i \beta}"', curve={height=24pt}, from=1-1, to=3-1]
      \arrow[from=1-1, to=2-1]
    \end{tikzcd}\]
  where the composite square is cartesian, thus exhibiting $i^* \Sigma_i \beta$ as an object of $(\C/x)/\id_x$. Since the lower square and composite square are cartesian, the upper square is also cartesian. This forces the diagram to be of the form
  \[\begin{tikzcd}[ampersand replacement=\&]
      \bullet \& \bullet \\
      x \& x \\
      x \& {x \amalg y}
      \arrow["{\id_x}", from=2-1, to=2-2]
      \arrow["{\id_x}"', from=2-1, to=3-1]
      \arrow["i"', from=3-1, to=3-2]
      \arrow["i", from=2-2, to=3-2]
      \arrow["{\id_x}", from=1-1, to=1-2]
      \arrow["\beta", from=1-2, to=2-2]
      \arrow["{\Sigma_i \beta}", curve={height=-24pt}, from=1-2, to=3-2]
      \arrow["{i^* \Sigma_i \beta}"', curve={height=24pt}, from=1-1, to=3-1]
      \arrow["\beta"', from=1-1, to=2-1]
    \end{tikzcd}\]
  Commutativity of this diagram now implies that $i^* \Sigma_i \beta \cong \id_x \circ \beta = \beta$.

  We conclude that $i^*(\Sigma_i \beta \amalg \Sigma_j \gamma) \cong \beta$ naturally in $\beta$. Symmetrically, $j^*(\Sigma_i \beta \amalg \Sigma_j \gamma) \cong \gamma$ naturally in $\gamma$. Thus, the composite
  \[\C/x \times \C/y \to \C/(x \amalg y) \to \C/x \times \C/y\]
  is naturally isomorphic to the identity.
\end{proof}

\subsection[The Proof of \Cref*{lem:restriction-commutes-with-sum}]{The Proof of \Cref{lem:restriction-commutes-with-sum}}

\begin{lem*}
  Let $\C$ be an LCCDC category and let $f : x \to y$ and $g : x' \to y'$ be morphisms in $\C$. Let $i_x : x \to x \amalg x'$ and $i_y : y \to y \amalg y'$ be the canonical inclusions. Then
  \[\begin{tikzcd}[ampersand replacement=\&]
      x \& y \\
      {x \amalg x'} \& {y \amalg y'}
      \arrow["f", from=1-1, to=1-2]
      \arrow["{i_x}"', from=1-1, to=2-1]
      \arrow["{f \amalg g}"', from=2-1, to=2-2]
      \arrow["{i_y}", from=1-2, to=2-2]
    \end{tikzcd}\]
  is a cartesian square.
\end{lem*}
\begin{proof}
  First, we note that the square commutes, simply by definition of coproduct. Since $\C$ is locally cartesian, we know that there exists some cartesian square
  \[\begin{tikzcd}[ampersand replacement=\&]
      \bullet \& y \\
      {x \amalg x'} \& {y \amalg y'}
      \arrow["{i_y}", from=1-2, to=2-2]
      \arrow[from=1-1, to=1-2]
      \arrow["t"', from=1-1, to=2-1]
      \arrow["\lrcorner"{anchor=center, pos=0.125}, draw=none, from=1-1, to=2-2]
      \arrow["{f \amalg g}"', from=2-1, to=2-2]
    \end{tikzcd}\]
  and we wish to show that this square is isomorphic to the original square in the statement of the lemma. To do so, we claim that it will suffice to show that $t$ is isomorphic (as an object of $\C/(x \amalg x')$) to $i_x$ -- then we will have a commutative diagram
  \[\begin{tikzcd}[ampersand replacement=\&]
      x \& \bullet \& y \\
      {x \amalg x'} \& {x \amalg x'} \& {y \amalg y'}
      \arrow["{f \amalg g}"', from=2-2, to=2-3]
      \arrow["{i_y}", from=1-3, to=2-3]
      \arrow[from=1-2, to=1-3]
      \arrow["t"', from=1-2, to=2-2]
      \arrow["\lrcorner"{anchor=center, pos=0.125}, draw=none, from=1-2, to=2-3]
      \arrow["\cong", from=1-1, to=1-2]
      \arrow["{i_x}"', from=1-1, to=2-1]
      \arrow["\id"', from=2-1, to=2-2]
    \end{tikzcd}\]
  The left-hand square is cartesian because it commutes and its horizontal arrows are isomorphisms, whence the composite square is also cartesian. This gives a cartesian square
  \[\begin{tikzcd}[ampersand replacement=\&]
      \bullet \& y \\
      {x \amalg x'} \& {y \amalg y'}
      \arrow["{f \amalg f'}"', from=2-1, to=2-2]
      \arrow["{i_y}", from=1-2, to=2-2]
      \arrow[from=1-1, to=1-2]
      \arrow["{i_x}"', from=1-1, to=2-1]
      \arrow["\lrcorner"{anchor=center, pos=0.125}, draw=none, from=1-1, to=2-2]
    \end{tikzcd}\]
  We then note that replacing the top morphism with $f$ would make the square commute, and $i_y$ is a monomorphism (by the LCCDC hypothesis on $\C$). Thus, the top morphism equals $f$, and we have the desired cartesian square.

  So, we have left to show that $t$ is isomorphic to $i_x$ in $\C/(x \amalg x')$, and for this purpose we make use of \Cref{prop:disjoint-coproducts-slice-products} -- letting $i_{x'} : x \to x \amalg x'$ denote the canonical inclusion, it will suffice to show that $i_x^* t \cong i_x^* i_x$ and $i_{x'}^* t \cong i_{x'}^* i_x$. We will tackle these two isomorphisms in order. First, we form a further pullback
  \begin{equation}\label{eq:restriction-commutes-with-sum:1}\begin{tikzcd}[ampersand replacement=\&]
      \bullet \& \bullet \& y \\
      x \& {x \amalg x'} \& {y \amalg y'}
      \arrow["{f \amalg g}"', from=2-2, to=2-3]
      \arrow["{i_y}", from=1-3, to=2-3]
      \arrow[from=1-2, to=1-3]
      \arrow["t"', from=1-2, to=2-2]
      \arrow["\lrcorner"{anchor=center, pos=0.125}, draw=none, from=1-2, to=2-3]
      \arrow["{i_x}"', from=2-1, to=2-2]
      \arrow["{i_x^*t}"', from=1-1, to=2-1]
      \arrow[from=1-1, to=1-2]
      \arrow["\lrcorner"{anchor=center, pos=0.125}, draw=none, from=1-1, to=2-2]
    \end{tikzcd}\end{equation}
  Now the bottom row composes to give $(f \amalg g) \circ i_x = i_y \circ f$, and so the composite square can also be formed by pulling back $i_y$ first along $i_y$ and then along $f$. Since $\C$ is LCCDC, $i_y$ is monic, i.e.
  \[\begin{tikzcd}[ampersand replacement=\&]
      y \& y \\
      y \& {y \amalg y'}
      \arrow["{i_y}", from=1-2, to=2-2]
      \arrow["\id", from=1-1, to=1-2]
      \arrow["{i_y}"', from=2-1, to=2-2]
      \arrow["\id"', from=1-1, to=2-1]
      \arrow["\lrcorner"{anchor=center, pos=0.125}, draw=none, from=1-1, to=2-2]
    \end{tikzcd}\]
  is cartesian. So the composite square of \eqref{eq:restriction-commutes-with-sum:1} is also the composite square in
  \[\begin{tikzcd}[ampersand replacement=\&]
      x \& y \& y \\
      x \& y \& {y \amalg y'}
      \arrow["{i_y}", from=1-3, to=2-3]
      \arrow["\id", from=1-2, to=1-3]
      \arrow["{i_y}"', from=2-2, to=2-3]
      \arrow["\id"', from=1-2, to=2-2]
      \arrow["\lrcorner"{anchor=center, pos=0.125}, draw=none, from=1-2, to=2-3]
      \arrow["f", from=1-1, to=1-2]
      \arrow["f"', from=2-1, to=2-2]
      \arrow["\id"', from=1-1, to=2-1]
      \arrow["\lrcorner"{anchor=center, pos=0.125}, draw=none, from=1-1, to=2-2]
    \end{tikzcd}\]
  In particular, we have $i_x^* t \cong \id_x$, and since $i_x$ is monic we also have $i_x^* i_x \cong \id_x$. Thus $i_x^* t \cong i_x^* i_x$.

  We only have left to show that $i_{x'}^* t \cong i_{x'}^* i_x$, and so we consider the pullback
  \[\begin{tikzcd}[ampersand replacement=\&]
      \bullet \& \bullet \& y \\
      {x'} \& {x \amalg x'} \& {y \amalg y'}
      \arrow["{f \amalg g}"', from=2-2, to=2-3]
      \arrow["{i_y}", from=1-3, to=2-3]
      \arrow[from=1-2, to=1-3]
      \arrow["t"', from=1-2, to=2-2]
      \arrow["\lrcorner"{anchor=center, pos=0.125}, draw=none, from=1-2, to=2-3]
      \arrow["{i_{x'}}"', from=2-1, to=2-2]
      \arrow["{i_{x'}^*t}"', from=1-1, to=2-1]
      \arrow[from=1-1, to=1-2]
      \arrow["\lrcorner"{anchor=center, pos=0.125}, draw=none, from=1-1, to=2-2]
    \end{tikzcd}\]
  The bottom row composes to give $(f \amalg g) \circ i_{x'} = i_{y'} \circ g$, where $i_{y'} : y' \to y \amalg y'$ is the canonical inclusion. So, the composite square can also be obtained by pulling back $i_y$ first along $i_{y'}$ and then along $g$. Since $\C$ is LCCDC, this first step of pulling back $i_y$ along $i_{y'}$ gives us $\varnothing$. Then pulling back further still yields $\varnothing$ (for example, since $g^*$ is a left adjoint and therefore preserves initial objects). Thus, $i_{x'}^* t \cong \varnothing$. Since $\C$ is LCCDC, we also have $i_{x'}^* i_x \cong \varnothing$, and so indeed $i_{x'}^* t \cong i_{x'}^* i_x$.
\end{proof}

\subsection[The Proof of \Cref*{prop:smack-factors}]{The Proof of \Cref{prop:smack-factors}}

\begin{prop*}
  If $F : \A{\C} \to \Set$ is a semi-Mackey functor, then $F$ factors uniquely through the forgetful functor $\CMon \to \Set$. This unique factorization is given by endowing each output set $F(x)$ with the binary operation $+_{F,x}$.
\end{prop*}
\begin{proof}
  The existence part of this claim amounts to checking that, for each morphism $\varphi : x \to y$ in $\A{\C}$, $F(\varphi) : F(x) \to F(y)$ is a monoid homomorphism with respect to $+_{F,x}$ and $+_{F,y}$. First, let $! : \varnothing \to x$ be the unique morphism in $\C$ from a chosen initial object to $x$. Then $\varphi \circ T_!$ is a morphism $\varnothing \to y$ in $\A{\C}$, but $\varnothing$ is initial in $\A{\C}$ (it is terminal by \Cref{product_in_A} and $\A{\C}$ is self-dual), and so $\varphi \circ T_! = T_{!!}$, where $!! : \varnothing \to y$ is the unique morphism in $\C$. Thus $F(T_{!!})$ (which picks out the identity element of $(F(y),+)$) is equal to $F(\varphi) \circ F(T_!)$. In other words, $F(\varphi)$ sends the identity element of $(F(x),+)$ to the identity element of $(F(y),+)$. Next, we must check that $F(\varphi)$ commutes with the binary operation $+$. We will show this separately for morphisms of type $T$ and $R$.

  First, suppose $\varphi = T_f$ for some $f : x \to y$ in $\C$. Then consider the diagram
  \begin{equation*}\begin{tikzcd}[ampersand replacement=\&]
      {F(x) \times F(x)} \& {F(x \amalg x)} \& {F(x)} \\
      {F(y) \times F(y)} \& {F(y \amalg y)} \& {F(y)}
      \arrow["\cong", from=1-1, to=1-2]
      \arrow["{F(T_\nabla)}", from=1-2, to=1-3]
      \arrow["{F(T_f)}"', from=1-3, to=2-3]
      \arrow["{F(T_{f \amalg f})}"', from=1-2, to=2-2]
      \arrow["{F(T_\nabla)}"', from=2-2, to=2-3]
      \arrow["{F(T_f) \times F(T_f)}"', from=1-1, to=2-1]
      \arrow["\cong", from=2-1, to=2-2]
    \end{tikzcd}\end{equation*}
  The right-hand square commutes because $f \circ \nabla_x = \nabla_y \circ (f \amalg f)$ in $\C$. To check that the left-hand square commutes, we invert the isomorphisms on top and bottom to obtain
  \[\begin{tikzcd}[ampersand replacement=\&]
      {F(x) \times F(x)} \&\& {F(x \amalg x)} \\
      {F(y) \times F(y)} \&\& {F(y \amalg y)}
      \arrow["{F(T_f) \times F(T_f)}"', from=1-1, to=2-1]
      \arrow["{(F(R_{i_1}), F(R_{i_2}))}"', from=1-3, to=1-1]
      \arrow["{(F(R_{i_1}), F(R_{i_2}))}", from=2-3, to=2-1]
      \arrow["{F(T_{f \amalg f})}"', from=1-3, to=2-3]
    \end{tikzcd}\]
  The two composites we must show are equal are morphisms $F(x \amalg x) \to F(y) \times F(y)$, so we check the two components separately. For the first component, we must compare $F(T_f) \circ F(R_{i_1})$ with $F(R_{i_1}) \circ F(T_{f \amalg f})$. So, it suffices to prove $T_f \circ R_{i_1} = R_{i_1} \circ T_{f \amalg f}$. For this, it suffices to prove that
  \[\begin{tikzcd}[ampersand replacement=\&]
      x \& y \\
      {x \amalg x} \& {y \amalg y}
      \arrow["f", from=1-1, to=1-2]
      \arrow["{i_1}"', from=1-1, to=2-1]
      \arrow["{f \amalg f}"', from=2-1, to=2-2]
      \arrow["{i_1}", from=1-2, to=2-2]
    \end{tikzcd}\]
  is cartesian in $\C$, but this is just a special case of \Cref{lem:restriction-commutes-with-sum}. Thus, the composite square commutes, which says exactly that $F(T_f)$ commutes with $+$.

  Now suppose $\varphi = R_g$ for some $g : y \to x$ in $\C$. Then
  \[\begin{tikzcd}[ampersand replacement=\&]
      {F(x) \times F(x)} \&\& {F(x \amalg x)} \\
      {F(y) \times F(y)} \&\& {F(y \amalg y)}
      \arrow["{F(R_g) \times F(R_g)}"', from=1-1, to=2-1]
      \arrow["{(F(R_{i_1}), F(R_{i_2}))}"', from=1-3, to=1-1]
      \arrow["{(F(R_{i_1}), F(R_{i_2}))}", from=2-3, to=2-1]
      \arrow["{F(R_{g \amalg g})}"', from=1-3, to=2-3]
    \end{tikzcd}\]
  commutes, because $(g \amalg g) \circ i_1 = i_1 \circ g$ and $(g \amalg g) \circ i_2 = i_2 \circ g$ in $\C$. Inverting the horizontal arrows, we get that
  \[\begin{tikzcd}[ampersand replacement=\&]
      {F(x) \times F(x)} \&\& {F(x \amalg x)} \\
      {F(y) \times F(y)} \&\& {F(y \amalg y)}
      \arrow["{F(R_g) \times F(R_g)}"', from=1-1, to=2-1]
      \arrow["\cong", from=1-1, to=1-3]
      \arrow["\cong"', from=2-1, to=2-3]
      \arrow["{F(R_{g \amalg g})}"', from=1-3, to=2-3]
    \end{tikzcd}\]
  commutes. We only have left to show that
  \[\begin{tikzcd}[ampersand replacement=\&]
      {F(x \amalg x)} \& {F(x)} \\
      {F(y \amalg y)} \& {F(y)}
      \arrow["{F(T_\nabla)}", from=1-1, to=1-2]
      \arrow["{F(R_g)}", from=1-2, to=2-2]
      \arrow["{F(R_{g \amalg g})}"', from=1-1, to=2-1]
      \arrow["{F(T_\nabla)}"', from=2-1, to=2-2]
    \end{tikzcd}\]
  commutes, for which it suffices to show that $R_g \circ T_\nabla = T_\nabla \circ R_{g \amalg g}$. For this claim, it suffices to show that
  \[\begin{tikzcd}[ampersand replacement=\&]
      {y \amalg y} \& {x \amalg x} \\
      y \& x
      \arrow["{g \amalg g}", from=1-1, to=1-2]
      \arrow["\nabla", from=1-2, to=2-2]
      \arrow["\nabla"', from=1-1, to=2-1]
      \arrow["g"', from=2-1, to=2-2]
    \end{tikzcd}\]
  is cartesian. This is \Cref{cor:fold-square-is-cartesian}.

  We have now established that $F : \A{\C} \to \Set$ factors through $\CMon$. For uniqueness, suppose we have some other factorization, i.e. on each $F(x)$ we have a commutative monoid operation $\cdot$ such that $F(\varphi)$ commutes with $\cdot$ for all morphisms $\varphi \in \A{\C}$. Then $+ = F(T_\nabla) \circ (F(R_{i_1}), F(R_{i_2}))^{-1} : F(x) \times F(x) \to F(x)$ is a monoid homomorphism with respect to $\cdot$. The Eckmann-Hilton argument now shows that $\cdot = +$.
\end{proof}

\subsection{Hoyer's Lemma 2.3.5}

Now we build up to an important technical lemma, generalizing \cite[Lemma 2.3.5]{Hoyer}, which is key to the proof of \Cref{actual-main-theorem}.

\begin{lem}
  \label{lem:slogan}
  Let $\C$ be a category, and let $i : x \to y$ be a morphism in $\C$. For any object $\alpha \in \C/x$, the functor $\Sigma_i / \alpha : (\C/x)/\alpha \to (\C/y)/\Sigma_i \alpha$ is an isomorphism.
\end{lem}
\begin{proof}
  This is essentially another rephrasing of \Cref{slogan:slice-of-slice-is-slice}. Both $(\C/x)/\alpha$ and $(\C/y)/\Sigma_i \alpha$ are canonically isomorphic to $\C/\operatorname{dom}(\alpha)$, and via these isomorphisms $\Sigma_i / \alpha$ factors as the identity.
\end{proof}

\begin{lem}
  \label{lem:dep-sum-pres-refl-pbs}
  Let $\C$ be a locally cartesian category and let $i : x \to y$ be a morphism in $\C$. Then $\Sigma_i : \C/x \to \C/y$ preserves and reflects pullbacks.
\end{lem}
\begin{proof}
  Consider an arbitrary commutative square (A) in $\C/x$
  \[\begin{tikzcd}
      \alpha & \beta \\
      \gamma & \delta
      \arrow["e", from=1-1, to=1-2]
      \arrow[""{name=0, anchor=center, inner sep=0}, "f"', from=1-1, to=2-1]
      \arrow[""{name=1, anchor=center, inner sep=0}, "g", from=1-2, to=2-2]
      \arrow["h"', from=2-1, to=2-2]
      \arrow["{(A)}"{description}, draw=none, from=0, to=1]
    \end{tikzcd}\]
  and let (B) be its image under $\Sigma_i$:
  \[\begin{tikzcd}
      {\Sigma_i \alpha} & {\Sigma_i \beta} \\
      {\Sigma_i \gamma} & {\Sigma_i \delta}
      \arrow["{\Sigma_i e}", from=1-1, to=1-2]
      \arrow[""{name=0, anchor=center, inner sep=0}, "{\Sigma_i f}"', from=1-1, to=2-1]
      \arrow["{\Sigma_i h}"', from=2-1, to=2-2]
      \arrow[""{name=1, anchor=center, inner sep=0}, "{\Sigma_i g}", from=1-2, to=2-2]
      \arrow["{(B)}"{description}, draw=none, from=0, to=1]
    \end{tikzcd}\]
  Note that (A) is cartesian iff $\gamma \xleftarrow{f} \alpha \xrightarrow{e} \beta$ is a product diagram in $(\C/x)/\delta$. By \Cref{lem:slogan}, $\Sigma_i/\delta$ is an isomorphism, so this happens iff $\Sigma_i \gamma \xleftarrow{\Sigma_i f} \Sigma_i \alpha \xrightarrow{\Sigma_i e} \Sigma_i \beta$ is a product diagram in $(\C/y)/\Sigma_i \delta$. This happens iff (B) is cartesian.
\end{proof}

\begin{prop}
  \label{prop:induction-is-cartesian}
  Let $\C$ be a locally cartesian category and let $i : x \to y$ be a morphism in $\C$. Then:
  \begin{enumerate}
    \item $\Sigma_i$ and $i^*$ preserve cartesian squares;
    \item Each naturality square for the unit and counit of the adjunction is cartesian.
  \end{enumerate}
\end{prop}
\begin{proof}
  By \Cref{lem:dep-sum-pres-refl-pbs}, $\Sigma_i$ preserves pullbacks, and $i^*$ preserves pullbacks because it is a right adjoint. Next we must check that the naturality squares for the unit and counit are cartesian.

  We begin with the counit. Let $f : \alpha \to \beta$ be an arbitrary morphism in $\C/y$, where $\alpha : a \to y$ and $\beta : b \to y$ are any objects. Then pull back along $i$ to get the commuting triangular prism below.
  
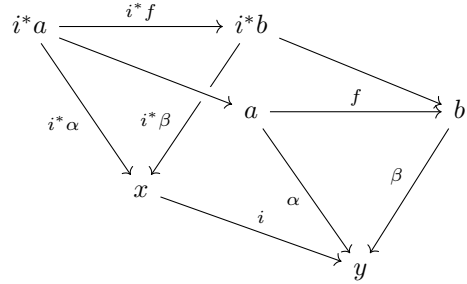
\[\begin{tikzcd}
      {i^* a} && {i^* b} \\
      && a && b \\
      & x \\
      &&& y
      \arrow["i", from=3-2, to=4-4]
      \arrow["\alpha"', from=2-3, to=4-4]
      \arrow["\beta"', from=2-5, to=4-4]
      \arrow["{i^* \alpha}"', from=1-1, to=3-2]
      \arrow["{i^* \beta}"'{pos=0.7}, from=1-3, to=3-2]
      \arrow[crossing over, from=1-1, to=2-3]
      \arrow[from=1-3, to=2-5]
      \arrow["f", from=2-3, to=2-5]
      \arrow["{i^*f}", from=1-1, to=1-3]
    \end{tikzcd}\]
  The ``left and right faces''
  \(\begin{tikzcd}
    {i^*a} & a \\
    x & y
    \arrow[from=1-1, to=1-2]
    \arrow["{i^* \alpha}"', from=1-1, to=2-1]
    \arrow["i"', from=2-1, to=2-2]
    \arrow["\alpha", from=1-2, to=2-2]
  \end{tikzcd}\)
  and
  \(\begin{tikzcd}
    {i^*b} & b \\
    x & y
    \arrow[from=1-1, to=1-2]
    \arrow["\beta", from=1-2, to=2-2]
    \arrow["i"', from=2-1, to=2-2]
    \arrow["{i^*\beta}"', from=1-1, to=2-1]
  \end{tikzcd}\)
  are cartesian by construction. Noe the ``top face''
  \(\begin{tikzcd}
    {i^*a} & {i^*b} \\
    a & b
    \arrow["{i^*f}", from=1-1, to=1-2]
    \arrow[from=1-1, to=2-1]
    \arrow[from=1-2, to=2-2]
    \arrow["f", from=2-1, to=2-2]
  \end{tikzcd}\)
  is then cartesian, because pasting with the right face yields the left face. The naturality square for the counit of the adjunction and the morphism $f$ is then cartesian, because it is essentially the aforementioned top face:
  
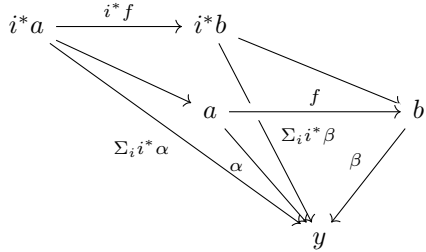
\[\begin{tikzcd}
      {i^* a} && {i^* b} \\
      && a && b \\
      \\
      &&& y
      \arrow["\alpha"'{pos=0.3}, from=2-3, to=4-4]
      \arrow["\beta"', from=2-5, to=4-4]
      \arrow[from=1-1, to=2-3]
      \arrow[from=1-3, to=2-5]
      \arrow["{i^*f}", from=1-1, to=1-3]
      \arrow["{\Sigma_i i^* \alpha}"', from=1-1, to=4-4]
      \arrow["{\Sigma_i i^* \beta}"{pos=0.6}, from=1-3, to=4-4]
      \arrow["f", crossing over, from=2-3, to=2-5]
    \end{tikzcd}\]
  Next, we will show that the naturality squares for the unit of the adjunction are cartesian.

  Let $f : \alpha \to \beta$ now be an arbitrary morphism in $\C/x$. The naturality square in question is
  \[\begin{tikzcd}[ampersand replacement=\&]
      \alpha \& {i^* \Sigma_i \alpha} \\
      \beta \& {i^*\Sigma_i\beta}
      \arrow["{\eta_\alpha}", from=1-1, to=1-2]
      \arrow[""{name=0, anchor=center, inner sep=0}, "f"', from=1-1, to=2-1]
      \arrow["{\eta_\beta}"', from=2-1, to=2-2]
      \arrow[""{name=1, anchor=center, inner sep=0}, "{i^*\Sigma_i f}", from=1-2, to=2-2]
      \arrow["{(A)}"{description}, draw=none, from=0, to=1]
    \end{tikzcd}\]
  The image of $(A)$ under $\Sigma_i$ is
  \[\begin{tikzcd}[ampersand replacement=\&]
      {\Sigma_i\alpha} \& {\Sigma_i i^* \Sigma_i \alpha} \\
      {\Sigma_i \beta} \& {\Sigma_i i^*\Sigma_i\beta}
      \arrow["{\Sigma_i \eta_\alpha}", from=1-1, to=1-2]
      \arrow[""{name=0, anchor=center, inner sep=0}, "{\Sigma_i f}"', from=1-1, to=2-1]
      \arrow["{\Sigma_i \eta_\beta}"', from=2-1, to=2-2]
      \arrow[""{name=1, anchor=center, inner sep=0}, "{\Sigma_i i^*\Sigma_i f}", from=1-2, to=2-2]
      \arrow["{(A')}"{description}, draw=none, from=0, to=1]
    \end{tikzcd}\]
  By the triangle identities, $(A')$ fits in the commutative diagram
  \[\begin{tikzcd}[ampersand replacement=\&]
      {\Sigma_i\alpha} \& {\Sigma_i i^* \Sigma_i \alpha} \& {\Sigma_i \alpha} \\
      {\Sigma_i \beta} \& {\Sigma_i i^*\Sigma_i\beta} \& {\Sigma_i \beta}
      \arrow["{\Sigma_i \eta_\alpha}", from=1-1, to=1-2]
      \arrow[""{name=0, anchor=center, inner sep=0}, "{\Sigma_i f}"', from=1-1, to=2-1]
      \arrow["{\Sigma_i \eta_\beta}"', from=2-1, to=2-2]
      \arrow[""{name=1, anchor=center, inner sep=0}, "{\Sigma_i i^*\Sigma_i f}", from=1-2, to=2-2]
      \arrow["{\varepsilon_{\Sigma_i \alpha}}", from=1-2, to=1-3]
      \arrow["{\Sigma_i f}", from=1-3, to=2-3]
      \arrow["{\varepsilon_{\Sigma_i \beta}}"', from=2-2, to=2-3]
      \arrow["{\operatorname{id}}", curve={height=-24pt}, from=1-1, to=1-3]
      \arrow["{\operatorname{id}}"', curve={height=24pt}, from=2-1, to=2-3]
      \arrow["{(A')}"{description}, draw=none, from=0, to=1]
    \end{tikzcd}\]
  The right-hand square is the naturality square of the counit $\eta$, which we showed above is cartesian. The composite square is the ``identity square'' of $\Sigma_i f$, which is also cartesian. We conclude that $(A')$ is cartesian. By \Cref{lem:dep-sum-pres-refl-pbs}, $\Sigma_i$ reflects pullbacks, so $(A)$ is cartesian, as desired.
\end{proof}

\begin{cor}
  \label{cor:adjunct-is-cartesian}
  A commutative square (A) is cartesian if and only if its adjunct (B) is.
  \[\begin{tikzcd}[ampersand replacement=\&]
      {\Sigma_ia} \& c \& a \& {i^* c} \\
      {\Sigma_i b} \& d \& b \& {i^* d}
      \arrow[from=1-1, to=1-2]
      \arrow[""{name=0, anchor=center, inner sep=0}, "{\Sigma_i p}"', from=1-1, to=2-1]
      \arrow[from=2-1, to=2-2]
      \arrow[""{name=1, anchor=center, inner sep=0}, "q", from=1-2, to=2-2]
      \arrow[""{name=2, anchor=center, inner sep=0}, "p"', from=1-3, to=2-3]
      \arrow[from=1-3, to=1-4]
      \arrow[from=2-3, to=2-4]
      \arrow[""{name=3, anchor=center, inner sep=0}, "{i^* q}", from=1-4, to=2-4]
      \arrow["{(A)}"{description}, draw=none, from=0, to=1]
      \arrow["{(B)}"{description}, draw=none, from=2, to=3]
    \end{tikzcd}\]
\end{cor}

\begin{lem}
  \label{lem:slice-adjunction}
  Let $\D$ and $\D'$ be categories such that $\D$ admits all (binary) pullbacks. If $F : \D \to \D'$ is left adjoint to $G : \D' \to \D$ and $d \in \D$ is some object, then $F/d$ is left adjoint to $\eta_d^* \circ (G/Fd)$, where $\eta_d : d \to GFd$ is the unit of $F \dashv G$.
\end{lem}
\begin{proof}
  Consider arbitrary objects $\alpha : a \to d$ in $\D/d$ and $\beta : b \to Fd$ in $\D'/Fd$. Then
  \begin{multline*}
    \Hom_{\D/d}(\alpha, \eta_d^* (G/Fd) \beta)\\
    = \left\{ f \in \Hom_\D(a, Gb) \;\middle| \begin{tikzcd}[ampersand replacement=\&] a \rar["f"] \dar["\alpha"] \& Gb \dar["G \beta"] \\ d \rar["\eta_d"] \& GFd \end{tikzcd} \;\text{commutes} \right\} \\
    \cong \left\{ g \in \Hom_{\D'}(Fa,b) \;\middle| \begin{tikzcd}[ampersand replacement=\&] Fa \rar["g"] \dar["F \alpha"] \& b \dar["\beta"] \\ Fd \rar["\id_{Fd}"] \& Fd \end{tikzcd} \;\text{commutes} \right\} \\
    = \Hom_{\D'/Fd}((F/d)\alpha,\beta)
  \end{multline*}
  and this bijection is natural in $\alpha$ and $\beta$.
\end{proof}

\begin{lem}
  \label{lem:comp-of-slice-functors}
  Let $F : \D \to \D'$ and $G : \D' \to \D''$ be functors, and let $d \in D$ be an object. Then $(G \circ F)/d = (G/Fd) \circ (F/d)$.
\end{lem}
\begin{proof}
  Unravel the definitions.
\end{proof}

\begin{lem}
  \label{lem:pullback-along-eps-ind}
  For all objects $\beta \in \C/y$, we have $(\varepsilon_\beta^{\text{ind}})^* = (\Sigma_i \circ i^*)/\beta$.
\end{lem}
\begin{proof}
  Let $b$ be the domain of $\beta$, so that $i^* \beta$ fits in the pullback square
  \[\begin{tikzcd}
      {i^*b} & b \\
      x & y
      \arrow["i", from=2-1, to=2-2]
      \arrow["{i^*\beta}"', from=1-1, to=2-1]
      \arrow["\beta", from=1-2, to=2-2]
      \arrow[from=1-1, to=1-2]
      \arrow["\lrcorner"{anchor=center, pos=0.125}, draw=none, from=1-1, to=2-2]
    \end{tikzcd}\]
  By definition, $\Sigma_i i^* \beta = i \circ i^* \beta$, so we can notice that $\varepsilon_b^{\text{ind}} : \Sigma_i i^* \beta \to \beta$ is actually the top morphism in the above square. Now we let $\gamma \in (\C/y)/\beta$ be arbitrary and form the further pullback
  \[\begin{tikzcd}
      {(\varepsilon_b^{\text{ind}})^*c} & c \\
      {i^*b} & b \\
      x & y
      \arrow["i", from=3-1, to=3-2]
      \arrow["{i^*\beta}"', from=2-1, to=3-1]
      \arrow["\beta", from=2-2, to=3-2]
      \arrow["{\varepsilon_b^{\text{ind}}}", from=2-1, to=2-2]
      \arrow["\lrcorner"{anchor=center, pos=0.125}, draw=none, from=2-1, to=3-2]
      \arrow["\gamma", from=1-2, to=2-2]
      \arrow[from=1-1, to=1-2]
      \arrow["{(\varepsilon_b^{\text{ind}})^* \gamma}"', from=1-1, to=2-1]
      \arrow["\lrcorner"{anchor=center, pos=0.125}, draw=none, from=1-1, to=2-2]
    \end{tikzcd}\]
  Now since both squares are pullbacks, the composite rectangle is a pullback as well. This says precisely that $((\Sigma_i \circ i^*)/\beta)\gamma$ equals $(\varepsilon_b^{\text{ind}})^* \gamma$ as an object of $(\C/y)/\Sigma_i i^* \beta$.
\end{proof}

\begin{prop}[cf. \cite{Hoyer}, Lemma 2.3.5]
  \label{prop:Hoyer-Lemma}
  The functors $\Pi_{\varepsilon^\text{ind}_b} \circ (\Sigma_i / i^* b)$ and $(\eta^\text{coind}_b)^* \circ (\Pi_i / i^*b)$ are naturally isomorphic for all objects $b \in \C/y$.
\end{prop}

\begin{proof}
  We have isomorphisms
  \begin{align*}
     & \Hom_{(\C/y)/b}(c,(\eta_b^{\text{coind}})^* (\Pi_i/i^*b) a)                                &                                                                                   \\
     & \quad \cong \Hom_{(\C/x)/i^*b}((i^*/b)c,a)                                                 & \text{(\Cref{lem:slice-adjunction})}                                              \\
     & \quad \cong \Hom_{(\C/y)/\Sigma_i i^* b}((\Sigma_i / i^*b) (i^*/b) c, (\Sigma_i/i^* b) a)  & \text{(\Cref{lem:slogan})}                                                        \\
     & \quad = \Hom_{(\C/y)/\Sigma_i i^* b}(((\Sigma_i \circ i^*)/b) c, (\Sigma_i/i^* b) a)       & \text{(\Cref{lem:comp-of-slice-functors})}                                        \\
     & \quad = \Hom_{(\C/y)/\Sigma_i i^* b}((\varepsilon_b^{\text{ind}})^* c, (\Sigma_i/i^* b) a) & \text{(\Cref{lem:pullback-along-eps-ind})}                                        \\
     & \quad \cong \Hom_{(\C/y)/b}(c, \Pi_{\varepsilon_b^{\text{ind}}} (\Sigma_i/i^* b) a)        & \text{($(\varepsilon_b^{\text{ind}})^* \dashv \Pi_{\varepsilon_b^{\text{ind}}}$)}
  \end{align*}
  natural in $a \in (\C/x)/i^*b$ and $c \in (\C/y)/b$.
\end{proof}

\newpage
\printbibliography[heading=bibintoc]

\end{document}